\documentclass[11pt,a4paper]{article}

\usepackage{amsfonts}
\usepackage{amssymb}
\usepackage{amsmath}
\usepackage{amsthm}
\usepackage{mathrsfs}
\usepackage{caption}

\usepackage{booktabs}

\usepackage{varwidth}
\usepackage[usenames,dvipsnames,table]{xcolor}

\usepackage{tikz}
\usepackage{tikz-cd}
\usetikzlibrary{patterns}

\usepackage{tcolorbox}

\usepackage{fancyhdr}

\usepackage{setspace}
\usepackage{enumitem}
\usepackage{titlesec}
\usepackage{geometry}

\usepackage{mathpazo}
\usepackage[euler-digits,euler-hat-accent]{eulervm}

\usepackage{manfnt}

\setlist[enumerate,1]{label=(\roman*)}

\newlist{admissible}{enumerate}{10}
\setlist[admissible]{leftmargin=1.4cm,label=(A\arabic*),start=0}

\usepackage[colorlinks=true, linkcolor=blue, citecolor=black]{hyperref}

\usepackage[capitalize]{cleveref}

\crefname{admissiblei}{}{}
\Crefname{admissiblei}{}{}

\titleformat{\part}[block]{\Large\fillast}{\thepart.}{1em}{\scshape}[\vspace{0.25cm}]
\titleformat{\section}[block]{\large\fillast}{\thesection.}{1em}{\bfseries}[\vspace{1mm}]
\titleformat{\subsection}[block]{\normalsize}{\thesubsection}{0.5em}{\itshape}[]
\titlespacing{\subsection}{0mm}{1.5em}{0.5em}[0mm]

\newtheorem{theorem}{Theorem}[section]
\newtheorem{proposition}[theorem]{Proposition}
\newtheorem{lemma}[theorem]{Lemma}
\newtheorem{corollary}[theorem]{Corollary}
\newtheorem{mtheorem}{Theorem}

\theoremstyle{definition}
\newtheorem{definition}[theorem]{Definition}
\newtheorem{algorithm}[theorem]{Algorithm}

\theoremstyle{remark}
\newtheorem{rmk}[theorem]{Remark}
\newtheorem{example}[theorem]{Example}

\crefname{mtheorem}{Theorem}{Theorems}

\definecolor{asmpgray}{gray}{0.85}
\newsavebox{\asmpbox}
\newenvironment{assumption}{
\begin{lrbox}{\asmpbox}
\begin{varwidth}{0.75\textwidth}
\ignorespaces
}
{
\end{varwidth}
\end{lrbox}
\begin{center}
\setlength{\fboxsep}{5pt}
\fcolorbox{black}{asmpgray}{\usebox{\asmpbox}}
\end{center}
\ignorespacesafterend
}

\newenvironment{lem}{\begin{lemma}}{\end{lemma}}
\newenvironment{prop}{\begin{proposition}}{\end{proposition}}
\newenvironment{cor}{\begin{corollary}}{\end{corollary}}
\newenvironment{thm}{\begin{theorem}}{\end{theorem}}
\newenvironment{rem}{\begin{rmk}}{\end{rmk}}
\newenvironment{exmp}{\begin{example}}{\end{example}}

\newcommand\bfA{\mathbf{A}}
\newcommand\bfB{\mathbf{B}}
\newcommand\bfC{\mathbf{C}}
\newcommand\bfD{\mathbf{D}}
\newcommand\bfG{\mathbf{G}}
\newcommand\bfH{\mathbf{H}}
\newcommand\bfK{\mathbf{K}}
\newcommand\bfL{\mathbf{L}}
\newcommand\bfM{\mathbf{M}}
\newcommand\bfP{\mathbf{P}}
\newcommand\bfS{\mathbf{S}}
\newcommand\bfT{\mathbf{T}}
\newcommand\bfU{\mathbf{U}}
\newcommand\bfV{\mathbf{V}}
\newcommand\bfX{\mathbf{X}}
\newcommand\bfZ{\mathbf{Z}}

\newcommand{\Chevie}{{\sf CHEVIE}}
\newcommand{\simto}{\,\mathop{\rightarrow}\limits^\sim\,}
\newcommand{\reg}{\ensuremath{\mathrm{reg}}}
\newcommand{\uni}{\ensuremath{\mathrm{uni}}}
\newcommand{\spr}{\ensuremath{\mathrm{spr}}}
\newcommand{\simc}{\ensuremath{\mathrm{sc}}}
\newcommand{\der}{\ensuremath{\mathrm{der}}}
\newcommand{\sgn}{\ensuremath{\mathrm{sgn}}}
\newcommand{\Hei}{\ensuremath{\mathrm{H}}}
\newcommand{\SHei}{\ensuremath{\mathrm{SH}}}

\newcommand{\Ql}{\ensuremath{\overline{\mathbb{Q}}_{\ell}}}

\newcommand\bA\bfA
\newcommand\bB\bfB
\newcommand\bC\bfC
\newcommand\bD\bfD
\newcommand\bG\bfG
\newcommand\bH\bfH
\newcommand\bK\bfK
\newcommand\bL\bfL
\newcommand\bM\bfM
\newcommand\bP\bfP
\newcommand\bS\bfS
\newcommand\bT\bfT
\newcommand\bU\bfU
\newcommand\bV\bfV
\newcommand\bX\bfX
\newcommand\bZ\bfZ
\newcommand{\lie}[1]{\ensuremath{\mathfrak{#1}}}

\DeclareMathOperator{\Ad}{Ad}
\DeclareMathOperator{\Res}{Res}
\DeclareMathOperator{\Ind}{Ind}

\DeclareMathOperator{\Inf}{Inf}
\DeclareMathOperator{\Out}{Out}
\DeclareMathOperator{\ad}{ad}
\DeclareMathOperator{\Id}{Id}
\DeclareMathOperator{\Lie}{Lie}
\DeclareMathOperator{\Aut}{Aut}
\DeclareMathOperator{\Inn}{Inn}
\DeclareMathOperator{\GL}{GL}
\DeclareMathOperator{\SL}{SL}
\DeclareMathOperator{\Sp}{Sp}
\DeclareMathOperator{\SO}{SO}
\DeclareMathOperator{\Or}{O}
\DeclareMathOperator{\PGL}{PGL}
\DeclareMathOperator{\Cl}{Cl}
\DeclareMathOperator{\St}{St}
\DeclareMathOperator{\Ker}{Ker}

\DeclareMathOperator{\Hom}{Hom}

\DeclareMathOperator{\Tr}{Tr}

\DeclareMathOperator{\Irr}{Irr}
\DeclareMathOperator{\Uch}{UCh}
\DeclareMathOperator{\USh}{USh}
\DeclareMathOperator{\Fam}{Fam}
\DeclareMathOperator{\Class}{Class}
\DeclareMathOperator{\Span}{Span}
\DeclareMathOperator{\Char}{Char}

\crefname{equation}{}{}
\Crefname{equation}{}{}

\DeclareFontFamily{U}{mathx}{\hyphenchar\font45}
\DeclareFontShape{U}{mathx}{m}{n}{
      <5> <6> <7> <8> <9> <10>
      <10.95> <12> <14.4> <17.28> <20.74> <24.88>
      mathx10
      }{}
\DeclareSymbolFont{mathx}{U}{mathx}{m}{n}
\DeclareFontSubstitution{U}{mathx}{m}{n}

\DeclareFontFamily{OT1}{cmbr}{\hyphenchar\font45}
\DeclareFontShape{OT1}{cmbr}{m}{n}{%
<-9>cmbr8%
<9-10>cmbr9%
<10-17>cmbr10%
<17->cmbr17%
}{}
\DeclareSymbolFont{cmbr}{OT1}{cmbr}{m}{n}

\DeclareMathAccent{\widecheck}{0}{mathx}{"71}

\DeclareMathSymbol{\A}{7}{cmbr}{"41}
\DeclareMathSymbol{\B}{7}{cmbr}{"42}
\DeclareMathSymbol{\C}{7}{cmbr}{"43}
\DeclareMathSymbol{\D}{7}{cmbr}{"44}
\DeclareMathSymbol{\E}{7}{cmbr}{"45}
\DeclareMathSymbol{\F}{7}{cmbr}{"46}
\DeclareMathSymbol{\G}{7}{cmbr}{"47}

\numberwithin{equation}{section}

\allowdisplaybreaks

\title{}
\author{Olivier Brunat, Olivier Dudas, Jay Taylor}
\date{\today}

\begin{document}
\fancyhf{}
\fancyhead[R]{\thepage}
\thispagestyle{empty}
\setstretch{1.2}

{\rule{0.85\linewidth}{0.2mm}\\[10pt]
\centering
\begin{minipage}[c]{0.75\textwidth}
\centering\Large {\bf Unitriangular Shape of Decomposition Matrices of Unipotent Blocks}\\[6pt]
{\it Olivier Brunat, Olivier Dudas, and Jay Taylor}\\
\end{minipage}\\[10pt]
\rule{0.5\linewidth}{0.2mm}\\}

\begin{abstract}
We show that the decomposition matrix of unipotent $\ell$-blocks of a finite reductive group $\bG(\mathbb{F}_q)$ has a unitriangular shape, assuming $q$ is a power of a good prime and $\ell$ is very good for $\bG$. This was conjectured by Geck in 1990 as part of his PhD thesis. We establish this result by constructing projective modules using a modification of generalised Gelfand--Graev characters introduced by Kawanaka. We prove that each such character has at most one unipotent constituent which occurs with multiplicity one. This establishes a 30 year old conjecture of Kawanaka.
\end{abstract}

\tableofcontents

\section{Introduction}\label{sec:introduction}

\subsection{Overview}
One of the core features of the representation theory of finite groups is the ability to take an ordinary representation, defined over a field of characteristic zero, and reduce it to obtain a modular representation, defined over a field of positive characteristic $\ell>0$. The $\ell$-decomposition matrix of a finite group is a matrix whose rows are indexed by ordinary irreducible representations and whose columns are indexed by modular irreducible representations. The entries of this matrix give the multiplicities of modular irreducibles in the reduction of ordinary irreducibles. 

Since ordinary representations are in general better understood than modular representations, it is crucial to have a good understanding of the decomposition matrix. The focus of this paper is to obtain new information on decomposition matrices when $G = \bG^F$ is a finite reductive group and $\ell \nmid q$ (the non-defining characteristic setting). By definition $G$ is the subgroup of $F$-fixed points of a connected reductive algebraic group $\bG$, defined over an algebraic closure $K = \overline{\mathbb{F}_q}$, where $F : \bG \to \bG$ is a Frobenius endomorphism endowing $\bG$ with an $\mathbb{F}_q$-structure.

Decomposition matrices can be computed blockwise. In addition, for finite reductive groups, it is conjectured that any $\ell$-block is Morita equivalent to a so-called unipotent $\ell$-block of a possibly disconnected group \cite{Brou}, which is a block containing ordinary unipotent characters, and that in this case the decomposition matrices coincide. Such a reduction, known as the Jordan decomposition of blocks, was proven to hold in many cases by work of Bonnaf\'e--Rouquier \cite{BoRo03} and more recently Bonnaf\'e--Dat--Rouquier \cite{BDR17}. Therefore from now on we shall focus only on the unipotent $\ell$-blocks of $G$.

Even for groups of small rank, determining the decomposition matrix of unipotent $\ell$-blocks is considered a very hard problem. For example, the $3\times3$-decomposition matrix for the 3-dimensional
unitary groups $\mathrm{SU}_3(q)$ when $\ell \mid q+1$ was determined only in 2002 by Okuyama--Waki \cite{OW02}. Since then, new decomposition matrices were obtained using the $\ell$-adic cohomology of Deligne--Lusztig varieties for larger unitary groups or other small rank groups. However, for this method to be successful, one needs prior knowledge on the general shape of the decomposition matrix. 

In 1990 Geck observed in his thesis \cite{GeThesis} two general properties of the decomposition matrices of unipotent blocks which he conjectured to hold for any finite reductive group, assuming that $\ell$ is not too small, namely:
 \begin{itemize}
   \item the $\ell$-reduction of a \emph{cuspidal} unipotent character remains irreducible,
   \item the rows and columns of the decomposition matrix can be arranged so that it has a lower unitriangular shape, with identity blocks on families.
 \end{itemize}
The first conjecture was recently solved by Malle and the second author \cite{DuMa18}. Our main result is a proof of the second conjecture.

\subsection{Unitriangular shape}
The Weyl group $W$ of $\bG$ is its combinatorial skeleton. In \cite[\S4]{Lu84} Lusztig has defined a partition $\Fam(W)$ of the irreducible characters of the Weyl group $W$ into families
\begin{equation*}
\Irr(W) = \bigsqcup_{\mathscr{F} \in \Fam(W)} \mathscr{F}.
\end{equation*}
We have a natural action of $F$ on $W$ which permutes the set $\Fam(W)$ and we denote by $\Fam(W)^F$ those families that are fixed by this action. For each $F$-stable family $\mathscr{F} \in \Fam(W)^F$ there is a corresponding family of ordinary unipotent characters of $G$, which we denote by $\Uch(\mathscr{F})$. This gives a partition
\begin{equation*}
\Uch(G) = \bigsqcup_{\mathscr{F} \in \Fam(W)^F} \Uch(\mathscr{F})
\end{equation*}
of the set of all ordinary unipotent characters of $G$.

To each family $\mathscr{F} \in \Fam(W)$ there is a corresponding unipotent conjugacy class $\mathcal{O}_{\mathscr{F}} \subseteq \bG$, related via the Springer correspondence. If $\mathscr{F}$ is $F$-stable then $\mathcal{O}_{\mathscr{F}}$ is characterised by the fact that $\Uch(\mathscr{F})$ is exactly the set of unipotent characters with \emph{unipotent support} $\mathcal{O}_{\mathscr{F}}$, see \cite[Prop.~4.2]{GeMa00}. If $\mathcal{O}_1,\mathcal{O}_2 \subseteq \bG$ are two unipotent classes then we set $\mathcal{O}_1 \preceq \mathcal{O}_2$ if $\mathcal{O}_1 \subseteq \overline{\mathcal{O}_2}$ (the Zariski closure). This defines a natural partial ordering of the classes. Roughly, the unipotent support of a character is then defined to be the largest $F$-stable unipotent conjugacy class, with respect to $\preceq$, on which the character takes a non-zero value, see \cite[Thm.~1.26]{Tay19}.

If $\mathscr{F} = \{1_W\}$ is the family containing the trivial character of $W$ then $\mathcal{O}_{\mathscr{F}}$ is the regular unipotent class, which is the unique maximal element with respect to $\preceq$. We have $\Uch(\mathscr{F}) = \{1_G\}$ contains only the trivial character whose $\ell$-reduction is irreducible. On the other hand, if $\mathscr{F} = \{\sgn_W\}$ is the family containing the sign character of $W$ then $\mathcal{O}_{\mathscr{F}}$ is the trivial class, which is the unique minimal element with respect to $\preceq$. We have $\Uch(\mathscr{F}) = \{\St_G\}$ contains only the Steinberg character whose $\ell$-reduction can potentially have many irreducible constituents.

Throughout we will work under some mild restrictions on the characteristic of the fields involved. For clarity we recall that a prime $p$ is said to be \emph{bad}  for a quasi-simple group $\bG$ if $p=2$ and $\bG$ is not of type $\A_n$, $p=3$ and $\bG$ is of exceptional type, or $p=5$ and $\bG$ is of type $\E_8$. For any arbitrary group $\bG$ we say $p$  is \emph{good} if it is not bad for all the quasi-simple components of $\bG$. Another restriction will come from the order of the finite group $Z(\bfG)_F$ which is the largest (finite) quotient of $Z(\bfG)$ on which $F$ acts trivially. This order involves only bad primes unless $\bG$ has quasi-simple components of type $\A_n$.

The following, being our main result, gives the general shape of the decomposition matrix of unipotent blocks with respect to the ordering on unipotent classes. Our result solves a conjecture of Geck \cite[Conj.~2.1]{Ge12}, which strengthens earlier conjectures stated in \cite{GeThesis,GeHi94}.

\begin{mtheorem}\label{thm:mainA}
Recall that $\bG$ is a connected reductive algebraic group defined over $K$ with Frobenius endomorphism $F : \bG \to \bG$. We make the following assumptions on $\ell$ and $p = \Char(K)$:
\begin{itemize}
\item $\ell \neq p$ (non-defining characteristic),
\item $\ell$ and $p$ are good for $\bfG$,
\item $\ell$ does not divide the order of $Z(\bfG)_F$, the largest quotient of $Z(\bfG)$ on which $F$ acts trivially.
\end{itemize}
Let $\mathscr{F}_1 \leqslant \cdots \leqslant \mathscr{F}_r$ be any total ordering of the $F$-stable families $\Fam(W)^F$ such that if $\mathcal{O}_{\mathscr{F}_i} \preceq \mathcal{O}_{\mathscr{F}_j}$ then $\mathscr{F}_i \geqslant \mathscr{F}_j$. Then the modular irreducible unipotent representations of $G = \bG^F$ can be ordered such that the decomposition matrix of the unipotent $\ell$-blocks of $G$ has the following shape
\begin{equation*}
\left(\begin{array}{cccc}
D_{\mathscr{F}_1} & 0 & \cdots & 0 \\ 
\star & D_{\mathscr{F}_2}   & \ddots & \vdots \\
\vdots & \ddots & \ddots & 0\\
 \star & \cdots & \star & D_{\mathscr{F}_r}  \\
 \star & \cdots & \cdots & \star\\
\end{array}\right),
\end{equation*}
where each diagonal block $D_{\mathscr{F}_i}$ is the identity matrix, whose rows are indexed by the elements of $\Uch(\mathscr{F}_i)$.
\end{mtheorem}

This result had been previously obtained for groups of type $A$ by Dipper--James \cite{DipJa89} and Geck \cite{Ge91}, classical groups at a linear prime $\ell$ by Gruber--Hiss \cite{GrHi97}, and exceptional groups of type ${}^3D_4$ by Geck \cite{Ge91bis}, $G_2$ by Hiss \cite{Hiss89} and $F_4$ by Wings \cite{Wings}. The assumption on $\ell$ ensures that the unipotent characters form a basic set for the unipotent $\ell$-blocks \cite{Ge93}. As stated \cref{thm:mainA} need not hold when $\ell$ divides the order of $Z(\bG)_F$. Already this is the case when $G = \SL_2(q)$ and $\ell = 2$. In future work we intend to explain what must be done to extend the previous theorem to the case where $\ell$ is bad or divides $Z(\bfG)_F$. 

\subsection{Kawanaka characters}
The strategy for proving \cref{thm:mainA} goes back to Geck and Geck--Hiss \cite{Ge91,GeHi94}. It relies on constructing projective modules whose characters involve only one unipotent character of a given family, with multiplicity one, and other unipotent characters from larger families only. When $p$ is good, such projective modules were constructed by Kawanaka in \cite{Ka87} using a modified version of the generalised Gelfand-Graev representations. The difficult part is to show the statement about the unipotent constituents of their character, which was conjectured by Kawanaka in \cite[2.4.5]{Ka87}. In order to explain Kawanaka's construction and our second main result we need to recall a few further facts regarding Lusztig's classification of unipotent characters.

Recall that we have a natural involution $-\otimes\sgn_W : \Irr(W) \to \Irr(W)$ given by tensoring with the sign character. This involution permutes the families of $W$ and we denote by $\mathscr{F}^* = \mathscr{F}\otimes\sgn_W$ the \emph{dual family}. This duality on families corresponds to an order reversing duality on unipotent classes. Namely, we have $\mathcal{O}_{\mathscr{F}} \preceq \mathcal{O}_{\mathscr{F}'}$ if and only if $\mathcal{O}_{\mathscr{F}'^*} \preceq \mathcal{O}_{\mathscr{F}^*}$.

To simplify the notation we assume now that the $\mathbb{F}_q$-structure on $\bG$, determined by $F$, is split. To each family $\mathscr{F} \in \Fam(W)$ Lusztig has associated a small finite group $\bar A_{\mathscr{F}}$ known as the \emph{canonical quotient}. The elements in $\Uch(\mathscr{F}^*)$ are parametrised by the set $\mathscr{M}(\bar{A}_{\mathscr{F}})$ of $\bar A_{\mathscr{F}}$-conjugacy classes of pairs $(a,\psi)$ where $a \in \bar A_{\mathscr{F}}$ and $\psi \in \Irr(C_{\bar A_{\mathscr{F}}}(a))$. Under some extra assumption on $\mathscr{F}$ one can form the Kawanaka character $K_{[a,\psi]}$ of $G$ attached to the conjugacy class of $(a,\psi)$, see \cref{ssec:kawnakadef}. When $\ell$ is not too small (not dividing the order of the finite group $\bar A_{\mathscr{F}}$), then $K_{[a,\psi]}$ is the character of a projective module. The following settles Kawanaka's conjecture in the affirmative.

\begin{mtheorem}\label{thm:mainB}
Assume $\bG$ is an adjoint simple group, $p$ is a good prime for $\bG$, and $\mathscr{F} \in \Fam(W)^F$ is an $F$-stable family. Then for each $[a,\psi] \in \mathscr{M}(\bar{A}_{\mathscr{F}})$ there is a Kawanaka character $K_{[a,\psi]}$ of $G$ whose projection on the space spanned by the unipotent characters $\Uch(G)$ is of the form
\begin{equation*}
\rho_{[a,\psi]} + \bigoplus_{\substack{\mathscr{F}' \in \Fam(W)^F\\ \mathcal{O}_{\mathscr{F}} \prec \mathcal{O}_{\mathscr{F}'} }} \mathbb{Z}\Uch(\mathscr{F}'^*).
\end{equation*}
for a unique unipotent character $\rho_{[a,\psi]} \in \Uch(\mathscr{F}^*)$. Moreover, $K_{[a,\psi]}$ is the character of a projective module when $\ell$ is a good prime for $G$.
\end{mtheorem}

\begin{rem}
It is unclear whether the unipotent character $\rho_{[a,\psi]}$ from \cref{thm:mainB} coincides with the unipotent character labelled by $[a,\psi]$ under Lusztig's parametrisation in \cite[\S4]{Lu84}.
\end{rem}

\subsection{Content}
Much of this paper concerns the actual construction and properties of the characters $K_{[a,\psi]}$. The first technical obstruction concerns Lusztig's canonical quotient $\bar A_{\mathscr{F}}$. This is naturally a quotient of the component group $A_\bG(u) := C_\bG(u)/C_\bG^\circ(u)$ with $u \in \mathcal{O}_{\mathscr{F}}$. To define the Kawanaka characters we must find a suitable subgroup $A \leqslant C_{\bG}(u)$ which maps surjectively onto $\bar{A}_{\mathscr{F}}$ under the natural map $C_{\bG}(u) \to A_{\bG}(u) \to \bar{A}_{\mathscr{F}}$.

We note that, in general, the subgroup we choose will not map onto $A_{\bG}(u)$. For the Kawanaka characters to exist, $A$ must satisfy the conditions of what we call an \emph{admissible covering} of $\bar{A}_{\mathscr{F}}$. The exact details of this are given in \cref{def:admissiblesplitting}. Much of the latter part of the paper concerns the existence of such admissible coverings. We show the existence of admissible coverings for classical groups in \cref{ssec:symplectic} and for exceptional groups in \cref{ssec:exc}. For this, we consider ${}^3\D_4$ to be an exceptional group.

Of course, the main issue, which is the core of the paper, is to show that $K_{[a,\psi]}$ can involve at most one unipotent character of the family $\Uch(\mathscr{F}^*)$ and that such a constituent occurs with multiplicity one. This is achieved by considering a Fourier transform $F_{[a,\psi]}$ of $K_{[a,\psi]}$ with respect to $A$. To simplify, let us assume that $A \cong \bar A_{\mathscr{F}}$. The Fourier transforms of unipotent characters in $\Uch(\mathscr{F}^*)$ with respect to the finite group $\bar A_{\mathscr{F}}$ coincide with the (suitably normalised) characteristic functions of certain character sheaves $\USh(\mathscr{F}^*)$ associated to the family $\mathscr{F}$. \Cref{thm:mainB} amounts to showing that $F_{[a,\psi]}$ involves only one character sheaf.

The advantage of working in this setting is that there are very few conjugacy classes where both $F_{[a,\psi]}$ and the characteristic functions of character sheaves in $\USh(\mathscr{F}^*)$ take non-zero values. Furthermore, these values can be explicitly computed. This is shown in \cref{prop:fourier-values} for the class functions $F_{[a,\psi]}$ and recently by Lusztig in \cite{Lu15} for character sheaves. We show in \cref{prop:fouriermultiplicities} that the contribution of $\Uch(\mathscr{F}^*)$ in $F_{[a,\psi]}$ coincides, up to a root of unity, with the characteristic function of the character sheaf associated with $[a,\psi]$ in \cite{Lu15}. Some technical difficulties arise when $A$ is larger than $\bar A_{\mathscr{F}}$. However, in these cases we may take $A$ to be abelian, which allows us to overcome these difficulties and establish \cref{thm:mainB} whenever an admissible covering exists.

\subsection{Applications}
The solution to Kawanaka's conjecture given in \cref{thm:mainB} provides a geometric parametrisation of the unipotent characters, in terms of the Drinfeld doubles of the various admissible coverings constructed in this paper. Since the action of $\mathrm{Aut}(G)$ is easy to describe on the Kawanaka characters we obtain a new way to compute that action on characters. We expect that the same should hold for irreducible characters in other isolated series, hence providing new information on their behaviour under the action of $\mathrm{Aut}(G)$.
Such a strategy was already successfully applied to prove that the inductive McKay condition holds for groups of type $\A$ \cite{CaSp17}, where only GGGCs are needed. We note that such methods can also be applied to Galois automorphisms, which is current work of the third author and A. A. Schaeffer Fry.

On the other hand, a result of Navarro--Tiep \cite{NT11} gives a reduction of the Alperin Weight Conjecture to the case of simple groups. Their reduction involves the action of the automorphism group on irreducible Brauer characters. Under our assumption on $\ell$ the set of unipotent characters $\Uch(G)$ is a basic set for the unipotent blocks of $G$ by \cite[Thm.~A]{Ge93}. As the set $\Uch(G)$ is stable under the action of the automorphism group it follows from \cref{thm:mainA} that there is an $\Aut(G)$-equivariant bijection between $\Uch(G)$ and the irreducible Brauer characters in the unipotent blocks of $G$, see \cite[Lem.~2.3]{Den18}. The action of $\Aut(G)$ on $\Uch(G)$ is well understood, hence we get an explicit description of $\Aut(G)$ on unipotent Brauer characters.

\subsection{Further directions}
There are two restrictions on the prime numbers in \cref{thm:mainA}. The first restriction, on the characteristic $p$ of the ground field, comes from our construction of the projective modules via GGGCs. Recent work of Geck \cite{Ge18} provides an extension of this construction for bad prime numbers $p$. It will be interesting to see whether one can generalise that construction to the Kawanaka characters and extend \cref{thm:mainB} to the case where $p$ is bad. This would involve dealing with Lusztig's special pieces instead of special unipotent classes.

The second restriction concerns the characteristic $\ell$ of the coefficient field of the representations. When $\ell$ is bad one needs to take into account the modular representation theory of the canonical quotient ${\bar A}_\mathscr{F}$ which is no longer an $\ell'$-group. Inducing projective characters of the admissible coverings instead of irreducible characters provides projective characters for the finite reductive group using \cref{thm:mainB}. However, as observed in \cite{Ge94,Cha18} one needs to consider other unipotent classes, called $\ell$-special, and prove a generalisation of Kawanaka's conjecture for some isolated families of non-unipotent characters. This would show that when $\ell$ is bad there is still a unitriangular basic set which is labelled as in \cite{Cha18}.

Finally, let us note that \cref{thm:mainA} provides a classification of unipotent Brauer characters in terms of families. We expect this classification to coincide with the one obtained via the mod-$\ell$ intersection cohomology of Deligne--Lusztig varieties (probably replacing intersection cohomology sheaves by parity sheaves \cite{JMW14}). The existence of $\ell$-torsion in the cohomology should also explain the difference of behaviour between the cases of good and bad characteristic.

\subsection*{Acknowledgments}
The first and second authors gratefully acknowledge financial support by
the ANR grant GeRepMod ANR-16-CE40-0010-01. The second author also acknowledges financial support by
the DFG grant SFB-TRR 195.
During part of this work the third author was supported by an INdAM Marie-Curie Fellowship. Visits of the first and second authors to the University of Padova were supported by the third authors fellowship along with grants CPDA125818/12 and 60A01-4222/15 of the University of Padova. A visit of the third author to Paris VII was supported by the ANR grant ANR-16-CE40-0010-01. The authors warmly thank the hospitality of these institutions for facilitating their collaboration.

Our proof of \cref{prop:fusion-nilp-orb} is based on an argument using Slodowy slices, originally communicated to us by Daniel Juteau. We warmly thank him for sharing his idea with us. We also thank Anthony Henderson for clarifying remarks concerning \cref{lem:slod-slice-prop} in the case of Slodowy slices. Finally, we thank Gunter Malle for his valuable comments on a preliminary version of our paper.

\section{Basic Setup and Notation}\label{sec:basic-setup}
\subsection{Finite Groups}
Throughout, we fix a prime number $\ell$. Assume $G$ is a group acting on a set $\Omega$ with action $\cdot : G \times \Omega \to \Omega$. Unless stated otherwise all group actions are assumed to be left actions. For any set $X$ we have natural left and right actions of $G$ on the set of functions $f : \Omega \to X$ defined by $f^g(\omega) = f(g\cdot\omega)$ and ${}^gf(\omega) = f(g^{-1}\cdot\omega)$ for any $g \in G$ and $\omega \in \Omega$. Similarly we have natural left and right actions of $G$ on the set of functions $f : X \to \Omega$ defined by ${}^gf(x) = g\cdot f(x)$ and $f^g(x) = g^{-1}\cdot f(x)$ for any $g \in G$ and $x \in X$.

If $G$ is a finite group then we denote by $\Irr(G)$ the irreducible $\Ql$-characters of $G$. This is an orthonormal basis of the space of $\Ql$-class functions $\Class(G)$ with respect to the usual inner product
\begin{equation*}
\langle f,f'\rangle_G = \frac{1}{|G|}\sum_{g \in G} f(g)\overline{f'(g)}
\end{equation*}
where $\overline{\phantom{x}} : \Ql \to \Ql$ is a fixed involutive automorphism sending each root of unity to its inverse.

Let $\iota : H \to G$ be a homomorphism of groups then we have a natural restriction map $\iota^* : \Class(G) \to \Class(H)$, defined by $\iota^*(f) = f\circ\iota$, and an induction map $\iota_* : \Class(H) \to \Class(G)$ defined by
\begin{equation*}
(\iota_*(f))(x) = \frac{1}{|G|}\sum_{\substack{(g,h) \in G\times H\\ {}^gx = \iota(h)}} f(h)
\end{equation*}
for all $x \in G$. The maps $\iota^*$ and $\iota_*$ are adjoint with respect to $\langle-,-\rangle_G$ and $\langle -,-\rangle_H$. Moreover, if $\kappa : H \to L$ is another group homomorphism then $(\kappa\circ\iota)^* =  \iota^*\circ\kappa^*$ and $(\kappa\circ\iota)_* = \kappa_* \circ \iota_*$.

We note that if $\iota$ is injective, and we identify $H$ with $\iota(H)$, then $\iota_*$ is simply identified with the usual induction map $\Ind_H^G$. Furthermore, if $\iota$ is surjective then $\iota_*(\chi) = 0$ for any $\chi \in \Irr(H)$ with $\Ker(\iota) \not\leqslant \Ker(\chi)$ and $\iota_*(\chi)$ is the deflation of $\chi$ if $\Ker(\iota) \leqslant \Ker(\chi)$. Hence, in general, we have $\iota_*$ is deflation from $H$ to $\iota(H)$ followed by induction from $\iota(H)$ to $G$.

\subsection{Reductive Groups}\label{subsec:reductive-groups}

Throughout we assume that $\bG$ is a connected reductive algebraic group over $K := \overline{\mathbb{F}}_p$, an algebraic closure of the finite field $\mathbb{F}_p$ with prime cardinality $p>0$. We will tacitly assume that $p$ is a \emph{good} prime for $\bG$ which is different from $\ell$. The Lie algebra of $\bG$ will be denoted by $\Lie(\bG)$ and the derived subgroup will be denoted by $\bG_{\der} \leqslant \bG$.

For any $x \in \Lie(\bG)$ and $g \in \bG$ we will write $g\cdot x$ for the element $\Ad(g)x$ where $\Ad : \bG \to \GL(V)$ is the adjoint representation, with $V = \Lie(\bG)$. This defines an action of $\bG$ on $\Lie(\bG)$ which we will refer to as the adjoint action. Moreover, we denote by $\ad~:~\Lie(\bG) \to \GL(V)$ the derivative of $\Ad$. This is the adjoint representation of $\Lie(\bG)$ which satisfies $(\ad x)(y) = [x,y]$ for all $x,y\in \Lie(\bG)$ where $[-,-] : \Lie(\bG)\times\Lie(\bG) \to \Lie(\bG)$ is the Lie bracket.

Let $F : \bG \to \bG$ be a Frobenius endomorphism endowing $\bG$ with an $\mathbb{F}_q$-rational structure. Given an $F$-stable subgroup $\bH$ of $\bG$ we will often denote by $H:=\bH^F$ the finite group of fixed points under $F$.
We assume fixed an $F$-stable maximal torus and Borel subgroup $\bT_0 \leqslant \bB_0 \leqslant \bG$. If $\bT \leqslant \bG$ is any $F$-stable maximal torus of $\bG$ then $F$ induces an automorphism of the Weyl group $W_{\bG}(\bT) := N_{\bG}(\bT)/\bT$, which we also denote by $F$. We will typically denote the Weyl group $W_{\bG}(\bT_0)$ simply by $W$.

We recall here that a morphism $\varphi : \bG \to \widetilde{\bG}$ between algebraic groups is an \emph{isotypic morphism} if the image $\varphi(\bG)$ contains the derived subgroup of $\widetilde{\bG}$ and the kernel $\Ker(\varphi)$ is contained in the centre $Z(\bG)$ of $\bG$. If $\widetilde{\bG}$ is a connected reductive algebraic group and $\varphi$ is isotypic then $\widetilde{\bG} = \varphi(\bG)\cdot Z(\widetilde{\bG})$. If one of these groups is defined over $\mathbb{F}_q$ then we will implicitly assume $\varphi$ is also defined over $\mathbb{F}_q$.

If $\mathbb{G}_m$ denotes the multiplicative group of the field $K$ then for any algebraic group $\bH$ we denote by $X(\bH) = \Hom(\bH,\mathbb{G}_m)$ the character group of $\bH$ and $\widecheck{X}(\bH) = \Hom(\mathbb{G}_m,\bH)$ the set of cocharacters. The product on $X(\bH)$ will be written additively. As above, $\bH$ acts on $X(\bH)$ and $\widecheck{X}(\bH)$ and we denote by $C_{\bH}(\lambda) = C_{\bH}(\lambda(\mathbb{G}_m))$ the stabiliser of $\lambda \in \widecheck{X}(\bH)$.

We note that if $\bH$ is abelian then $\widecheck{X}(\bH)$ is also naturally an abelian group whose product we again write additively. Moreover, if $\bH$ is a torus then we denote by $\langle -,-\rangle : X(\bH) \times \widecheck{X}(\bH) \to \mathbb{Z}$ the usual perfect pairing. We also denote by $\langle -,-\rangle$ the natural extension of this pairing to a non-degenerate $\mathbb{Q}$-bilinear form $\big(\mathbb{Q}\otimes_{\mathbb{Z}}X(\bH)\big) \times \big(\mathbb{Q}\otimes_{\mathbb{Z}}\widecheck{X}(\bH)\big) \to \mathbb{Q}$.

Let us recall that to the triple $(\bG,\bB_0,\bT_0)$ we have a corresponding based root datum
\begin{equation*}
\mathcal{R}(\bG,\bT_0,\bB_0) = (X(\bT_0),\Phi(\bT_0),\Delta(\bT_0,\bB_0),\widecheck{X}(\bT_0),\widecheck{\Phi}(\bT_0),\widecheck{\Delta}(\bT_0,\bB_0)).
\end{equation*}
Here $\Phi(\bT_0) \subset X(\bT_0)$ are the roots of $\bG$ relative to $\bT_0$, and $\Delta(\bT_0,\bB_0) \subseteq \Phi(\bT_0)$ are the simple roots determined by $\bB_0$. Similarly $\widecheck{\Phi}(\bT_0) \subset \widecheck{X}(\bT_0)$ are the coroots and $\widecheck{\Delta}(\bT_0,\bB_0) \subseteq \widecheck{\Phi}(\bT_0)$ are the simple coroots. Forgetting the simple roots and coroots we get the usual root datum
\begin{equation*}
\mathcal{R}(\bG,\bT_0) = (X(\bT_0),\Phi(\bT_0),\widecheck{X}(\bT_0),\widecheck{\Phi}(\bT_0)).
\end{equation*}
We will say a root datum $\mathcal{R} = (X,\Phi,\widecheck{X},\widecheck{\Phi})$ is \emph{torsion free} if both the quotients $X/\mathbb{Z}\Phi$ and $\widecheck{X}/\mathbb{Z}\widecheck{\Phi}$ have no torsion. Note that all root data are assumed to be reduced.

\part{Unipotent classes of reductive groups}
\section{Nilpotent Orbits}
We will denote by $\mathcal{N}(\bG) \subseteq \Lie(\bG)$ the nilpotent cone of the Lie algebra. The adjoint action of $\bG$ on $\lie{g} := \Lie(\bG)$ preserves $\mathcal{N}(\bG)$ and any orbit of this restricted action is called a \emph{nilpotent orbit}. The resulting set of orbits will be denoted by $\mathcal{N}(\bG)/\bG$.

\subsection{Cocharacters and Gradings}\label{subsection:cocharacters}
If $\lambda \in \widecheck{X}(\bG)$ is a cocharacter then we have a corresponding grading $\lie{g} = \bigoplus_{i \in \mathbb{Z}} \lie{g}(\lambda,i)$ where
\begin{equation*}
\lie{g}(\lambda,i) = \{x \in \lie{g} \mid \lambda(k)\cdot x = k^ix\text{ for all }k \in \mathbb{G}_m\}.
\end{equation*}
We note that $\lie{g}(\lambda,-i) = \lie{g}(-\lambda,i)$ for all $i \in \mathbb{Z}$. The group $C_{\bG}(\lambda)$ preserves each weight space $\lie{g}(\lambda,i)$ and, by a result of Richardson \cite{Ric85}, has finitely many orbits on $\lie{g}(\lambda,i)$ for any $i\neq 0$. In particular, there exists a unique open dense $C_{\bG}(\lambda)$-orbit $\lie{g}(\lambda,i)_{\reg} \subseteq \lie{g}(\lambda,i)$ for each $i \neq 0$.

Associated to $\lambda$ we have a corresponding parabolic subgroup $\bP_{\bG}(\lambda) \leqslant \bG$ defined as in \cite[3.2.15, 8.4.5]{Spr09}. This parabolic subgroup has the property that 
\begin{equation*}
\Lie(\bP_{\bG}(\lambda)) = \bigoplus_{i \geqslant 0} \lie{g}(\lambda,i).
\end{equation*}
The group $\bL_{\bG}(\lambda) := C_{\bG}(\lambda) \leqslant \bP_{\bG}(\lambda)$ is a Levi complement of the parabolic. In particular, if $\bU_{\bG}(\lambda) \lhd \bP_{\bG}(\lambda)$ is the unipotent radical then $\bP_{\bG}(\lambda) = \bL_{\bG}(\lambda) \ltimes \bU_{\bG}(\lambda)$. The Lie algebras of these subgroups have the following weight space decompositions
\begin{equation*}
\Lie(\bL_{\bG}(\lambda)) = \lie{g}(\lambda,0) \qquad\text{and}\qquad \Lie(\bU_{\bG}(\lambda)) = \bigoplus_{i > 0} \lie{g}(\lambda,i).
\end{equation*}
Furthermore, we have $\bP_{\bG}(-\lambda)$ is the unique opposite parabolic subgroup, i.e., we have $\bP_{\bG}(\lambda) \cap \bP_{\bG}(-\lambda) = \bL_{\bG}(\lambda) = \bL_{\bG}(-\lambda)$.

For any integer $i > 0$ there exist unique closed connected unipotent subgroups $\bU_{\bG}(\lambda,i) \leqslant \bU_{\bG}(\lambda)$ and $\bU_{\bG}(\lambda,-i) \leqslant \bU_{\bG}(-\lambda)$ such that $\Lie(\bU_{\bG}(\lambda,i)) = \lie{u_g}(\lambda,i)$ and $\Lie(\bU_{\bG}(\lambda,-i)) = \lie{u_g}(\lambda,-i)$ where
\begin{equation*}
\lie{u_g}(\lambda,i) :=  \bigoplus_{j \geqslant i} \lie{g}(\lambda,j)\qquad\text{and}\qquad\lie{u_g}(\lambda,-i) :=  \bigoplus_{j \geqslant i} \lie{g}(\lambda,-j).
\end{equation*}
Indeed $\bU_{\bG}(\lambda,i)$ is simply generated by the root subgroups of $\bG$, defined with respect to some maximal torus in $C_{\bG}(\lambda)$, corresponding to the root spaces that are contained in one of the weight spaces $\lie{g}(\lambda,j)$ for $j \geq i$.

With this notation we clearly have for any $i >0$ that $\bU_{\bG}(\lambda,-i) = \bU_{\bG}(-\lambda,i)$ and $\bU_{\bG}(\lambda,1) = \bU_{\bG}(\lambda)$ is the unipotent radical of the parabolic. If the ambient group $\bG$ is clear then we will drop the subscripts writing $\bP(\lambda)$, $\bU(\lambda)$, $\bL(\lambda)$, $\lie{u}(\lambda,i)$, etc., instead of $\bP_{\bG}(\lambda)$, $\bU_{\bG}(\lambda)$, $\bL_{\bG}(\lambda)$, $\lie{u_g}(\lambda,i)$, etc. For any $g \in \bG$ and cocharacter $\lambda \in \widecheck{X}(\bG)$ we have $\bP({}^g\lambda) = {}^g\bP(\lambda)$ and $\bL({}^g\lambda) = {}^g\bL(\lambda)$. Moreover, for each integer $i \in \mathbb{Z}$ we get that
\begin{align}\label{eq:conj-cochar-spaces}
\lie{g}({}^g\lambda,i) = g\cdot\lie{g}(\lambda,i) && \lie{g}({}^g\lambda,i)_{\reg} = g\cdot\lie{g}(\lambda,i)_{\reg},
\end{align}
where the last equality is defined only when $i \neq 0$. This implies immediately that $\bU({}^g\lambda,i) = {}^g\bU(\lambda,i)$ for any $0 \neq i \in \mathbb{Z}$, in particular $\bU({}^g\lambda) = {}^g\bU(\lambda)$.

Note also that we have a map $F : \widecheck{X}(\bT_0) \to \widecheck{X}(\bT_0)$, denoted $\lambda \mapsto F\cdot \lambda$, defined as follows. We have a Frobenius endomorphism $F_q : \mathbb{G}_m \to \mathbb{G}_m$ defined by $F_q(k) = k^q$. Given a cocharacter $\lambda \in \widecheck{X}(\bG)$ we define, as in \cite[3.23]{Tay16}, a cocharacter $F\cdot\lambda = F\circ\lambda\circ F_q^{-1} \in \widecheck{X}(\bG)$. We denote by $\widecheck{X}(\bG)^F$ those characters $\lambda \in \widecheck{X}(\bG)$ satisfying $F\cdot\lambda = \lambda$.

From the definition of $\bP(\lambda)$ and $\bL(\lambda)$ it follows that $F(\bP(\lambda)) = \bP(F\cdot\lambda)$ and $F(\bL(\lambda)) = \bL(F\cdot\lambda)$. In particular, if $\lambda \in \widecheck{X}(\bG)^F$ then $\bP(\lambda)$ and $\bL(\lambda)$ are $F$-stable. Moreover, for any $i \in \mathbb{Z}$ we have
\begin{align}\label{eq:Frob-weight-space}
F(\lie{g}(\lambda,i)) = \lie{g}(F\cdot\lambda,i) && F(\lie{g}(\lambda,i)_{\reg}) = \lie{g}(F\cdot\lambda,i)_{\reg},
\end{align}
where, again, the last equality is defined only when $i \neq 0$. This implies $F(\bU(\lambda,i)) = \bU(F\cdot\lambda,i)$ for any $i > 0$ so $\bU(\lambda,i)$ is $F$-stable if $\lambda \in \widecheck{X}(\bG)^F$.

\begin{rem}
Let us briefly point out why $F\cdot\lambda$ is a cocharacter. For this, let $\sigma : K[\bG] \to K[\bG]$ and $\sigma_q : K[\mathbb{G}_m] \to K[\mathbb{G}_m]$ be the arithmetic Frobenius endomorphisms given by $\sigma(f) = (F^*)^{-1}(f^q)$ and $\sigma_q(f) = (F_q^*)^{-1}(f^q)$ respectively. Note that $\sigma$ and $\sigma_q$ are semilinear ring isomorphisms in the sense that $\sigma(af) = a^q\sigma(f)$ for all $a \in K$ and $f \in K[\bG]$. With this we obtain a bijection
\begin{align*}
\Hom_{K\text{-alg}}(K[\bG],K[\mathbb{G}_m]) &\to \Hom_{K\text{-alg}}(K[\bG],K[\mathbb{G}_m]),\\
\gamma &\mapsto \gamma^{[q]} := \sigma_q \circ \gamma \circ \sigma^{-1}.
\end{align*}
Now for any $f \in K[\bG]$ we have $\sigma^{-1}(f)^q = F^*(f)$ hence we get that $(\lambda^*)^{[q]} = (F_q^*)^{-1} \circ \lambda^* \circ F^* = (F \cdot \lambda)^*$. This implies that $F\cdot \lambda$ is a cocharacter.
\end{rem}

\subsection{Weighted Dynkin Diagrams}\label{subsec:weighted-dynkin}
Let us fix a triple $(\bG_{\mathbb{C}},\bT_{\mathbb{C}},\bB_{\mathbb{C}})$, defined over $\mathbb{C}$, such that $\mathcal{R}(\bG_{\mathbb{C}},\bT_{\mathbb{C}},\bB_{\mathbb{C}}) \cong \mathcal{R}(\bG,\bT_0,\bB_0)$. We will denote by $\Phi = \Phi(\bT_{\mathbb{C}}) = \Phi(\bT_0)$ and $\Delta = \Delta(\bT_{\mathbb{C}},\bB_{\mathbb{C}}) = \Delta(\bT_0,\bB_0)$ the common sets of roots and simple roots which are identified under the isomorphism of based root data. Now, let $\{0\} \neq \mathcal{O} \in \mathcal{N}(\bG_{\mathbb{C}})/\bG_{\mathbb{C}}$ be a non-zero nilpotent orbit and let $\{e,h,f\} \subseteq \lie{g}_{\mathbb{C}} := \Lie(\bG_{\mathbb{C}})$ be an $\lie{sl}_2$-triple such that $e \in \mathcal{O}$. As $\lie{g}_{\mathbb{C}}$ is a Lie algebra over $\mathbb{C}$ such a triple exists by the Jacobson--Morozov Theorem, see for example \cite[Thm.~3.3.1]{CoMcG93}.

After possibly replacing $\{e,h,f\}$ by a $\bG_{\mathbb{C}}$-conjugate we can assume that $h \in \Lie(\bT_{\mathbb{C}})$ and $\alpha(h) \geqslant 0$ for all $\alpha \in \Delta$. Here we identify $\Phi$ with a subset of $\Hom(\Lie(\bT_{\mathbb{C}}),\mathbb{C})$ by differentiating. We then get an additive function $d_{\mathcal{O}} : \Phi \to \mathbb{Z}$, defined by
\begin{equation}\label{eq:def-weighted-dynk}
d_{\mathcal{O}}(\alpha) = \alpha(h),
\end{equation}
which is called a \emph{weighted Dynkin diagram}; by convention we set $d_{\{0\}}$ equal to the $0$ function, i.e., $d_{\{0\}}(\alpha) = 0$ for all $\alpha \in \Phi$. The set of all weighted Dynkin diagrams will be denoted by $\mathcal{D}(\Phi,\Delta) = \{d_{\mathcal{O}} \mid \mathcal{O} \in \mathcal{N}(\bG_{\mathbb{C}})/\bG_{\mathbb{C}}\}$. It is known that the assignment $\mathcal{N}(\bG_{\mathbb{C}})/\bG_{\mathbb{C}} \to \mathcal{D}(\Phi,\Delta)$, given by $\mathcal{O} \mapsto d_{\mathcal{O}}$, is a well-defined bijection.

Now assume $d \in \mathcal{D}(\Phi,\Delta)$ then there exists a unique vector $\lambda_d \in \mathbb{Q}\otimes_{\mathbb{Z}} \mathbb{Z}\widecheck{\Phi}$ such that
\begin{equation*}
d(\alpha) = \langle \alpha,\lambda_d\rangle
\end{equation*}
for all $\alpha \in \Phi$. Indeed, if $\widecheck{\omega}_{\beta} \in \mathbb{Q}\otimes_{\mathbb{Z}} \mathbb{Z}\widecheck{\Phi}$ denotes the dual basis element to $\beta \in \Delta$, so that $\langle \alpha, \widecheck{\omega}_{\beta}\rangle = \delta_{\alpha,\beta}$ for any $\alpha,\beta \in \Delta$, then $\lambda_d = \sum_{\alpha \in \Delta} d(\alpha)\widecheck{\omega}_{\alpha}$. It follows from the remark after \cite[5.6.5]{Car} that $\lambda_d \in \mathbb{Z}\widecheck{\Phi}$ so we obtain a unique cocharacter $\lambda_d \in \widecheck{X}(\bT_0) \subseteq \widecheck{X}(\bG)$. The following result, due to Kawanaka and Premet, gives the parameterisation of nilpotent orbits in terms of weighted Dynkin diagrams, see \cite[Thm.~2.1.1]{Kaw86}, \cite[Thm.~2.7]{Pre03} and \cite[3.22]{Tay16}.

\begin{thm}[Kawanaka, Premet]\label{thm:classification-nil-orbits}
Recall that $p$ is a good prime for $\bG$. 
The map $\mathcal{D}(\Phi,\Delta) \to \mathcal{N}(\bG)/\bG$ defined by
\begin{equation*}
d \mapsto \bG\cdot\lie{g}(\lambda_d,2)_{\reg}
\end{equation*}
is a bijection.
\end{thm}

Let $\mathcal{D}(\bG) = \{{}^g\lambda_d \mid d \in \mathcal{D}(\Phi,\Delta)$ and $g \in \bG\} \subseteq \widecheck{X}(\bG)$ be all the conjugates of the cocharacters coming from the weighted Dynkin diagrams; this is clearly a $\bG$-invariant subset of $\widecheck{X}(\bG)$. The elements of $\mathcal{D}(\bG)$ are called \emph{Dynkin cocharacters}. Note that the map $\mathcal{D}(\Phi,\Delta) \to \mathcal{D}(\bG)$, given by $d \mapsto \lambda_d$, defines a bijection $\mathcal{D}(\Phi,\Delta) \to \mathcal{D}(\bG)/\bG$. This is clear from \cref{thm:classification-nil-orbits,eq:conj-cochar-spaces} because if $d, d' \in \mathcal{D}(\Phi,\Delta)$ are weighted Dynkin diagrams and $\lambda_{d'} = {}^g\lambda_d$ then $\bG\cdot\lie{g}(\lambda_{d'},2)_{\reg} = \bG\cdot\lie{g}(\lambda_d,2)_{\reg}$, which implies $d = d'$. With this we can trivially rephrase \cref{thm:classification-nil-orbits} to get that the map $\mathcal{D}(\bG) \to \mathcal{N}(\bG)/\bG$ defined by $\lambda \mapsto \bG \cdot \lie{g}(\lambda,2)_{\reg}$ induces a bijection $\mathcal{D}(\bG)/\bG \to \mathcal{N}(\bG)/\bG$.

If $e \in \mathcal{N}(\bG)$ is a nilpotent element then we set
\begin{equation*}
\mathcal{D}_e(\bG) = \{\lambda \in \mathcal{D}(\bG) \mid e \in \lie{g}(\lambda,2)_{\reg}\}.
\end{equation*}
For any closed subgroup $\bH \leqslant \bG$ we set $\mathcal{D}_e(\bG,\bH) := \mathcal{D}_e(\bG) \cap \widecheck{X}(\bH)$. Now assume $d \in \mathcal{D}(\Phi,\Delta)$ is the weighted Dynkin diagram such that $e \in \bG\cdot \lie{g}(\lambda_d,2)_{\reg}$. If $\lambda \in \mathcal{D}_e(\bG)$ then we must have $\bG\cdot \lie{g}(\lambda,2)_{\reg} = \bG\cdot \lie{g}(\lambda_d,2)_{\reg}$ so $\lambda = {}^g\lambda_d$. In particular, we have
\begin{equation}\label{eq:dynk-cochars-are-conj}
\mathcal{D}_e(\bG) \subseteq {}^\bG\lambda_d,
\end{equation}
where ${}^\bG\lambda_d$ is the $\bG$-orbit of $\lambda_d$. Note also that $\mathcal{D}_{g\cdot e}(\bG) = {}^g\mathcal{D}_e(\bG)$ for any $g \in \bG$. In particular, the group $C_\bG(e)$ acts on $\mathcal{D}_e(\bG)$. We will see in \cref{lem:assoc-dynkin-cocharacters} that this action is transitive.

\subsection{Action by the Longest Element}
Let $w_0 \in W = W_{\bG}(\bT_0)$ denote the longest element of the Weyl group, relative to the choice of simple roots $\Delta = \Delta(\bT_0,\bB_0)$. Recall that we have a permutation of the roots $\rho : \Phi \to \Phi$ given by $\rho(\alpha) = -{}^{w_0}\alpha$ which satisfies $\rho(\Delta) = \Delta$. We will need the following concerning this action on the set of weighted Dynkin diagrams.

\begin{prop}\label{prop:weighted-dynk-conj-to-neg}
For any weighted Dynkin diagram $d \in \mathcal{D}(\Phi,\Delta)$ we have $d\circ\rho = d$ and ${}^{w_0}\lambda_d = -\lambda_d$.
\end{prop}

\begin{proof}
The case $d=0$ holds trivially, so we assume $d \neq 0$. As this statement is purely about the action of the Weyl group on the cocharacter lattice we may answer this over $\mathbb{C}$. In fact, it is clear that it suffices to show this assuming that $\bG_{\mathbb{C}}$ is adjoint; so we assume this is the case.

Let us choose a homomorphism $\varphi : \SL_2(\mathbb{C}) \to \bG_{\mathbb{C}}$ such that the image $\mathrm{Im}(\varphi) \leqslant \bG_{\mathbb{C}}$ is a closed subgroup and the kernel is central. Moreover, let $d_1\varphi : \lie{sl}_2(\mathbb{C}) \to \lie{g}_{\mathbb{C}}$ be the derivative of $\varphi$ and set
\begin{align*}
e &= d_1\varphi\left(\begin{bmatrix}
0 & 0\\
1 & 0
\end{bmatrix}\right), &
h &= d_1\varphi\left(\begin{bmatrix}
-1 & 0\\
0 & 1
\end{bmatrix}\right),
&
f &= d_1\varphi\left(\begin{bmatrix}
0 & 1\\
0 & 0
\end{bmatrix}\right).
\end{align*}
We then have $\{e,h,f\} \subseteq \lie{g}_{\mathbb{C}}$ is an $\lie{sl}_2$-triple. It follows from \cite[5.5.5, 5.5.6]{Car}, which applies in characteristic $0$, that we may assume $\varphi$ is chosen such that $\{e,h,f\}$ realises the weighted Dynkin diagram $d$ as in \cref{eq:def-weighted-dynk}. Moreover, the remark after \cite[5.6.5]{Car} implies that
\begin{equation*}
\lambda_d(k) = \varphi\begin{bmatrix}
k^{-1} & 0\\
0 & k
\end{bmatrix}
\end{equation*}
for any $k \in \mathbb{G}_m$.

Consider the usual element
\begin{equation*}
s = \begin{bmatrix}
0 & -1\\
1 & 0
\end{bmatrix} \in \SL_2(\mathbb{C})
\end{equation*}
and set $\tilde{s} = \varphi(s) \in \bG_{\mathbb{C}}$. By a straightforward computation we have that ${}^{\tilde{s}}\lambda_d = -\lambda_d$. Now, if $g = \dot{w}_0^{-1}\tilde{s}$, with $\dot{w}_0 \in N_{\bG_{\mathbb{C}}}(\bT_{\mathbb{C}})$ a representative of $w_0$, then for any $\alpha \in \Delta$ we have
\begin{equation*}
\langle \alpha , {}^g\lambda_d \rangle = \langle \alpha, -{}^{\dot{w}_0^{-1}}\lambda_d) \rangle = \langle - {}^{w_0}\alpha, \lambda_d \rangle =\langle \rho(\alpha) , \lambda_d \rangle =  d(\rho(\alpha)) \geq 0
\end{equation*}
because $\rho(\alpha) \in \Delta$. This implies that $d \circ\rho \in \mathcal{D}(\Phi,\Delta)$ is a weighted Dynkin diagram, being afforded by the triple $\{g\cdot e,g\cdot h,g\cdot f\}$ and the Dynkin cocharacter ${}^g\lambda_d$. However, $g\cdot e$ is in the same $\bG_{\mathbb{C}}$-orbit as $e$ so $d\circ\rho = d$. Correspondingly, we must have ${}^g\lambda_d = \lambda_d$ so ${}^{\dot w_0}\lambda_d = {}^{\tilde{s}}\lambda_d = -\lambda_d$.
\end{proof}

\begin{cor}\label{cor:conj-to-neg-space}
For any weighted Dynkin diagram $d \in \mathcal{D}(\Phi,\Delta)$ we have $\bG\cdot \lie{g}(\lambda_d,-2)_{\reg} = \bG\cdot \lie{g}(\lambda_d,2)_{\reg}$.
\end{cor}

\begin{proof}
This is clear as $\lie{g}(-\lambda,i) = \lie{g}(\lambda,-i)$.
\end{proof}

\begin{cor}\label{cor:conj-to-neg-space-rat}
Any rational cocharacter $\lambda \in \mathcal{D}(\bG)^F := \mathcal{D}(\bG) \cap \widecheck{X}(\bG)^F$ is $\bG^F$-conjugate to $-\lambda$. In particular, for any rational element $e \in \lie{g}(\lambda,2)_{\reg}^F$ we have
\begin{equation*}
(\bG^F\cdot e) \cap \lie{g}(\lambda,-2)_{\reg} \neq \emptyset.
\end{equation*}
\end{cor}

\begin{proof}
We have $\lambda = {}^x\lambda_d$ for some Dynkin diagram $d \in \mathcal{D}(\Phi,\Delta)$ and element $x \in \bG$. Therefore by \cref{prop:weighted-dynk-conj-to-neg} there exists an element $g \in \bG$ such that ${}^g\lambda = -\lambda$. As $F\cdot \lambda = \lambda$ we must have $F\cdot (-\lambda) = -\lambda$ so ${}^{F(g)}\lambda = {}^g\lambda$. By a standard application of the Lang--Steinberg theorem inside the connected (Levi) subgroup $C_{\bG}(\lambda)$ we may thus assume that $g \in \bG^F$ and ${}^g\lambda = -\lambda$. After \cref{eq:conj-cochar-spaces} we get that
\begin{equation*}
g\cdot e \in \lie{g}({}^g\lambda,2)_{\reg} = \lie{g}(-\lambda,2)_{\reg} = \lie{g}(\lambda,-2)_{\reg},
\end{equation*}
which proves the last statement.
\end{proof}

\subsection{Associated Cocharacters}\label{ssec:cocharacter}
Let $e \in \mathcal{N}(\bG)$ be a nilpotent element. Recall that $e$ is said to be \emph{$
\bG$-distinguished}, or simply \emph{distinguished}, if every torus $\bT \leqslant C_{\bG}(e)$ is contained in $Z(\bG)$. Following \cite[5.3]{Jan04} we say that a cocharacter $\lambda \in \widecheck{X}(\bG)$ is \emph{associated to $e$} if the following hold:
\begin{itemize}
	\item $e \in \lie{g}(\lambda,2)$,
	\item there exists a Levi subgroup $\bL \leqslant \bG$ such that $e \in \mathcal{N}(\bL)$ is $\bL$-distinguished and $\lambda(\mathbb{G}_m) \leqslant \bL_{\der}$.
\end{itemize}
We will denote by $\mathcal{A}_e(\bG) \subseteq \widecheck{X}(\bG)$ the subset of cocharacters that are associated to $e$. If $\bH \leqslant \bG$ is a closed subgroup then, as above, we set $\mathcal{A}_e(\bG,\bH) = \mathcal{A}_e(\bG) \cap \widecheck{X}(\bH)$. The two sets of cocharacters $\mathcal{A}_e(\bG)$ and $\mathcal{D}_e(\bG)$ represent the two classifications of nilpotent orbits; namely the Bala--Carter classification and the classification by weighted Dynkin diagrams. The following clarifies the relationship between these two sets; this is essentially shown in \cite{Pre03}.

\begin{lem}[Premet]\label{lem:assoc-dynkin-cocharacters}
For any nilpotent element $e \in \mathcal{N}(\bG)$ we have $\mathcal{A}_e(\bG) = \mathcal{D}_e(\bG)$. Consequently, $C_\bG^\circ(e)$ acts transitively on $\mathcal{D}_e(\bG)$.
\end{lem}

\begin{proof}
By \cref{thm:classification-nil-orbits} there exists a unique weighted Dynkin diagram $d \in \mathcal{D}(\Phi,\Delta)$ such that $e \in \bG \cdot \lie{g}(\lambda_d,2)_{\reg}$. Arguing as in the proof of \cite[3.22]{Tay16} it follows from \cite[Prop.~2.5]{Pre03} and \cite[Prop.~16]{Mcn04} that $\lambda_d$ is associated to any nilpotent element $f \in \lie{g}(\lambda_d,2)_{\reg}$. In particular, there exists a Levi subgroup $\bM \leqslant \bG$ such that $f \in \mathcal{N}(\bM)$ is $\bM$-distinguished and $\lambda_d(\mathbb{G}_m) \leqslant \bM_{\der}$.

Now assume $\lambda \in \mathcal{D}_e(\bG)$ is a cocharacter then $\lambda = {}^g\lambda_d$ for some $g \in \bG$ by \cref{eq:dynk-cochars-are-conj}. By assumption $e \in \lie{g}(\lambda,2)_{\reg} = g\cdot\lie{g}(\lambda_d,2)_{\reg}$ so there exists an element $f \in \lie{g}(\lambda_d,2)_{\reg}$ such that $e = g\cdot f$. If $\bL = {}^g\bM$ then we have $\Lie(\bL) = g\cdot \Lie(\bM)$, so $e \in \Lie(\bL)$, and $\lambda(\mathbb{G}_m) \leqslant \bL_{\der} = {}^g\bM_{\der}$. As $C_{\bL}(e) = {}^gC_{\bM}(f)$ we clearly have $e$ is $\bL$-distinguished so $\lambda \in \mathcal{A}_e(\bG)$ is associated to $e$. This shows $\mathcal{D}_e(\bG) \subseteq \mathcal{A}_e(\bG)$.

Finally, by \cite[5.3]{Jan04} the group $C_{\bG}^{\circ}(e)$ acts transitively on $\mathcal{A}_e(\bG)$. Since it stabilises $\mathcal{D}_e(\bG)$ (see \cref{subsec:weighted-dynkin}) we must have $\mathcal{D}_e(\bG) = \mathcal{A}_e(\bG)$.
\end{proof}

\section{Unipotent Classes}\label{sec:unipotent-classes}

\subsection{Springer Morphisms}\label{subsec:springer-morph}
We will denote by $\mathcal{U}(\bG) \subseteq \bG$ the unipotent variety of $\bG$. Clearly the conjugation action of $\bG$ on itself preserves $\mathcal{U}(\bG)$ and any orbit of this restricted action is called a \emph{unipotent class}. The resulting set of orbits will be denoted by $\mathcal{U}(\bG)/\bG$.

As $\bG$ is equipped with a Frobenius endomorphism $F : \bG \to \bG$ so is its Lie algebra $\lie{g}$; we will denote this again by $F : \lie{g} \to \lie{g}$. Note that the adjoint action is $F$-equivariant, in the sense that
\begin{equation*}
F(g\cdot x) = F(g)\cdot F(x)
\end{equation*}
for any $g \in \bG$ and $x \in \lie{g}$. Recall that a \emph{Springer homeomorphism} is a homeomorphism $\phi_{\spr} : \mathcal{U}(\bG) \to \mathcal{N}(\bG)$ which is $\bG$-equivariant, in the sense that $\phi_{\spr}({}^g u) = g \cdot \phi_{\spr}(u)$ for all $g \in \bG$ and $u \in \mathcal{U}(\bG)$, and $F$-equivariant, in the sense that $F\circ\phi_{\spr} = \phi_{\spr}\circ F$. If such a homeomorphism exists then we obtain a corresponding $F$-equivariant bijection 
$$\mathcal{U}(\bG)/\bG \simto \mathcal{N}(\bG)/\bG$$
between the unipotent classes and nilpotent orbits of $\bG$.

In general there can be many different Springer homeomorphisms, as observed by Serre \cite[\S10]{Mcn05}. In this regard, we will need the following classic result of Springer and Serre, see \cite[III, 3.12]{SpSt} and \cite[\S10]{Mcn05}.

\begin{thm}[Springer, Serre]\label{thm:springer-morphism}
Recall that $p$ is assumed to be good for $\bG$. There exists a Springer homeomorphism $\phi_{\spr} : \mathcal{U}(\bG) \to \mathcal{N}(\bG)$ and the resulting bijection $\mathcal{U}(\bG)/\bG \simto \mathcal{N}(\bG)/\bG$ given by $\phi_{\spr}$ is independent of the choice of $\phi_{\spr}$.
\end{thm}

We often have a stronger statement available to us concerning Springer homeomorphisms. Namely, that there exists a Springer homeomorphism which is an isomorphism of varieties; we call such a morphism a \emph{Springer isomorphism}. In general, Springer isomorphisms need not exist even in good characteristic, as is exhibited by the example of $\bG = \PGL_2$ when $p=2$, see \cite[7.0.3]{Sob18}. However, if $\bG$ is proximate, in the following sense, then every Springer homeomorphism is an isomorphism, see \cite[3.4]{Tay16}.

\begin{definition}[{}{\cite[2.10]{Tay16}}]\label{def:proximate}
We say $\bG$ is \emph{proximate} if some (any) simply connected covering $\bG_{\simc} \to \bG_{\der}$ of the derived subgroup $\bG_{\der} \leqslant \bG$ is a separable morphism.
\end{definition}

The assumption that $\bG$ is proximate is required for the construction of Kawanaka's GGGCs. We take this opportunity to mention the following fact concerning proximate groups, which will be used in \cref{sec:GGGCs}.

\begin{lem}\label{lem:subs-prox-are-prox}
Assume $\bG$ is proximate and $p$ is good for $\bG$ then any closed connected reductive subgroup $\bH \leqslant \bG$ of maximal rank is also proximate.
\end{lem}

\begin{proof}
Let $\bT \leqslant \bH$ be a maximal torus of $\bH$ then, by assumption, we have $\bT \leqslant \bG$ is a maximal torus of $\bG$. Let $\mathcal{R}(\bG,\bT) = (X,\Phi,\widecheck{X},\widecheck{\Phi})$ and $\mathcal{R}(\bH,\bT) = (X,\Psi,\widecheck{X},\widecheck{\Psi})$ be the corresponding root data. As remarked in \cite[2.15]{Tay16} we have $\bG$ is proximate if and only if $\widecheck{X}/\mathbb{Z}\widecheck{\Phi}$ has no $p$-torsion.

Now, we have a short exact sequence
\begin{equation*}
0 \longrightarrow \mathbb{Z}\widecheck{\Phi}/\mathbb{Z}\widecheck{\Psi} \longrightarrow \widecheck{X}/\mathbb{Z}\widecheck{\Psi} \longrightarrow \widecheck{X}/\mathbb{Z}\widecheck{\Phi} \longrightarrow 0.
\end{equation*}
Since $p$ is good for $\bG$ we have $\mathbb{Z}\widecheck{\Phi}/\mathbb{Z}\widecheck{\Psi}$ has no $p$-torsion by \cite[4.5]{SpSt}. On the other hand, $\widecheck{X}/\mathbb{Z}\widecheck{\Phi}$ has no $p$-torsion since $\bG$ is proximate. Therefore by the previous short exact sequence, $\widecheck{X}/\mathbb{Z}\widecheck{\Psi}$ has no $p$-torsion and $\bH$ is proximate.
\end{proof}

\subsection{Canonical Parabolic and Levi Subgroups}\label{subsec:canonical-para-levi}
We assume fixed a Springer homeomorphism $\phi_{\spr} : \mathcal{U}(\bG) \to \mathcal{N}(\bG)$ as in \cref{thm:springer-morphism}. Using this Springer homeomorphism we may invoke the classification of nilpotent orbits, as in \cref{thm:classification-nil-orbits}, to obtain a classification of unipotent conjugacy classes. Now, for any unipotent element $u \in \mathcal{U}(\bG)$ and closed subgroup $\bH \leqslant \bG$ we set
\begin{equation*}
\mathcal{D}_u(\bG) := \mathcal{D}_{\phi_{\spr}(u)}(\bG) \qquad \text{and} \qquad \mathcal{D}_u(\bG,\bH) := \mathcal{D}_{\phi_{\spr}(u)}(\bG,\bH).
\end{equation*}

If $\bH$ is any closed subgroup of $\bG$ containing $u$, we will write $A_{\bH}(u)$ for the component group $C_{\bH}(u)/C_{\bH}^{\circ}(u)$ of the centraliser $C_\bH(u)$. We record the following concerning these component groups, see \cite[\S1.4]{Sp85} or \cite[\S 8.B]{Bon06}.

\begin{lem}\label{lem:inj-comp-grp}
Assume $\bG$ is connected reductive and $\bL \leqslant \bG$ is a Levi subgroup. If $g \in \bL$ then the natural homomorphism $A_{\bL}(g) \to A_{\bG}(g)$ is injective.
\end{lem}

\begin{proof}
The kernel of this map is $C_{\bL}(g) \cap C_{\bG}^{\circ}(g) = \bL \cap C_{\bG}^{\circ}(g)$. As $\bL$ is a Levi subgroup $\bL = C_{\bG}(Z^{\circ}(\bL))$ so the kernel is $C_{C_{\bG}^{\circ}(g)}(Z^{\circ}(\bL))$ which is connected by \cite[6.4.7(i)]{Spr09}.
\end{proof}

\begin{prop}\label{prop:bumper-para-levi}
Assume $u \in \bG$ is a unipotent element and $\lambda \in \mathcal{D}_u(\bG)$ is a Dynkin cocharacter. If $C_{\bL(\lambda)}(u) := \bL(\lambda) \cap C_{\bG}(u)$ then the following hold:
\begin{enumerate}
	\item $C_{\bG}(u) = C_{\bL(\lambda)}(u)\cdot C_{\bU(\lambda)}(u)$ and $C_{\bU(\lambda)}(u)$ is the unipotent radical of $C_{\bG}(u)$,
	\item the natural map $A_{\bL(\lambda)}(u) \to A_{\bG}(u)$ is an isomorphism.
\end{enumerate}
\end{prop}

\begin{proof}
Assume (i) holds. As $C_{\bU(\lambda)}(u)$ is the unipotent radical it is connected so $C_{\bU(\lambda)}(u) \leqslant C_{\bG}^{\circ}(u)$. This implies that $C_{\bG}^{\circ}(u) = (C_{\bL(\lambda)}(u) \cap C_{\bG}^{\circ}(u))\cdot C_{\bU(\lambda)}(u)$ and (ii) follows immediately. It is easy to check that the property in (i) is preserved under taking isotypic morphisms. The statement then follows from \cite[2.3(iii)]{Pre03}. We leave the details to the reader.
\end{proof}

\begin{rem}
Note (ii) does not follow from \cref{lem:inj-comp-grp} because we do not necessarily have that $u \in \bL(\lambda)$.
\end{rem}

It follows from \cref{eq:Frob-weight-space} that $F\cdot \mathcal{D}_e(\bG) = \mathcal{D}_{F(e)}(\bG)$ for any $e \in \mathcal{N}(\bG)$. Hence, if $e \in \mathcal{N}(\bG)^F$ is $F$-fixed then $F$ preserves the subset $\mathcal{D}_e(\bG) \subseteq \widecheck{X}(\bG)$. It is shown in \cite[3.25]{Tay16} that if $e \in \mathcal{N}(\bG)^F$ is $F$-fixed then there exists an $F$-fixed cocharacter $\lambda \in \mathcal{D}_e(\bG)^F$, which is unique up to $\bG^F$-conjugacy. Combining this with \cref{cor:conj-to-neg-space-rat} we get the following.

\begin{lem}\label{lem:rat-conj-cochar-to-neg}
Assume $u \in \mathcal{U}(\bG)^F$ is a rational unipotent element then the following hold:
\begin{enumerate}
	\item $\mathcal{D}_u(\bG)^F\neq \emptyset$ and the natural action of $C_{\bG}(u)^F$ on $\mathcal{D}_u(\bG)^F$ is transitive,
	\item any cocharacter $\lambda \in \mathcal{D}_u(\bG)^F$ is $\bG^F$-conjugate to $-\lambda$.
\end{enumerate}
\end{lem}

The fact that there exists a character $\lambda \in \mathcal{D}_u(\bG)^F$ when $u \in \mathcal{U}(\bG)^F$ is rational implies that the corresponding subgroups $\bP(\lambda)$, $\bL(\lambda)$ and $\bU(\lambda)$ are all $F$-stable, see \cref{subsection:cocharacters}.

\subsection{Rational Class Representatives}\label{ssec:rationalclassesrep}
Given an element $u \in \bG$ we will denote by $\Cl_{\bG}(u) = \{gug^{-1} \mid g \in \bG\}$ the corresponding $\bG$-conjugacy class of $u$. Recall that if $u \in \bG^F$ is $F$-fixed then the conjugacy class $\Cl_{\bG}(u)$ is $F$-stable. Moreover, the finite group $\bG^F$ acts by conjugation on the fixed points $\Cl_{\bG}(u)^F$. We denote by $\Cl_{\bG}(u)^F/\bG^F$ the resulting set of orbits, which are simply the $\bG^F$-conjugacy classes that are contained in $\Cl_{\bG}(u)^F$.

Recall from \cref{subsec:canonical-para-levi} that $A_{\bG}(u) = C_{\bG}(u)/C_{\bG}^{\circ}(u)$ denotes the component group of the centraliser of $u$. If $g \in \bG$ is an element such that ${}^gu \in \Cl_{\bG}(u)^F$ is $F$-fixed then $g^{-1}F(g) \in C_{\bG}(u)$. As $F(u) = u$ we have $F$ induces an automorphism $F : A_{\bG}(u) \to A_{\bG}(u)$ of the component group. We have an equivalence relation $\sim_F$ on $A_{\bG}(u)$ given by $a \sim_F b$ if and only if $a = x^{-1}bF(x)$ for some $x \in A_{\bG}(u)$. We denote by $H^1(F,A_{\bG}(u))$ the resulting set of equivalence classes, which are the $F$-conjugacy classes of $A_{\bG}(u)$. We note the following well-known parameterisation result, see for example \cite[Prop.~3.21]{DiMibook}.

\begin{prop}\label{prop:rat-class-param}
The map $\Cl_{\bG}(u)^F \to A_{\bG}(u)$ defined by ${}^gu \mapsto g^{-1}F(g)C_{\bG}^{\circ}(u)$ induces a well-defined bijection $\Cl_{\bG}(u)^F/\bG^F \simto H^1(F,A_{\bG}(u))$.
\end{prop}

\section{Fusion of Unipotent Classes}\label{sec:fusion-unip-cls}
If $u \in \mathcal{U}(\bG)$ is a unipotent element then, by definition, $\Cl_{\bG}(u) \subseteq \mathcal{U}(\bG)$ is a unipotent class. Now, assume $\bH \leqslant \bG$ is a closed connected reductive subgroup of $\bG$. We have a natural map $\mathcal{U}(\bH)/\bH \to \mathcal{U}(\bG)/\bG$ given by $\Cl_{\bH}(u) \mapsto \Cl_{\bG}(u)$ which we call the fusion map; we will also denote this by $\mathcal{U}(\bH) \rightsquigarrow_{\bG} \mathcal{U}(\bG)$.

One way to attempt the problem of determining the fusion map is to characterise unipotent classes by their Jordan normal form under some rational representation of the group. When $\bG$ is simple of exceptional type and $\bH$ is a maximal subgroup then this has been done by Lawther \cite{Law09} in all characteristics. Here we are interested in the case where $\bH$ contains a maximal torus of $\bG$ and $p$ is good for $\bG$. Using the following result of Fowler--R\"ohrle we can give an easy combinatorial algorithm for computing the fusion map in terms of weighted Dynkin diagrams.

\begin{thm}[{}{Fowler--R\"ohrle, \cite[Thm.~1.1]{FowRohr08}}]\label{thm:fowler-rohrle}
Assume $p$ is a good prime for $\bG$ and $\bH \leqslant \bG$ is a closed connected reductive subgroup of $\bG$ of maximal rank. If $e \in \mathcal{N}(\bH)$ is a nilpotent element then $\mathcal{A}_e(\bH) = \mathcal{A}_e(\bG,\bH)$. In particular, if $u \in \mathcal{U}(\bH)$ is a unipotent element then $\mathcal{D}_u(\bH) = \mathcal{D}_u(\bG,\bH)$.
\end{thm}

\begin{proof}
The second conclusion follows from \cref{lem:assoc-dynkin-cocharacters}.
\end{proof}

To calculate the fusion map we will only need the inclusion $\mathcal{D}_e(\bH) \subseteq \mathcal{D}_e(\bG)$. We will also exclusively apply this in the case where $\bH = C_{\bG}^{\circ}(s)$ for $s \in \bG$ a semisimple element. We note that, in this case, the result we need is also proved by McNinch--Sommers \cite[Rem.~25]{McnSo03} using a different argument. The algorithm we give here will be used in the later parts of the paper. An analogous algorithm, given in terms of $\lie{sl}_2$-triples, was used by Sommers in \cite{So98}.

\begin{algorithm}[Unipotent Fusion]\label{alg:fusion}
Assume $p$ is a good prime for $\bG$ and $\bH \leqslant \bG$ is a closed connected reductive subgroup containing a maximal torus $\bT \leqslant \bH$ of $\bG$. Let $\mathcal{R}(\bG,\bT) = (X,\Phi,\widecheck{X},\widecheck{\Phi})$ and $\mathcal{R}(\bH,\bT) = (X,\Psi,\widecheck{X},\widecheck{\Psi})$ be the corresponding root data. Fix simple systems of roots $\Pi \subseteq \Psi$ and $\Delta \subseteq \Phi$. For $\alpha \in \Delta$, resp., $\alpha \in \Pi$, we denote by $\widecheck{\omega}_{\alpha} \in \mathbb{Q}\widecheck{\Phi}$, resp., $\widecheck{\pi}_{\alpha} \in \mathbb{Q}\widecheck{\Psi}$, the corresponding fundamental dominant coweight.
\begin{enumerate}[label={\bfseries [Step \arabic*]},leftmargin=2.1cm]
	\item[{\bfseries [Input]}] A weighted Dynkin diagram $d \in \mathcal{D}(\Psi,\Pi)$,
	\item Determine the matrix $A = (a_{\alpha,\beta})_{\alpha \in \Pi,\beta \in \Delta}$ such that $\widecheck{\alpha} = \sum_{\beta \in \Delta} a_{\alpha,\beta}\widecheck{\beta}$.
	\item Let $C = (\langle \widecheck{\alpha},\beta\rangle)_{\alpha,\beta \in \Pi}$ and $D = (\langle \widecheck{\alpha},\beta\rangle)_{\alpha,\beta \in \Delta}$ be the transposed Cartan matrices of $\Psi$ and $\Phi$ respectively. The natural inclusion map $\mathbb{Q}\widecheck{\Psi} \to \mathbb{Q}\widecheck{\Phi}$, with respect to the bases $(\widecheck{\pi}_{\alpha})_{\alpha \in \Pi}$ and $(\widecheck{\omega}_{\alpha})_{\alpha \in \Delta}$, is represented by the matrix $C^{-1}AD$. Define a function $f : \Delta \to \mathbb{Z}$ by setting
	\begin{equation*}
	(f(\alpha))_{\alpha \in \Delta} = (d(\beta))_{\beta \in \Pi}C^{-1}AD
	\end{equation*}
	and extend this linearly to a function $f : \Phi \to \mathbb{Z}$.
	
	\item Let $\Phi^+ \subseteq \Phi$ be the positive roots determined by $\Delta$. Find an element $w \in W_{\bG}(\bT)$, using \cite[Algorithm A]{GePf} for instance, such that
	\begin{equation*}
	\{\alpha \in \Phi^+ \mid {}^{w^{-1}}\alpha \not\in \Phi^+\} = \{\alpha \in \Phi^+ \mid f(\alpha) < 0\}.
	\end{equation*}
	We then have $f^w(\alpha) = f({}^w\alpha) \geqslant 0$ for all $\alpha \in \Phi^+$.
	
	\item[{\bfseries [Return]}] The weighted Dynkin diagram $f^w \in \mathcal{D}(\Phi,\Delta)$.
\end{enumerate}
Identifying unipotent classes with weighted Dynkin diagrams, as in \cref{thm:classification-nil-orbits,thm:springer-morphism}, we have the map $\mathcal{D}(\Psi,\Pi) \to \mathcal{D}(\Phi,\Delta)$ describes the fusion map $\mathcal{U}(\bH) \rightsquigarrow_{\bG} \mathcal{U}(\bG)$.
\end{algorithm}

\begin{rem}
That this algorithm returns a weighted Dynkin diagram is part of the content of the Fowler--R\"ohrle Theorem. This algorithm is easily implemented in \Chevie{} \cite{Chv} as it contains the weighted Dynkin diagrams of all unipotent classes. We note that our transposed Cartan matrices are simply the Cartan matrices in \Chevie{}. For Step 3 of the algorithm we can use \Chevie{}'s command \texttt{ElementWithInversions}.
\end{rem}

\begin{exmp}
Consider the case where $\bG$ is simple of type $\B_2$ and $p \neq 2$. We assume $\Delta = \{\alpha_1,\alpha_2\}$ is a simple system of roots with the corresponding positive roots being $\Phi^+ = (10,01,11,12)$ in Bourbaki notation \cite{Bki}. Hence, $\alpha_2$ is the short simple root. Let $\Psi = \{\pm 10,\pm 12\}$ be a subsystem of type $2\A_1$ with simple system $\Pi = \{10,-12\}$. The coroot corresponding to $-12$ is $-\widecheck{\alpha}_1-\widecheck{\alpha}_2$ so we have
\begin{equation*}
C = \begin{bmatrix}
2 & 0\\
0 & 2
\end{bmatrix}
\qquad
A = \begin{bmatrix}
1 & 0\\
-1 & -1
\end{bmatrix}
\qquad
D = \begin{bmatrix}
2 & -1\\
-2 & 2
\end{bmatrix}
\end{equation*}

Let $d$ be the weighted Dynkin diagram so that $d(\alpha) = 2$ for all $\alpha \in \Delta$, i.e., the weighted Dynkin diagram of the regular class. The function $f$ obtained in Step 2 of the algorithm has values $(2,-2,0,-2)$ on the positive roots taken in the ordering listed above. The element $w = s_2s_1$ is such that $\{\alpha \in \Phi^+ \mid {}^{w^{-1}}\alpha \not\in \Phi^+\} = \{01,12\}$ hence satisfies the condition in Step 3. Now we have $(f^w(\alpha) \mid \alpha \in \Phi^+) = (2,0,2,2)$. In terms of weighted Dynkin diagrams the fusion map is given as follows.
\begin{center}
\begin{tikzpicture}[baseline=-0.1cm,
vertex/.style={inner sep=0pt,minimum size=4.5mm,draw,circle,fill=white,align=center}]
\node[vertex] at (-0.1,0.4) {\small $2$};
\node[vertex] at (1.1,0.4) {\small $2$};

\node[vertex] at (-0.1,-0.4) {\small $0$};
\node[vertex] at (1.1,-0.4) {\small $2$};

\node at (2.2,-0.3) {\rotatebox{20}{$\rightsquigarrow$}};
\node at (2.2,0.2) {\rotatebox{-20}{$\rightsquigarrow$}};

\draw (3.3,0.1) -- (4.4,0.1);
\draw (3.3,-0.1) -- (4.4,-0.1);
\draw (4.05,0) -- (3.65,0.25);
\draw (4.05,0) -- (3.65,-0.25);

\node[vertex] at (3.3,0) {\small $2$};
\node[vertex] at (4.4,0) {\small $0$};
\end{tikzpicture}
\qquad
\text{and}
\qquad
\begin{tikzpicture}[baseline=-0.1cm,
vertex/.style={inner sep=0pt,minimum size=4.5mm,draw,circle,fill=white,align=center}]
\node[vertex] at (-0.1,0) {\small $2$};
\node[vertex] at (1.1,0) {\small $0$};

\node at (2.2,0) {$\rightsquigarrow$};

\draw (3.3,0.1) -- (4.4,0.1);
\draw (3.3,-0.1) -- (4.4,-0.1);
\draw (4.05,0) -- (3.65,0.25);
\draw (4.05,0) -- (3.65,-0.25);

\node[vertex] at (3.3,0) {\small $0$};
\node[vertex] at (4.4,0) {\small $1$};
\end{tikzpicture}
\end{center}

We apply this same procedure with $\Psi = \{\pm 01\}$ and $\Pi = \{01\}$ and $d$ also defined by $d(01) = 2$. We obtain the same weighted Dynkin diagram as above. This implies a regular element of the $2\A_1$ is conjugate, in $\bG$, to a regular element of the Levi subgroup with simple root $01$. If we assume $\bG = \SO_5(K)$, defined as in \cite[1.7.3]{Ge03}, then we see that the element must lie in the unipotent class of $\bG$ whose elements act on the natural module with Jordan blocks of size $(3,1,1)$.
\end{exmp}

Inspecting the tables in \cite{Law09} one sees that for a maximal connected reductive subgroup $\bH \leqslant \bG$ it often happens that the fusion map $\mathcal{U}(\bH) \rightsquigarrow_{\bG} \mathcal{U}(\bG)$ is trivial, in the sense that if $u,v \in \mathcal{U}(\bH)$ are two unipotent elements with $\Cl_{\bG}(u) = \Cl_{\bG}(v)$ then $\Cl_{\bH}(u) = \Cl_{\bH}(v)$. Of course, there can be non-trivial fusion as is shown by the above example. We will need the following which guarantees trivial fusion.

\begin{thm}\label{thm:fusion-unip-classes}
Recall our assumption that $p$ is good for $\bG$ and let $\bH \leqslant \bG$ be a closed connected reductive subgroup of maximal rank. If $u,v \in \mathcal{U}(\bH)$ are unipotent elements such that $u \in \overline{\Cl_{\bH}(v)}$ then $\Cl_{\bG}(u) = \Cl_{\bG}(v)$ if and only if $\Cl_{\bH}(u) = \Cl_{\bH}(v)$.
\end{thm}

The proof of \cref{thm:fusion-unip-classes} is given at the end of this section. We will prove this statement by passing to the nilpotent cone, using a Springer morphism, and then proving the analogous statement for nilpotent orbits. By passing to the nilpotent cone we can use the existence of transverse slices to nilpotent orbits to help prove the statement. The existence of such transverse slices will invoke a restriction on $p$ but this will be removed using some standard reduction arguments. For this restriction on $p$ we will need the following notion due to Herpel.

\begin{definition}[{}{Herpel, \cite{Herp13}}]\label{pa:root-data}
Let $\mathcal{R} = (X,\Phi,\widecheck{X},\widecheck{\Phi})$ be a root datum then we say $p$ is \emph{pretty good} for $\mathcal{R}$ if the following properties hold:
\begin{itemize}
	\item for any subset $\Psi \subseteq \Phi$ we have $X/\mathbb{Z}\Psi$ has no $p$-torsion
	\item for any subset $\widecheck{\Psi} \subseteq \widecheck{\Phi}$ we have $\widecheck{X}/\mathbb{Z}\widecheck{\Psi}$ has no $p$-torsion.
\end{itemize}
We say $p$ is \emph{pretty good} for $\bG$ if it is pretty good for the root datum $\mathcal{R}(\bG,\bT_0)$.
\end{definition}

If $\mathcal{R} = (X,\Phi,\widecheck{X},\widecheck{\Phi})$ is a root datum, as above, then for any closed and symmetric subset $\Psi \subseteq \Phi$ we have a corresponding root datum $\mathcal{R}_{\Psi} = (X,\Psi,\widecheck{X},\widecheck{\Psi})$. We will need the following fact concerning pretty good primes which is obvious from the definition.

\begin{lem}
If $p$ is pretty good for the root datum $\mathcal{R} = (X,\Phi,\widecheck{X},\widecheck{\Phi})$ then for any closed symmetric subset $\Psi \subseteq \Phi$ we have $p$ is pretty good for $\mathcal{R}_{\Psi}$.
\end{lem}

We note that if $p$ is a pretty good prime for $\bG$ then $\bG$ is proximate in the sense of \cref{subsec:springer-morph}, see \cite[2.15]{Tay16}. Moreover, by a result of Herpel \cite[Thm.~1.1]{Herp13}, we have the centraliser $C_{\bG}(x)$ of any element $x \in \lie{g}$ is separable, see \cite[3.10]{Tay16}. With this in mind we have the following existence result for transverse slices, which is shown in \cite[3.27]{Tay16}. 

\begin{prop}[{}{see \cite[3.27]{Tay16}}]\label{lem:sigma-exists}
Assume $p$ is a pretty good prime for $\bG$. If $e \in \mathcal{N}(\bG)$ and $\lambda \in \mathcal{D}_e(\bG)$ then there exists a $\lambda$-invariant complement $\mathfrak{s}\subseteq \lie{g}$ to $[\lie{g},e]$, i.e., $\lambda(\mathbb{G}_m)\cdot\mathfrak{s} = \mathfrak{s}$ and $\lie{g} = \lie{s} \oplus [\lie{g},e]$ as vector spaces. Furthermore, for any such subspace $\lie{s}$ the subset $\Sigma = e + \mathfrak{s}\subseteq \lie{g}$ is a transverse slice to the orbit $\bG\cdot e$.
\end{prop}

\begin{rem}
The separability of centralisers implies that the tangent space at $e$ is $T_e(\bG\cdot e) = [\lie{g},e] = \{[x,e] \mid x \in \lie{g}\}$. Inspecting the proof of \cite[3.27]{Tay16} we see that if $\lie{s} \subseteq \lie{g}$ is a $\lambda$-invariant complement of $[\lie{g},e]$ then we automatically have $\lie{s} \subseteq \bigoplus_{i \leqslant 0} \lie{g}(\lambda,i)$ because $\lie{u}(\lambda,1) \subseteq [\lie{g},e]$.
\end{rem}

If $p \gg 0$ is large then, by the Jacobson--Morozov Theorem, any nilpotent element $e \in \mathcal{N}(\bG)$ is contained in an $\lie{sl}_2$-triple $\{e,h,f\}$. In this case we have $\lie{g} = [\lie{g},e] \oplus \lie{c_g}(f)$ and the subset $e + \lie{c_g}(f) \subseteq \lie{g}$ is a transverse slice to the orbit $\bG\cdot e$ known as the Slodowy slice. Here $\lie{c_g}(f) = \Lie(C_{\bG}(f))$ is the centraliser of $f \in \mathcal{N}(\bG)$ in the Lie algebra. Therefore, \cref{lem:sigma-exists} provides an alternative to the Slodowy slice when $p$ is pretty good for $\bG$. In this direction we will need the following properties of transverse slices which are well-known in the case of the Slodowy slice.

\begin{lem}\label{lem:slod-slice-prop}
Assume $p$ is a pretty good prime for $\bG$. If $e \in \mathcal{N}(\bG)$ and $\lie{s}\subseteq \lie{g}$ is a $\lambda$-invariant complement to $[\lie{g},e]$ then the transverse slice $\Sigma = e + \lie{s}$ satisfies the following properties:
\begin{enumerate}[label=(\alph*)]
	\item $\Sigma \cap \bG\cdot e = \{e\}$,
	\item if $e' \in \mathcal{N}(\bG)$ is such that $\overline{\bG\cdot e'} \cap \Sigma \neq\emptyset$ then $\bG\cdot e' \cap\Sigma\neq\emptyset$.
\end{enumerate}
\end{lem}

\begin{proof}
(a). Clearly we have $T_e(\Sigma \cap \bG\cdot e) \subseteq T_e(\Sigma)\cap T_e(\bG\cdot e) = \mathfrak{s} \cap [\lie{g},e] = \{0\}$ so $\dim(\Sigma \cap \bG\cdot e) = 0$ which implies $\Sigma \cap \bG\cdot e$ is finite. We define a $\mathbb{G}_m$-action on $\lie{g}$ via the homomorphism $\rho : \mathbb{G}_m \to \GL(\lie{g})$ given by
\begin{equation*}
\rho(k)x = k^2(\lambda(k^{-1}) \cdot x)
\end{equation*}
for all $x \in \lie{g}$ and $k \in \mathbb{G}_m$. This action preserves $\Sigma$ and is a contracting $\mathbb{G}_m$-action with unique fixed point $e \in \Sigma^{\mathbb{G}_m}$ since $\lie{s} \subseteq \bigoplus_{i \leqslant 0} \lie{g}(\lambda,i)$. Now, by \cite[Lem.~2.10]{Jan04} the orbit $\bG\cdot e$ is preserved by this action hence so is the intersection $\Sigma\cap\bG\cdot e$. If $x \in \Sigma\cap\bG\cdot e$ then the $\mathbb{G}_m$-orbit of $x$ is irreducible because $\mathbb{G}_m$ is irreducible. As $\Sigma\cap\bG\cdot e$ is finite the $\mathbb{G}_m$-orbit of $x$ must be $\{x\}$ so $x \in \Sigma^{\mathbb{G}_m} = \{e\}$.

(b). As $\Sigma$ is a transverse slice to the orbit $\bG\cdot e$ we have the action map $\bG\times\Sigma \to \lie{g}$ is a smooth morphism. This implies the image $\bG\cdot\Sigma$ is an open subset of $\lie{g}$, see, for instance, the proof of \cite[3.27]{Tay16}. In particular, we have $\mathcal{O} = \bG\cdot\Sigma \cap \overline{\bG\cdot e'}$ is a non-empty open subset of $\overline{\bG\cdot e'}$ (it is non-empty by assumption).  As $\bG$ is irreducible so is $\overline{\bG\cdot e'}$ which implies
\begin{equation*}
\emptyset \neq \mathcal{O} \cap \bG\cdot e' = \bG\cdot \Sigma\cap\bG\cdot e' = \bG \cdot (\Sigma\cap\bG\cdot e')
\end{equation*}
because both $\mathcal{O}$ and $\bG\cdot e'$ are non-empty open subsets of $\overline{\bG\cdot e'}$. This shows that $\Sigma\cap\bG\cdot e'$ must be non-empty.
\end{proof}

Assume $\bT \leqslant \bG$ is a maximal torus and $\mathcal{R}(\bG,\bT) = (X,\Phi,\widecheck{X},\widecheck{\Phi})$. For each root $\alpha \in \Phi$ we have a corresponding $1$-dimensional root subspace $\lie{g}_{\alpha} \subseteq \lie{g}$ and, moreover, we have
\begin{equation*}
\lie{g} = \lie{t} \oplus \bigoplus_{\alpha \in \Phi} \lie{g}_{\alpha}
\end{equation*}
where $\lie{t} := \Lie(\bT)$.

We denote by $\mathcal{C}_{\bT}(\bG)$ the set of closed subgroups $\bH \leqslant \bG$ such that $\bH$ is connected reductive and $\bT \leqslant \bH$. In particular, the elements in  $\mathcal{C}_{\bT}(\bG)$ have maximal rank. If $\bH \in \mathcal{C}_{\bT}(\bG)$ and $\lie{h} := \Lie(\bH)$ then there exists a subset $\Psi \subseteq \Phi$ such that
\begin{equation*}
\lie{h} = \lie{t} \oplus\bigoplus_{\alpha \in \Psi} \lie{g}_{\alpha}.
\end{equation*}
As $p$ is a good prime for $\bG$ we have $\Psi$ is closed and symmetric by \cite[1.14]{DiMibook} and \cite[2.5]{BorTits65}, see also \cite[\S13.1]{MT}. In this case we have the root datum of $\bH$ is given by $\mathcal{R}(\bH,\bT) = (X,\Psi,\widecheck{X},\widecheck{\Psi})$.

\begin{lem}\label{lem:direct-sum-orbits}
If $\lie{n} = \bigoplus_{\alpha \in \Phi\setminus\Psi} \lie{g}_{\alpha}$ then we have $\lie{g} = \lie{h} \oplus \lie{n}$. Moroever, for any element $e \in \lie{h}$ we have $[\lie{g},e] = [\lie{h},e] \oplus [\lie{n},e]$.
\end{lem}

\begin{proof}
The first statement is obvious and so clearly we have $[\lie{g},e] = [\lie{h},e] + [\lie{n},e]$. To see that this sum is direct it suffices to show that $[\lie{n},\lie{h}] \subseteq \lie{n}$ because then we have $[\lie{h},e] \cap [\lie{n},e] \subseteq \lie{h} \cap \lie{n} = \{0\}$ since $e \in \lie{h}$.

Recall that for any two roots $\alpha,\beta \in \Phi$ we have
\begin{equation*}
[\lie{g}_{\alpha},\lie{g}_{\beta}] \subseteq \begin{cases}
\lie{g}_{\alpha+\beta} &\text{if }\alpha+\beta \in \Phi\\
0 &\text{if }\alpha+\beta \not\in \Phi.
\end{cases}
\end{equation*}
Now, if $\alpha \in \Phi\setminus\Psi$ and $\beta \in \Psi$ are roots such that $\alpha + \beta \in \Phi$ is a root then we must have $\alpha + \beta \in \Phi\setminus \Psi$. If this were not the case then $\alpha = (\alpha+\beta) - \beta \in \Psi$ because $\Psi$ is closed and symmetric, a contradiction.
\end{proof}

\begin{prop}\label{prop:fusion-nilp-orb}
Assume $p$ is a pretty good prime for $\bG$ and $\bH \leqslant \bG$ is a closed connected reductive subgroup of maximal rank. If $e,e' \in \mathcal{N}(\bH)$ are nilpotent elements with $e \in \overline{\bH\cdot e'}$ then $\bG \cdot e = \bG\cdot e'$ if and only if $\bH\cdot e = \bH\cdot e'$.
\end{prop}

\begin{proof}
Let $\lambda \in \mathcal{A}_e(\bH)$ be a cocharacter associated to $e$ in $\bH$ then $\lambda \in \mathcal{D}_e(\bH)$ by \cref{lem:assoc-dynkin-cocharacters}. We now choose a $\lambda$-invariant complement $\lie{s} \subseteq \lie{h}$ to $[\lie{h},e]$. The set $\Sigma = e + \lie{s}$ is then a transverse slice to the orbit $\bH\cdot e$, as in \cref{lem:sigma-exists}. We assume $\bT \leqslant \bH$ is a maximal torus containing $\lambda(\mathbb{G}_m)$ then $\bT$ is also a maximal torus of $\bG$ by assumption.

Consider the decomposition $\lie{g} = \lie{h} \oplus \lie{n}$ defined in \cref{lem:direct-sum-orbits}. As $\lambda(\mathbb{G}_m) \leqslant \bT$ each root space $\lie{g}_{\alpha}$ is $\lambda$-invariant, so $\lie{n}$ is $\lambda$-invariant. Moreover $[\lie{n},e]$ is $\lambda$-invariant so we can choose a $\lambda$-invariant complement $\lie{s}' \subseteq \lie{n}$ to $[\lie{n},e]$ by choosing a complement in each degree. By \cref{lem:direct-sum-orbits} we thus have
\begin{equation*}
\lie{g} = \lie{h} \oplus \lie{n} = (\lie{s} \oplus [\lie{h},e]) \oplus (\lie{s}' \oplus [\lie{n},e]) = (\lie{s}\oplus \lie{s}') \oplus [\lie{g},e]
\end{equation*}
so $\tilde{\lie{s}} := \lie{s}\oplus\lie{s}' \subseteq \lie{g}$ is a $\lambda$-invariant complement to $[\lie{g},e]$. By \cref{thm:fowler-rohrle} we have $\lambda \in \mathcal{D}_e(\bH) = \mathcal{A}_e(\bH) \subseteq \mathcal{A}_e(\bG) = \mathcal{D}_e(\bG)$ is associated to $e$ in $\bG$ so $\tilde{\Sigma} = e + \tilde{\lie{s}}$ is a transverse slice to the orbit $\bG\cdot e$ by \cref{lem:sigma-exists}.

Clearly if $\bH\cdot e = \bH\cdot e'$ then $\bG\cdot e = \bG \cdot e'$ so assume conversely that $\bG\cdot e = \bG \cdot e'$. By assumption $e \in \overline{\bH\cdot e'}$ so we have $\overline{\bH\cdot e'} \cap \Sigma \neq \emptyset$ hence $\bH \cdot e' \cap \Sigma \neq \emptyset$ by \cref{lem:slod-slice-prop}. However, again by \cref{lem:slod-slice-prop}, we have
\begin{equation*}
\emptyset \neq \bH \cdot e' \cap \Sigma \subseteq \bG \cdot e' \cap \tilde{\Sigma} = \bG \cdot e \cap \tilde{\Sigma} = \{e\}.
\end{equation*}
This implies that $e \in \bH\cdot e'$ so $\bH\cdot e = \bH\cdot e'$ as desired.
\end{proof}

\begin{proof}[Proof (of \cref{thm:fusion-unip-classes})]
Assume $\varphi : \bG \to \widetilde{\bG}$ is an isotypic morphism, see \cref{subsec:reductive-groups}. Such a morphism restricts to a homeomorphism $\varphi : \mathcal{U}(\bG) \to \mathcal{U}(\widetilde{\bG})$ which is equivariant with respect to the conjugation actions.

If $X \subseteq \bG$ is a subset then we denote by $\widetilde{X} = XZ^{\circ}(\widetilde\bG) \subseteq \widetilde{\bG}$ the product. The assignment $X \mapsto \widetilde{X}$ yields a bijection between maximal tori of $\bG$ and maximal tori of $\widetilde{\bG}$. Moreover, if $\bT \leqslant \bG$ is a fixed maximal torus then this defines a bijection $\mathcal{C}_{\bT}(\bG) \to \mathcal{C}_{\widetilde{\bT}}(\widetilde{\bG})$ by \cite[Prop. 13.5 and Thm. 13.6]{MT}. If $\bH \in \mathcal{C}_{\bT}(\bG)$ then the restriction of $\varphi$ to $\bH$ is also an isotypic morphism $\bH \to \widetilde{\bH}$. With this one readily checks that the statement in \cref{thm:fusion-unip-classes} is true for the pair $(\bG,\bH)$ if and only if it is true for the pair $(\widetilde{\bG},\widetilde{\bH})$.

We first apply this remark to a smooth covering $\widetilde{\bG} \to \bG$, see \cite[1.24]{Tay19}, so that we may assume the derived subgroup of $\bG$ is simply connected. Next, we apply it to a smooth regular embedding, see \cite[7.5]{Tay19}, so that we may assume the centre of $\bG$ is smooth and connected. In particular, we have $\mathcal{R}(\bG,\bT)$ is torsion free, as defined in \cref{subsec:reductive-groups}, so $p$ is a pretty good prime for $\bG$. Applying a Springer morphism, see \cref{thm:springer-morphism}, we see that the statement in \cref{thm:fusion-unip-classes} is equivalent to that in \cref{prop:fusion-nilp-orb} so we are done.
\end{proof}

\part{Central functions on finite reductive groups}

Recall from \cref{sec:basic-setup} that for every $F$-stable subgroup $\bH$ of $\bG$ we denote by $H:=\bH^F$ the finite group of fixed points under $F$.

\section{Generalised Gelfand--Graev Characters}\label{sec:GGGCs}
To each rational unipotent element $u \in \mathcal{U}(\bG)^F$ Kawanaka associated a character $\Gamma_u^G$ known as a generalised Gelfand--Graev Character (GGGC). If $s \in C_G(u)$ is a semisimple element centralising $u$ then $u \in C_G^{\circ}(s)$ and we may consider the GGGC $\Gamma_u^{C_G^{\circ}(s)}$. Our purpose here is to show that the data defining both $\Gamma_u^G$ and $\Gamma_u^{C_G^{\circ}(s)}$ are comparable.

\subsection{Kawanaka Data}\label{sec:kaw-data}
As before we set $\lie{g} = \Lie(\bG)$. The definition of GGGCs requires data which is both global and local, in the sense that some data is the same for all unipotent elements and other data depend on the choice of $u$. We start by introducing the global data, which is a triple $(\phi_{\spr},\kappa,\chi_p)$ consisting of:
\begin{itemize}
	\item a Springer isomorphism $\phi_{\spr} : \mathcal{U}(\bG) \simto \mathcal{N}(\bG)$, see \cref{subsec:springer-morph},
	\item a symmetric bilinear form $\kappa : \lie{g} \times \lie{g} \to K$ which is $\bG$-invariant, with respect to the adjoint action, and is defined over $\mathbb{F}_q$,
	\item and a non-trivial additive character $\chi_p : \mathbb{F}_p^+ \to \Ql^\times$ of the additive group of the finite field $\mathbb{F}_p$.
\end{itemize}
If $V \subseteq \lie{g}$ is a subspace then we will denote by $V^{\perp} = \{x \in \lie{g} \mid \kappa(x,v) = 0\text{ for all }v \in V\}$ the subspace orthogonal to $V$ with respect to $\kappa$. Note that we do not assume that $\kappa$ is non-degenerate. 

Since $\phi_{\spr}$ is $\bG$-equivariant, then for any cocharacter $\lambda \in \widecheck{X}(\bG)$ it restricts to a $\bP(\lambda)$-equivariant isomorphism $\bU(\lambda) \simto \mathfrak{u}(\lambda,1)$, see for example \cite[Rem.~10]{Mcn05}. We will require additional properties on this restriction, given in the following definition.

\begin{definition}\label{def:Kawanaka-datum}
Any triple $\mathcal{K} = (\phi_{\spr},\kappa,\chi_p)$, as above, is called a \emph{Kawanaka datum} for $\bG$ if the following hold:
\begin{enumerate}[leftmargin=1.7cm, label=(KD\arabic*)]
	\item for any cocharacter $\lambda \in \widecheck{X}(\bG)$ the following hold:
	\begin{enumerate}[label=($\lambda$\arabic*)]
		\item $\phi_{\spr}(\bU(\lambda,2)) = \mathfrak{u}(\lambda,2)$,
		\item for any $i \in \{1,2\}$ there exists a non-zero constant $c_i \in K$ such that for any $u,v \in \bU(\lambda,i)$ we have
		\begin{itemize}
			\item $\phi_{\spr}(uv) - \phi_{\spr}(u) - \phi_{\spr}(v) \in \mathfrak{u}(\lambda,i+1)$,
			\item $\phi_{\spr}([u,v]) - c_i[\phi_{\spr}(u),\phi_{\spr}(v)] \in \mathfrak{u}(\lambda,2i+1)$,
		\end{itemize}
	\end{enumerate}

	\item for any maximal torus $\bS \leqslant \bG$ and root $\alpha \in \Phi(\bS)$ we have
	\begin{equation*}
	\lie{g}_{\alpha}^{\perp} = \Lie(\bS) \oplus \bigoplus_{\beta \in \Phi(\bS)\setminus\{-\alpha\}} \lie{g}_{\beta}.
	\end{equation*}
\end{enumerate}
\end{definition}

\begin{rem}\label{rem:kaw-iso}
The conditions ($\lambda$1) and ($\lambda$2) in (KD1) ensure that the restriction of $\phi_{\spr}$ to any unipotent radical $\bU(\lambda)$ is a Kawanaka isomorphism in the sense of \cite[4.1]{Tay16}. Note also that the condition in (KD2) automatically holds whenever $\kappa$ is non-degenerate by the $\bS$-invariance of $\kappa$.
\end{rem}

Let us observe that Kawanaka data exist for proximate groups. Firstly, it follows from \cite[4.6]{Tay16}, and our assumption that $\bG$ is proximate, that there exists a Springer isomorphism $\phi_{\spr}$ satisfying (KD1). Note that whilst it is only required that $\phi_{\spr}(\bU(\lambda,2)) \subseteq \lie{u}(\lambda,2)$ in \cite[4.1(K1)]{Tay16} equality must hold because $\phi_{\spr}$ is an isomorphism and $\bU(\lambda,2)$ and $\lie{u}(\lambda,2)$ are closed subsets of the same dimension.

Now, in \cite[5.6]{Tay16} it is shown that there exists a form $\kappa$ satisfying (KD2) when $\bS = \bT_0$. However, if $\bS \neq \bT_0$ then there exists a $g \in \bG$ such that $\bS = {}^g\bT_0$. We have $\Phi(\bS) = {}^g\Phi(\bT_0)$ and for any $\alpha \in \Phi(\bT_0)$ we have $\lie{g}_{{}^g\alpha} = g\cdot\lie{g}_{\alpha}$. By the $\bG$-invariance of $\kappa$ we have $(\lie{g}_{{}^g\alpha})^{\perp} = (g\cdot\lie{g}_{\alpha})^{\perp} = g\cdot(\lie{g}_{\alpha}^{\perp})$ so (KD2) holds. To recapitulate, we have the following.

\begin{lem}
If $\bG$ is proximate then there exists a Kawanaka datum $\mathcal{K} = (\phi_{\spr},\kappa,\chi_p)$ for $\bG$.
\end{lem}

\begin{rem}
Note we obtain a character $\chi_q = \chi_p\circ \Tr_{\mathbb{F}_q/\mathbb{F}_p} : \mathbb{F}_q^+ \to \Ql^\times$ where $\Tr_{\mathbb{F}_q/\mathbb{F}_p}$ is the field trace.
\end{rem}

\begin{assumption}
From now until the end of this section we assume that the group $\bG$ is proximate. We assume $\mathcal{K} = (\phi_{\spr},\kappa,\chi_p)$ is a fixed Kawanaka datum for $\bG$.
\end{assumption}

\subsection{Definition of GGGCs}\label{ssec:gggcdef}
We now briefly recall the construction of GGGCs; this construction is described in \cite[\S3.1]{Kaw86} and \cite[\S5]{Tay16}. Let $u \in \mathcal{U}(\bG)^F$ be a rational unipotent element. We denote by $\eta_u^G : \mathcal{U}(\bG)^F \to \Ql^\times$ the map defined by
\begin{equation*}
\eta_u^G(v) = \chi_q(\kappa(\phi_{\spr}(u),\phi_{\spr}(v))).
\end{equation*}
It follows from the $\bG$-equivariance of $\kappa$ and $\phi_\spr$ that for any $x \in G$ we have
\begin{equation}\label{eq:conj-eta-char}
{}^x\eta_{u}^G = \eta_{xux^{-1}}^G.
\end{equation}
If $\lambda \in \mathcal{D}_u(\bG)^F$ is a Dynkin cocharacter associated to $u$ then, by definition, $\phi_{\spr}(u) \in \lie{g}(\lambda,2)$ so
\begin{equation}\label{eq:ker-eta}
U(\lambda,-3) \subseteq \{v \in \mathcal{U}(\bG)^F \mid \eta_u^G(v) = 1\}
\end{equation}
and $\eta_u^G$ restricts to a linear character $U_G(\lambda,-2) \to \Ql^\times$, see \cite[5.8]{Tay16}.

\begin{lem}\label{lem:is-subgrp}
Let us consider $\lie{g}(\lambda,-1)$ as an algebraic group under addition. We have a surjective homomorphism of algebraic groups $\gamma : \bU(\lambda,-1) \to \lie{g}(\lambda,-1)$ defined by $\gamma(v) = \pi(\phi_{\spr}(v))$, where $\pi : \lie{g} \to \lie{g}(\lambda,-1)$ is the natural projection map. The kernel of $\gamma$ is $\bU(\lambda,-2)$. Hence, for any subspace $\lie{m} \subseteq \lie{g}(\lambda,-1)$ there exists a closed subgroup $\bU(\lambda,\lie{m}) := \gamma^{-1}(\lie{m})$ containing $\bU(\lambda,-2)$ as a normal subgroup. Moreover, we have $F(\bU(\lambda,\lie{m})) = \bU(F\cdot\lambda,F(\lie{m}))$ and for any $a \in \bG$ we have ${}^a\bU(\lambda,\lie{m}) = \bU({}^a\lambda,a\cdot\lie{m})$.
\end{lem}

\begin{proof}
That $\gamma$ is a morphism of varieties is clear and as $\phi_{\spr}$ is an isomorphism $\bU(\lambda,-1) \to \lie{u}(\lambda,-1)$ it is clearly surjective. Note that for any $i \in \mathbb{Z}$ we have $\bU(\lambda,-i) = \bU(-\lambda,i)$ and $\lie{u}(\lambda,-i) = \lie{u}(-\lambda,i)$. For any $u,v \in \bU(\lambda,-1)$ we have
\begin{equation*}
\phi_{\spr}(uv) + \lie{u}(\lambda,-2) = \phi_{\spr}(u)+\phi_{\spr}(v)+ \lie{u}(\lambda,-2),
\end{equation*}
which follows from ($-\lambda$2), so $\gamma$ is a homomorphism. Its kernel is $\bU(\lambda,-2)$ by ($-\lambda$1). The remaining statements are clear.
\end{proof}

Now consider the bilinear form $\omega : \lie{g}(\lambda,-1) \times \lie{g}(\lambda,-1) \to K$ defined by
\begin{equation}\label{eq:defomega}
\omega(x,y) = \kappa(\phi_{\spr}(u),[x,y]).
\end{equation}
Under our assumptions $\omega$ is a non-degenerate alternating bilinear form on $\lie{g}(\lambda,-1)$, see \cite[5.9]{Tay16}. For a subspace $V \subseteq \lie{g}(\lambda,-1)$ we denote by
\begin{equation*}
V^{\ddagger} = \{x \in \lie{g}(\lambda,-1) \mid \omega(x,y) = 0\text{ for all }y \in V\}
\end{equation*}
the perpendicular space of $V$ with respect to $\omega$.

Now let $\lie{m} \subseteq \lie{g}(\lambda,-1)$ be an $F$-stable Lagrangian subspace, by which we mean that $\lie{m} = \lie{m}^{\ddagger}$. As in \cref{lem:is-subgrp} we have a corresponding $F$-stable closed subgroup $\bU(\lambda,\lie{m}) \leqslant \bU(\lambda,-1)$ and we set $U(\lambda,\lie{m}) := \bU(\lambda,\lie{m})^F$. By Chevalley's commutator formula we have $[U(\lambda,-1),U(\lambda,-2)] \subseteq U(\lambda,-3)$ so it follows from \cref{eq:ker-eta} that the character $\eta_u^G$ of $U(\lambda,-2)$ is invariant under $U(\lambda,-1)$. Moreover, it can be extended to a character $\tilde{\eta}_u^G$ of $U(\lambda,\lie{m})$, see \cite[5.13(ii)]{Tay16} for instance.

We will denote by $\zeta_{u,\lambda}^G := \Ind_{U(\lambda,\lie{m})}^{U(\lambda,-1)}(\tilde{\eta}_u^G)$ the induction of the extension. It is an irreducible character of $U(\lambda,-1)$ that satisfies
\begin{equation}\label{eq:ind-char-values}
\Ind_{U(\lambda,-2)}^{U(\lambda,-1)}(\eta_u^G) = q^{\dim \lie{g}(\lambda,-1)/2}\zeta_{u,\lambda}^G,
\end{equation}
see the proof of \cite[5.15]{Tay16}. In particular, the character $\zeta_{u,\lambda}^G$ depends only on $\eta_u^G$ and not on the choice of extension $\tilde{\eta}_u^G$. It follows from \cref{eq:ind-char-values,eq:conj-eta-char} that for any $x \in G$ we have
\begin{equation}\label{eq:conj-zeta-char}
{}^x\zeta_{u,\lambda}^G = \zeta_{{}^xu,{}^x\lambda}^G
\end{equation}

\begin{definition}[Kawanaka]\label{def:GGGC}
Fix a Kawanaka datum $\mathcal{K} := (\phi_{\spr},\kappa,\chi_p)$ for $\bG$. To the pair $(u,\lambda)$ where:
	\begin{itemize}
	\item $u$ is an $F$-stable unipotent element of $\bG$,
	\item $\lambda \in \mathcal{D}_u(\bG)^F$ is an $F$-stable Dynkin cocharacter associated to $u$,
\end{itemize}
we associate the \emph{Generalised Gelfand--Graev Character}
\begin{equation*}
\Gamma_u^G = \Ind_{U(\lambda,-1)}^G(\zeta_{u,\lambda}^G).
\end{equation*}
\end{definition}

By \cref{lem:rat-conj-cochar-to-neg,eq:conj-zeta-char} we have $\Gamma_u^G$ does not depend upon the choice of cocharacter $\lambda \in \mathcal{D}_u(\bG)^F$. Moreover, $\Gamma_u^G = \Gamma_{{}^xu}^G$ for any $x \in G$. We must now relate the definition given in \cref{def:GGGC} with that occurring previously in the literature, see \cite{Kaw86,Lu92,Tay16}. By \cref{cor:conj-to-neg-space-rat} there exists an element $g \in \bG^F$ such that ${}^g\lambda = -\lambda$ so clearly ${}^gU(\lambda,i) = U({}^g\lambda,i) = U(-\lambda,i) = U(\lambda,-i)$ for any $0 \neq i \in \mathbb{Z}$. Hence, conjugating we get that
\begin{equation*}
\Gamma_u^G = {}^g\Gamma_u^G = \Ind_{U(\lambda,1)}^G(\zeta_{{}^gu,-\lambda}^G).
\end{equation*}
Here $\zeta_{{}^gu,-\lambda}^G$ is an irreducible character obtained, as in \cref{eq:ind-char-values}, from the linear character of $U(\lambda,2)$ given by $v \mapsto \chi_q(\kappa(g\cdot \phi_{\spr}(u), \phi_{\spr}(v)))$.

This setting now matches, almost exactly, that used in \cite{Kaw86,Tay16}, except there the element $g\cdot \phi_{\spr}(u)$ should be replaced by $\phi_{\spr}(u)^{\dag}$ where ${}^{\dag} : \lie{g} \to \lie{g}$ is an $\mathbb{F}_q$-opposition automorphism of the Lie algebra. This notion depends upon the choice of a Chevalley basis for $\lie{g}$. Now $\phi_{\spr}(u)$ and $\phi_{\spr}(u)^{\dag}$ are known to be $\bG$-conjugate, see \cite[Prop.~5.3]{Tay16}, but not necessarily $\bG^F$-conjugate. This means that $\Gamma_u^G$, as defined in \cref{def:GGGC}, is a generalised Gelfand--Graev character as defined in \cite{Kaw86,Tay16}. However, they may differ up to a permutation of the rational classes inside a given $\bG$-conjugacy class.

The arguments in \cite{Tay16} can be carried out verbatim with the element $g\cdot \phi_{\spr}(u)$ replacing the element $-\phi_{\spr}(u)^{\dag}$. The key point here is that if $\lie{s} \subseteq \lie{g}$ is a $(-\lambda)$-invariant complement subspace as in \cref{lem:sigma-exists} then $\Sigma = g\cdot \phi_{\spr}(u) + \lie{s}$ is a transverse slice to the nilpotent orbit $\bG\cdot \phi_{\spr}(u)$. For instance, the proof of the key result \cite[Prop.~6.9]{Tay16} relies only on the fact that $-\phi_{\spr}(u)^{\dag} \in \lie{g}(\lambda,-2)$.

\begin{rem}
The definition used here seems to be the correct one. For instance, it matches that used in \cite{Lu92}. Indeed, in \cite{Lu92} Lusztig works with an $\lie{sl}_2$-triple $\{e,h,f\}$. His linear character on $U(\lambda,2)$ is defined as above but using the element $-f$. As is easily checked, in this case the elements $e$ and $-f$ are in the same $\SL_2(q)$-orbit.
\end{rem}

\subsection{The Structure of the Radical}\label{subsec:strut-rad}
To define Kawanaka characters we will need to consider the structure of the unipotent radical $U(\lambda,-1)$. The kernel $\Ker(\eta_u^G) \lhd U(\lambda,-1)$ of the linear character $\eta_u^G$ is a normal subgroup of $U(\lambda,-1)$ as $\eta_u^G$ is $U(\lambda,-1)$-invariant. Our interest is in the structure of the quotient group $Q = U(\lambda,-1)/\Ker(\eta_u^G)$. The following are immediate from Chevalley's commutator formula:
\begin{itemize}
	\item $U(\lambda,-2)/U(\lambda,-3)$ is an elementary abelian $p$-group,
	\item $[U(\lambda,-1),U(\lambda,-1)] \subseteq U(\lambda,-2)$ and $[U(\lambda,-2),U(\lambda,-1)] \subseteq U(\lambda,-3)$.
\end{itemize}
Since $U(\lambda,-3) \leqslant \Ker(\eta_u^G)$ this implies $[Q,Q] \leqslant U(\lambda,-2)/\Ker(\eta_u^G) \leqslant Z(Q)$. Moreover, as $U(\lambda,-2)/\Ker(\eta_u^G)$ necessarily has a faithful irreducible character, it must be a cyclic group. Hence, it must be cyclic of order $p$ as $U(\lambda,-2)/\Ker(\eta_u^G)$ is a quotient of the elementary abelian group $U(\lambda,-2)/U(\lambda,-3)$. The following was observed by Kawanaka in \cite[\S2.2]{Ka87}.

\begin{lem}\label{lem:rad-extraspecial}
Assume $p \neq 2$. The group $Q = U(\lambda,-1)/\Ker(\eta_u^G)$ is either cyclic of order $p$ or is an extraspecial $p$-group of exponent $p$.
\end{lem}

\begin{proof}
It is shown in \cite[Satz 1.4.1]{GeThesis} that we have $Z(Q) = U(\lambda,-2)/\Ker(\eta_u^G)$. Therefore $Q/Z(Q) \simeq  U(\lambda,-1)/U(\lambda,-2)$ is elementary abelian and $Z(Q)$ is cyclic of order $p$. This shows that $Q$ is either cyclic or an extraspecial group. Now consider the question of the exponent. If $x,y \in Q$ then by \cite[8.6]{Asch} we have $(xy)^p = x^py^p[x,y]^{p(p-1)/2} = x^py^p$ because $[x,y] \in Z(Q)$ and $Z(Q)$ has exponent $p$. Here we use our assumption that $p$ is odd. As any non-trivial element of $U(\lambda,-1)$ (hence of its quotient $Q$) can be written as a product of elements of order $p$, namely from the root groups, it follows that $Q$ has exponent $p$.
\end{proof}

Let us continue to assume that $p \neq 2$. Recall that for any extraspecial $p$-group one can associate a symplectic form $\omega_Q : \overline{Q} \times \overline{Q}  \to Z(Q)$ on $\overline{Q}  = Q/Z(Q)$ given by taking commutators, see \cite[23.10]{Asch} for more details. Let $\Aut^*(Q) \leqslant \Aut(Q)$ be the kernel of the restriction map $\Aut(Q) \to \Aut(Z(Q))$. It is clear that this group acts on $\overline{Q} $ and preserves the form $\omega_Q$. 
By \cite[Thm.~1]{Win72} this induces an isomorphism $\Out^*(Q) := \Aut^*(Q)/\Inn(Q) \simto \Sp(\overline{Q}  \mid \omega_Q)$ onto the symplectic group defined by the form. We record here the following to be used below.

\begin{lem}\label{lem:extra-iso-lift}
Assume $p\neq 2$ and $Q_1$ and $Q_2$ are extraspecial $p$-groups of exponent $p$ of the same finite cardinality. Let $\overline{Q}_i = Q_i/Z(Q_i)$. 
We assume that we are given:
\begin{itemize}
\item a finite $p'$-group $B$ and homomorphisms $\varphi_i : B \to \Aut^*(Q_i)$,
\item a $B$-equivariant isomorphism $\overline{\pi} : \overline{Q}_1 \simto \overline{Q}_2$,
\item and an isomorphism $\gamma : Z(Q_1) \simto Z(Q_2)$ such that $\omega_{Q_2}(\overline{\pi}(x),\overline{\pi}(y)) = \gamma(\omega_{Q_1}(x,y))$ for all $x,y \in \overline{Q}_1$.
\end{itemize}
Then there exists a $B$-equivariant isomorphism $\pi : Q_1\to Q_2$ lifting $\overline{\pi}$. Such an isomorphism is unique up to composing with a $B$-invariant element of $\Inn(Q_1)$.
%
%
%
\end{lem}

\begin{proof}
There is only one isomorphism class of extraspecial groups of exponent $p$ of a given cardinality, so there exists some isomorphism $\phi : Q_1 \simto Q_2$ which necessarily factors through an isomorphism $\overline{\phi} : \overline{Q}_1 \simto \overline{Q}_2$ satisfying $\omega_{Q_2}(\overline{\phi}(x),\overline{\phi}(y)) = \phi(\omega_{Q_1}(x,y))$ for all $x,y \in \overline{Q}_1$. It is shown by Winter \cite[Thm.~1]{Win72} that the restriction map $\Aut(Q_1) \to \Aut(Z(Q_1))$ is surjective and the kernel $\Aut^*(Q_1)$ has a complement in $\Aut(Q_1)$. It follows that there exists some automorphism ${\alpha} \in \Aut(Q_1)$ such that $\pi = \alpha \phi$ is an isomorphism $Q_1 \simto Q_2$ which lifts $\overline{\pi}$. The choice of $\alpha$ is unique up to composing with an automorphism trivial on $Z(Q_1)$ and $Q_1/Z(Q_1)$, hence an element of $\Inn(Q_1)$.

Let $p_i : \Aut^*(Q_i) \twoheadrightarrow \Aut(\overline{Q}_i)$ be the natural surjection, whose kernel $\Inn(Q_i) \cong \overline{Q}_i$ is an abelian $p$-group. Note that we have $p_2(\pi\alpha\pi^{-1}) = \overline{\pi} p_1(\alpha) \overline{\pi}^{-1}$ for all $\alpha \in \Aut^*(Q_1)$. Let $\psi_i = p_i\circ \varphi_i : B \to \Aut(\overline{Q}_i)$ be the composition. The $B$-equivariance of $\overline{\pi}$ implies that $\psi_2(a) = \overline{\pi}\psi_1(a)\overline{\pi}^{-1}$ for any $a \in B$. In turn, this implies that $q(a) = \pi\varphi_1(a)\pi^{-1}\varphi_2(a)^{-1} \in \Inn(Q_2)$ for all $a \in B$. The map $q : B \to \Inn(Q_2)$ is a $1$-cocycle in the sense that $q(ab) = q(a)({}^{\varphi_2(a)}q(b))$ for all $a,b \in B$. However as $B$ and $\Inn(Q_2)$ have coprime orders the cohomology group $H^1(B,\Inn(Q_2))$ is trivial \cite[Thm.~9.42]{Rot} so $q$ is a $1$-coboundary. This means there exists an $\iota \in \Inn(Q_2)$ such that $q(a) = \iota \varphi_2(a) \iota^{-1}\varphi_2(a)^{-1}$ for all $a \in B$. Hence $\pi' := \iota^{-1}\pi$ is $B$-equivariant.
\end{proof}

Later we will use an explicit version of this form, defined as follows. By definition of $\eta_{u}^G$ the map $\gamma : U(\lambda, -2) \to \mathbb{F}_p$ defined by
\begin{equation*}
\gamma(v) = \Tr_{\mathbb{F}_q/\mathbb{F}_p}\big(\kappa(\phi_{\spr}(u),\phi_{\spr}(v))\big)
\end{equation*}
induces an isomorphism between $Z(Q) = U(\lambda,-2)/\Ker(\eta_{u}^G)$ and $\mathbb{F}_p^+$. Using this identification we get a symplectic form $\gamma\circ\omega_Q$ on $\overline{Q}$ with values in $\mathbb{F}_p$ given by
\begin{equation}\label{eq:omegaQ}
\gamma(\omega_Q(xZ(Q),yZ(Q))) = \gamma([x,y]) = \Tr_{\mathbb{F}_q/\mathbb{F}_p}\big(\kappa(\phi_{\spr}(u),\phi_{\spr}([x,y])))\big).
\end{equation}

\subsection{Isotypic Morphisms}

Now let $\iota : \bG \to \widetilde{\bG}$ be an isotypic morphism defined over $\mathbb{F}_q$, see \cref{subsec:reductive-groups}. We will say that $\iota$ is \emph{separable} if the restriction $\iota : \bG_{\der} \to \widetilde{\bG}_{\der}$ of $\iota$ to the derived subgroup is a separable isogeny. As we assume $\bG$ is proximate this implies $\widetilde{\bG}$ is proximate. Indeed, if $\pi : \bG_{\simc} \to \bG_{\der}$ is a simply connected covering then $\iota\circ\pi : \bG_{\simc} \to \widetilde{\bG}_{\der}$ is a simply connected covering. By assumption both $\pi$ and $\iota$ are separable so $\iota\circ\pi$ is separable.

As the restriction $\iota : \bG_{\der} \to \widetilde{\bG}_{\der}$ is a seperable isogeny and $\mathcal{U}(\bG) = \mathcal{U}(\bG_{\der})$ we get that the restriction $\iota : \mathcal{U}(\bG) \to \mathcal{U}(\widetilde{\bG})$ is an isomorphism of varieties, see \cite[3.3]{Tay16}. Similarly we have the restriction of the differential $\delta := \mathrm{d}_1\iota : \mathcal{N}(\bG) \to \mathcal{N}(\widetilde{\bG})$ is an isomorphism of varieties.

Recall that $\mathcal{K} = (\phi_{\spr},\kappa,\chi_p)$ is a fixed Kawanaka datum of $\bG$. As $\widetilde{\bG}$ is proximate there exists a Kawanaka datum $\widetilde{\mathcal{K}} = (\widetilde{\phi}_{\spr},\widetilde{\kappa},\chi_p)$ for $\widetilde{\bG}$. We say that $\mathcal{K}$ and $\widetilde{\mathcal{K}}$ are \emph{$\iota$-compatible} if the following hold:
\begin{itemize}
	\item $\widetilde{\phi}_{\spr} \circ \iota  = \delta \circ \phi_{\spr}$
	\item $\kappa(x,y) = \widetilde{\kappa}(\delta(x),\delta(y))$ for any $x,y \in \mathcal{N}(\bG)$.
\end{itemize}
With this notion in place we may prove the following, which unifies the discussion in \cite[14.12, 14.13]{Tay16} and provides more details.

\begin{lem}\label{lem:induction-GGGC}
Assume $\iota : \bG \to \widetilde{\bG}$ is a separable isotypic morphism defined over $\mathbb{F}_q$ then
\begin{enumerate}
	\item there exist $\iota$-compatible Kawanaka data $\mathcal{K} = (\phi_{\spr},\kappa,\chi_p)$ and $\widetilde{\mathcal{K}} = (\widetilde{\phi}_{\spr},\widetilde{\kappa},\chi_p)$.
\end{enumerate}
Moreover, if $\mathcal{K}$ and $\widetilde{\mathcal{K}}$ are $\iota$-compatible Kawanaka data then
\begin{enumerate}[resume]
	\item for each $u \in \mathcal{U}(\bG)^F$ we have $\iota_*(\Gamma_u^G) = \Gamma_{\iota(u)}^{\widetilde{G}}$ where $\Gamma_u^G$ is defined with respect to $\mathcal{K}$ and $\Gamma_{\iota(u)}^{\widetilde{G}}$ is defined with respect to $\widetilde{\mathcal{K}}$.
\end{enumerate}
\end{lem}

\begin{proof}
(i). From the preceding discussion it is clear that $\widetilde{\phi}_{\spr} = \delta\circ\phi_{\spr}\circ\iota^{-1}$ is a Springer isomorphism. Let $\pi : \widetilde{\bG} \to \bG_{\ad}$ be an adjoint quotient of $\widetilde{\bG}$ then, because $\iota$ is separable, we have $\pi\circ\iota : \bG \to \bG_{\ad}$ is also an adjoint quotient of $\bG$, see \cite[V, Prop.~22.15]{Bor91}. As in \cite[5.6]{Tay16} we may assume that $\kappa(x,y) = \kappa_{\ad}(\mathrm{d}(\pi\circ\iota)(x),\mathrm{d}(\pi\circ\iota)(y))$ and $\widetilde{\kappa}(x,y) = \kappa_{\ad}(\mathrm{d}\pi(x),\mathrm{d}\pi(y))$ where $\kappa_{\ad}$ is a fixed bilinear form on $\Lie(\bG_{\ad})$. Hence, the resulting Kawanaka data are $\iota$-compatible.

(ii). If $\lambda \in \widecheck{X}(\bG)$ is a cocharacter then for any $i \in \mathbb{Z}$ we have $\widetilde{\lie{g}}(\iota_*(\lambda),i) = \delta(\lie{g}(\lambda,i))$. In particular, if $e \in \mathcal{N}(\bG)$ is nilpotent then $\iota_*(\mathcal{D}_e(\bG)) = \mathcal{D}_{\delta(e)}(\widetilde{\bG})$. Hence, for any unipotent element $u \in \mathcal{U}(\bG)$ we have
\begin{equation*}
\iota_*(\mathcal{D}_u(\bG)) = \iota_*(\mathcal{D}_{\phi_{\spr}(u)}(\bG)) = \mathcal{D}_{\delta(\phi_{\spr}(u))}(\widetilde{\bG}) = \mathcal{D}_{\widetilde{\phi}_{\spr}(\iota(u)))}(\widetilde{\bG}) = \mathcal{D}_{\iota(u)}(\widetilde{\bG}).
\end{equation*}
In other words, if we fix a Dynkin cocharacter $\lambda\in \mathcal{D}_u(\bG)$ then $\iota_*(\lambda) \in \mathcal{D}_{\iota(u)}(\widetilde{\bG})$ is a Dynkin cocharacter for $\iota(u)$.

Let $j : U_G(\lambda,-2) \to G$ be the natural inclusion map then $\Ind_{U_G(\lambda,-2)}^G = j_*$. We have $\iota_*\circ j_* = (\iota\circ j)_* = (\lambda\circ\gamma)_* = \lambda_*\circ\gamma_*$ where $\lambda : U_{\widetilde{G}}(\iota_*(\lambda),-2) \to \widetilde{G}$ is the natural inclusion and $\gamma : U_G(\lambda,-2) \to U_{\widetilde{G}}(\iota_*(\lambda),-2)$ is the isomorphism given by the restriction of $\iota$. Using the compatibility between $\mathcal{K}$ and $\widetilde{\mathcal{K}}$ we have for all $v \in U_G(\lambda,-2)$ that
\begin{align*}
\eta_{u}^G(v) &= \chi_q\big(\kappa(\phi_{\spr}(u),\phi_{\spr}(v))\big)\\
&= \chi_q\big(\tilde{\kappa}(\delta(\phi_{\spr}(u)),\delta(\phi_{\spr}(v)))\big)\\
&= \chi_q\big(\tilde{\kappa}( \tilde{\phi}_{\spr}(\iota(u)),\tilde{\phi}_{\spr}(\iota(v)))\big)\\
&= \eta_{\iota(u)}^{\widetilde{G}}(\iota(v)).
\end{align*}
This shows that $\gamma_*(\eta_{u}^G) = \eta_{\iota(u)}^{\widetilde{G}}$. As $\lambda_* = \Ind_{U_{\widetilde{G}}(\iota_*(\lambda),-2)}^{\widetilde{G}}$ it follows from \cref{eq:ind-char-values} that $\iota_*(\Gamma_u^G) = \Gamma_{\iota(u)}^{\widetilde{G}}$. That $\dim(\lie{g}(\lambda,-1)) = \dim(\widetilde{\lie{g}}(\iota_*(\lambda),-1))$ is obvious.
\end{proof}

\subsection{Centralisers of Semisimple Elements}\label{sec:cent-s/s-elt}
Now assume, as at the start of this section, that $u \in \mathcal{U}(\bG)^F$ and $s \in C_{\bG}(u)^F$ is a semisimple element. To ease notation we set $\bG_s = C_{\bG}^{\circ}(s)$, $G_s = \bG_s^F$, and $\lie{g}_s = \Lie(\bG_s)$. As pointed out in \cref{lem:subs-prox-are-prox} the group $\bG_s$ is proximate because $\bG$ is proximate and $p$ is good for $\bG$. Hence, we can construct the GGGC of $G_s$ associated to $u$. To relate $\Gamma_u^G$ and $\Gamma_u^{G_s}$ we will need the following.

\begin{prop}\label{prop:res-kawanaka-datum}
Assume $\mathcal{K} = (\phi_{\spr},\kappa,\chi_p)$ is a Kawanaka datum for $\bG$ then the restricted triple $\mathcal{K}|_{\bG_s} = (\phi_{\spr}|_{\mathcal{U}(\bG_s)},\kappa|_{\lie{g}_s\times \lie{g}_s},\chi_p)$ is a Kawanaka datum for $\bG_s$.
\end{prop}

\begin{proof}
As $s$ is semisimple we have $\lie{g}_s = \{x \in \lie{g} \mid s \cdot x = x\}$ is the infinitesimal centraliser of $s$, see \cite[10.1]{BorTits65}. Hence we have $\mathcal{N}(\bG_s) = \mathcal{N}(\bG) \cap \lie{g}_s$ and $\mathcal{U}(\bG_s) = \mathcal{U}(\bG) \cap \bG_s$. By the $\bG$-equivariance of $\phi_{\spr}$ we have $\phi_{\spr}(\mathcal{U}(\bG_s)) = \mathcal{N}(\bG_s)$, which shows that $\phi_{\spr}|_{\mathcal{U}(\bG_s)}$ is a Springer isomorphism for $\bG_s$.

Let $\lambda \in \widecheck{X}(\bG_s)$ be a cocharacter of $\bG_s$. It is naturally a cocharacter of $\bG$ and we have
 two parabolic subgroups $\bP_{\bG_s}(\lambda)\leqslant \bP_{\bG}(\lambda)$ with unipotent radicals $\bU_{\bG_s}(\lambda) \leqslant \bU_{\bG}(\lambda)$. We need to check that the properties (KD1) and (KD2) of \cref{def:Kawanaka-datum} hold for $\bU_{\bG_s}(\lambda)$, given that they hold for $\bU_{\bG}(\lambda)$.

Since $\lambda(\mathbb{G}_m) \subset \bG_s$ we have that ${}^s\lambda = \lambda$. Therefore $s$ acts on $\lie{g}(\lambda,i)$ for any $i \in \mathbb{Z}$ and $\lie{g}_s(\lambda,i) = \lie{g}(\lambda,i) \cap \lie{g}_s$. In particular, if $i > 0$ then $\lie{u}_{\lie{g}_s}(\lambda,i) = \lie{u_g}(\lambda,i) \cap \lie{g}_s$ and $\bU_{\bG_s}(\lambda,i) = \bU_{\bG}(\lambda,i) \cap \bG_s$. Now ($\lambda$1) in (KD1) for $\bU_{\bG}(\lambda)$ states that $\phi_{\spr}(\bU_{\bG}(\lambda,2)) \subseteq \lie{u_g}(\lambda,2)$. As $\phi_{\spr}(\mathcal{U}(\bG_s)) \subseteq \lie{g}_s$ we get that
\begin{equation*}
\phi_{\spr}(\bU_{\bG_s}(\lambda,2)) = \phi_{\spr}(\bU_{\bG}(\lambda,2) \cap \bG_s) \subseteq  \lie{u_g}(\lambda,2) \cap \lie{g}_s = \lie{u}_{\lie{g}_s}(\lambda,2).
\end{equation*}
After the comment following \cref{rem:kaw-iso} we see that ($\lambda$1) holds for $\bU_{\bG_s}(\lambda)$. The same argument shows that ($\lambda$2) holds for $\bU_{\bG_s}(\lambda)$ so $\mathcal{K}|_{\bG_s}$ satisfies (KD1).

Assume $\bS \leqslant \bG_s$ is a maximal torus of $\bG_s$ then $\bS \leqslant \bG$ is also a maximal torus of $\bG$ because $\bG_s$ has maximal rank. If $\Phi_{\bG_s}(\bS)$ and $\Phi_{\bG}(\bS)$ are the roots of $\bG_s$ and $\bG$ respectively, defined with respect to $\bS$, then it is clear that
\begin{equation*}
\Phi_{\bG_s}(\bS) = \{\alpha \in \Phi_{\bG}(\bS) \mid \lie{g}_{\alpha} \subseteq \lie{g}_s\}.
\end{equation*}
Moreover, as root spaces are $1$-dimensional we have $(\lie{g}_s)_{\alpha} = \lie{g}_{\alpha}$ for all $\alpha \in \Phi_{\bG_s}(\bS)$. This implies that $(\lie{g}_s)_{\alpha}^{\perp} = \lie{g}_{\alpha}^{\perp} \cap \lie{g}_s$ so we see immediately that (KD2) holds for $\mathcal{K}|_{\bG_s}$ as it holds for $\mathcal{K}$.
\end{proof}

\begin{cor}\label{cor:restriction-of-eta}
Assume $\lambda \in \mathcal{D}_u(\bG_s)^F$ is a Dynkin cocharacter associated to $u$. If $\eta_{u}^G$, resp., $\eta_{u}^{G_s}$, is defined with respect to the Kawanaka datum $\mathcal{K}$, resp., $\mathcal{K}|_{\bG_s}$, then
\begin{equation*}
\eta_{u}^{G_s} = \Res^{U_G(\lambda,-2)}_{U_{G_s}(\lambda,-2)}(\eta_{u}^G)
\end{equation*}
as linear characters of $U_{G_s}(\lambda,-2)$.
\end{cor}

\subsection{Dynkin Cocharacters}\label{eq:dynk-cochar-centraliser}
Let us end by addressing the relationship between the sets of Dynkin cocharacters $\mathcal{D}_u(\bG_s)$ and $\mathcal{D}_u(\bG)$. As $\phi_{\spr}|_{\mathcal{U}(\bG_s)}$ is a Springer isomorphism onto its image we may, and will, assume that the set $\mathcal{D}_u(\bG_s)$ of Dynkin cocharacters is defined with respect to $\phi_{\spr}|_{\mathcal{U}(\bG_s)}$ so that
\begin{equation*}
\mathcal{D}_u(\bG_s) = \mathcal{D}_{\phi_{\spr}(u)}(\bG_s),
\end{equation*}
c.f., \cref{subsec:canonical-para-levi}. By the result of Fowler--R\"ohrle, c.f., \cref{thm:fowler-rohrle}, we get that $\mathcal{D}_u(\bG_s) = \mathcal{D}_u(\bG,\bG_s)$.

\section{Kawanaka Characters and their Values on Mixed Classes}\label{sec:kawanakageneral}

From now on we assume that $\bG$ is proximate and we fix a Kawanaka datum $\mathcal{K} = (\phi_{\spr},\kappa,\chi_p)$ for $\bG$. We recall here Kawanaka's modified version of the GGGCs and state a first result on their values.   

\subsection{Admissible Coverings} 
Let $u \in \mathcal{U}(\bG)^F$ be a unipotent element of $G = \bG^F$. Recall from \cref{ssec:rationalclassesrep} that the $\bG^F$-orbits on $\Cl_\bG(u)^F$ are parametrised by the $F$-conjugacy classes of the finite group $A_\bG(u)$.
On the other hand, Lusztig's classification of unipotent characters \cite{Lu84} involves a certain quotient of $A_{\bG}(u)$ known as the \emph{canonical quotient}. This classification will be recalled in \cref{ssec:unipotentchar}.

As $u$ is $F$-fixed the Frobenius endomorphism acts on $A_{\bG}(u)$. We will consider short exact sequences of the form
\begin{equation*}
1 \longrightarrow N \longrightarrow A_{\bG}(u) \longrightarrow \bar{A} \longrightarrow 1
\end{equation*}
so that $\bar{A}$ is a quotient of $A_{\bG}(u)$. We say such a sequence, or quotient $\bar{A}$, is \emph{defined over $\mathbb{F}_q$} if $F(N) = N$ so that $F$ acts on $\bar{A}$. To construct Kawanaka characters we will need to be able to lift the quotient $\bar{A}$ to a subgroup of $C_{\bG}(u)$. Such a lift will need to satisfy the conditions encapsulated in the following definition.

\begin{definition}\label{def:admissiblesplitting}
Let $u \in \mathcal{U}(\bG)^F$ be a rational unipotent element. A pair $(A,\lambda)$ consisting of a subgroup $A \leqslant C_{\bG}(u)$ and a Dynkin cocharacter $\lambda \in \mathcal{D}_u(\bG)^F$ is said to be \emph{admissible} for $u$ if the following hold:
\begin{admissible}
 \item\label{en:K0} $A \subset \bL(\lambda)^F$ so that $A$ is a finite group,
 \item\label{en:K1} $A$ consists of semisimple elements,
 \item\label{en:K2} $a \in C_{\bL(\lambda)}^\circ(C_A(a))$ for all $a \in A$.
\end{admissible}
Now assume $\bar{A}$ is a quotient of the component group $A_{\bG}(u)$ defined over $\mathbb{F}_q$. We say that an admissible pair $(A,\lambda)$ for $u$, as above, is an \emph{admissible covering} for $\bar{A}$ if:
\begin{admissible}[resume]
 \item\label{en:K3} the restriction to $A$ of the map $C_\bG(u) \twoheadrightarrow \bar A$ fits into the following 
 short exact sequence
 \begin{equation*}
 1 \longrightarrow Z \longrightarrow A \longrightarrow \bar A \longrightarrow 1
 \end{equation*}
 where $Z \leqslant Z(A)$ is a central subgroup with $Z\cap [A,A] = \{1\}$.
\end{admissible}
Note we necessarily have that $F$ must act trivially on $\bar{A}$ for such a covering to exist.
\end{definition}

The optimal scenario here is to be able to find an admissible covering of the entire component group $A_{\bG}(u)$, assuming $F$ acts trivially on $A_{\bG}(u)$. Unfortunately, even when $F$ acts trivially on $A_{\bG}(u)$ it will not always be possible to find such a covering, see the remark preceding \cref{prop:admissible-splitting-ortho}.

The main issue here is that condition \cref{en:K2} might not hold in general. Several of the following sections will be concerned with showing that admissible coverings exist when $\bG$ is simple and $\bar{A}$ is taken to be Lusztig's canonical quotient. We now make a few remarks on the definition of admissible coverings that are to be used later on.

\begin{rmk}\label{rmk:K3-and-centralisers}
Let $a \in A$ and $\bar a$ be its image in $\bar A$. Then under the assumption \cref{en:K3} there is a short exact sequence
\begin{equation*} 
1 \longrightarrow Z \longrightarrow C_A(a) \longrightarrow C_{\bar A}(\bar a) \longrightarrow 1.
\end{equation*}
Indeed, if $b \in A$ is such that $ab \in ba Z$ then $a^{-1}b^{-1}ab \in Z \cap [A,A]$, and therefore $ab = ba$. This shows the surjectivity of the map $C_A(a) \longrightarrow C_{\bar A}(\bar a)$. The fact that its kernel is $Z$ is obvious.
\end{rmk}

\begin{rem}\label{rem:cyclic-centraliser}
Assume $\lambda \in \widecheck{X}(\bG)$ is a cocharacter and $A \leqslant \bL(\lambda)$ is any subgroup consisting of semisimple elements. If $a \in A$ is an element with $C_A(a)$ cyclic then we have $C_{A}(a) \leqslant C_{\bL(\lambda)}^{\circ}(C_{A}(a))$. Indeed, if $b \in C_{A}(a)$ is a generator then $b \in C_{\bL(\lambda)}^{\circ}(b) = C_{\bL(\lambda)}^{\circ}(C_{A}(a))$ by \cite[Prop.~2.5]{DiMibook} because $b$ is semisimple.
\end{rem}

\begin{rem}\label{rem:in-max-torus}
Assume $\lambda \in \widecheck{X}(\bG)^F$ is a cocharacter and $\bT \leqslant \bL(\lambda)$ is an $F$-stable maximal torus. If $A \leqslant \bT^F$ is any subgroup then for any $a \in A$ we have
\begin{equation*}
a \in A \leqslant \bT \leqslant C_{\bL(\lambda)}^{\circ}(A) \leqslant C_{\bL(\lambda)}^{\circ}(C_{A}(a)).
\end{equation*}
Hence \cref{en:K2} will automatically hold in this situation.
\end{rem}

In the case of classical groups we will not work directly with adjoint groups but with symplectic and special orthogonal groups. The following allows us to descend to adjoint groups. For this we introduce the following notation: for any $u \in \bG$ we denote by $Z_{\bG}(u)$ the image of the natural map $Z(\bG) \to A_{\bG}(u)$.

\begin{lem}\label{lem:coveringforGad}
Let $\pi : \bG \to \bG_{\ad}$ be an adjoint quotient of $\bG$. If $u \in \bG$ is unipotent then we have a short exact sequence
\begin{equation*}
1 \longrightarrow Z_{\bG}(u) \longrightarrow A_{\bG}(u) \longrightarrow A_{\bG_{\ad}}(\pi(u)) \longrightarrow 1,
\end{equation*}
Assume we have a short exact sequence
\begin{equation*}
1 \longrightarrow N \longrightarrow A_{\bG}(u) \longrightarrow \bar{A} \longrightarrow 1,
\end{equation*}
as above, with $Z_{\bG}(u) \leqslant N$ then $\bar{A}$ is naturally a quotient of $A_{\bG_{\ad}}(\pi(u))$. Moreover, if $(A,\lambda)$ is an admissible covering for $u$ of $\bar{A}$ then $(\pi(A),\pi\circ\lambda)$ is an admissible covering for $\pi(u)$ of $\bar{A}$.
\end{lem}

\begin{proof}
The first short exact sequence is clear because $C_{\bG}(u) = \pi^{-1}(C_{\bG_{\ad}}(\pi(u)))$, which follows as $u$ is unipotent and $\Ker(\pi) = Z(\bG)$ consists of semisimple elements. We denote by $F : \bG_{\ad} \to \bG_{\ad}$ the Frobenius endomorphism induced by that on $\bG$, i.e., we have $F\circ\pi = \pi \circ F$. Moreover, we set $\bL = C_{\bG}(\lambda)$ and $\bM = C_{\bG_{\ad}}(\pi\circ\lambda)$. It is clear that $\pi(\bL) \leqslant \bM$ so certainly $\pi(A) \leqslant \bM^F$, because $\pi$ is $F$-equivariant, so \cref{en:K0} holds. That \cref{en:K1} holds is clear because the image of a $p'$-element is again a $p'$-element.

By assumption it follows that we have a short exact sequence
\begin{equation*}
1 \longrightarrow \pi(Z) \longrightarrow \pi(A) \longrightarrow \bar A\longrightarrow 1,
\end{equation*}
with the notation as in \cref{en:K3}. Note that the image of $Z$ in $A_{\bG}(u)$ must contain $Z_{\bG}(u)$ by assumption. Certainly $\pi(Z) \leqslant Z(\pi(A))$ and as $[\pi(A),\pi(A)] = \pi([A,A])$ we have $\pi(Z) \cap [\pi(A),\pi(A)] = \pi(Z) \cap \pi([A,A])$. Now if $x = \pi(y) = \pi(z)$ with $y \in [A,A]$ and $z \in Z$ then $y=zz'$ for some $z' \in  A \cap \Ker \pi  \subset Z$ and \cref{en:K3} for $A$ forces $y=1$ hence $x=1$. This shows that \cref{en:K3} holds for $\pi(A)$.

Finally let us consider \cref{en:K2}. For any $a \in A$ we certainly have $\pi(C_A(a)) \leqslant C_{\pi(A)}(\pi(a))$. Now assume $\pi(b) \in C_{\pi(A)}(\pi(a))$ for some $b \in A$. It follows from the discussion in \cref{rmk:K3-and-centralisers} that $bab^{-1}a^{-1} \in Z$ so $b \in C_A(a)$. Hence $\pi(C_A(a)) = C_{\pi(A)}(\pi(a))$ and
\begin{equation*}
\pi(a) \in \pi(C_{\bL}^{\circ}(C_A(a))) = \pi(C_{\bL}(C_A(a)))^{\circ} \leqslant C_{\bM}^{\circ}(\pi(C_A(a))) = C_{\bM}^{\circ}(C_{\pi(A)}(\pi(a)))
\end{equation*}
for any $a \in A$ so \cref{en:K2} holds.
\end{proof}

\subsection{Admissible Representatives for $\Cl_\bfG(u)^F$}
We have seen in \cref{prop:rat-class-param} that if $u \in \mathcal{U}(\bG)^F$ then the $F$-conjugacy classes of $A_{\bG}(u)$ parametrise the $G$-conjugacy classes of $\Cl_{\bG}(u)^F$. Now assume $(A,\lambda)$ is an admissible pair for $u$ as in \cref{def:admissiblesplitting}. Being a subgroup of $G$ the group $A$ acts on $\Cl_{\bG}(u)^F$ by conjugation. However, it also acts by conjugation on $A_{\bG}(u)$ via the natural homomorphism $A \to A_{\bG}(u)$. The following shows that we can choose representatives that, roughly speaking, are equivariant with respect to these actions.

\begin{lemma}\label{lem:admissiblerep}
Assume $(A,\lambda)$ is an admissible pair for $u \in \mathcal{U}(\bG)^F$. Then there exists a set of unipotent elements $\{u_a \mid a \in A\} \subseteq \Cl_\bfG(u)^F$ satisfying the following conditions, for all $a \in A$:
\begin{enumerate}
 \item the $G$-conjugacy class of $u_a$ corresponds to the image of $a$ in $A_{\bG}(u)$ under the correspondence of \cref{prop:rat-class-param},
 \item for all $b \in A$ we have $bu_ab^{-1} = u_{bab^{-1}}$,
 \item for all $b \in C_A(a)$ we have $u_a \in C_{\bG}^{\circ}(b)^F$ and $\lambda \in \mathcal{D}_{u_a}(C_{\bG}^{\circ}(b))$.
\end{enumerate}
Any set $\{u_a \mid a \in A\}$ satisfying these conditions will be called a set of \emph{admissible representatives}.
\end{lemma}

\begin{proof}
Given $a \in A$, we can use \cref{en:K2} to find $g_a \in C_{\bfL(\lambda)}^\circ(C_A(a))$ such that $g_a^{-1} F(g_a) = a$. If we set $u_a := g_a u g_a^{-1}$ then (i) follows by the definition of the correspondence. Every $b \in A$ is $F$-stable therefore 
$$(bg_ab^{-1})^{-1} F(bg_ab^{-1}) = bab^{-1}.$$
Since $g_a$ centralises $C_A(a)$ this shows that we can assume, without loss of generality, that $g_{bab^{-1}} = bg_ab^{-1}$. Consequently for all $b \in A \subset C_G(u)$ we have 
$$bu_ab^{-1} = bg_aug_a^{-1} b^{-1} = bg_ab^{-1}ubg_a^{-1} b^{-1} = u_{bab^{-1}}$$
which proves (ii). In particular $u_a \in C_G^\circ(b)$ whenever $b$ commutes with $a$. 

Finally, since $b \in L(\lambda) = C_{\bG}(\lambda)^F$ the image of $\lambda$ lies in $C_{\bfG}^{\circ}(b)$ so $\lambda \in \mathcal{D}_u(\bG,C_\bfG^\circ(b)) = \mathcal{D}_u(C_\bfG^\circ(b))$, see \cref{eq:dynk-cochar-centraliser}. Therefore ${}^{g_a}\lambda = \lambda$ lies in $\mathcal{D}_{u_a}({}^{g_a}C_\bfG^\circ(b))$. Thus (iii) follows from the fact that by construction $g_a \in C_\bfG^\circ(b)$ whenever $b \in C_A(a)$. 
\end{proof}

\subsection{Weil Representations of Symplectic Groups}\label{subsec:weil}
In this section we assume that $p = \Char(K)$ is odd. Moreover, $V$ is a finite-dimensional $K$-vector space equipped with a Frobenius endomorphism $F : V \to V$ endowing $V$ with an $\mathbb{F}_q$-rational structure and $\omega : V \times V \to K$ is a non-degenerate alternating bilinear form defined over $\mathbb{F}_q$.

Following G\'erardin \cite{Ger} we define the \emph{Heisenberg group} $H := \Hei(V^F \mid \omega)$ of the symplectic space $(V^F \mid \omega)$ to be the set $V \times \mathbb{F}_q^+$ with group law defined by
\begin{equation*}
(v,k)(v',k') = (v+v',k+k'+\omega(v,v')/2).
\end{equation*}
Note, this definition is such that $(0,\omega(v,v')) = [(v,k),(v',k')]$ for all $(v,k),(v',k') \in H$, where $[-,-]$ denotes the commutator in $H$, and $Z := Z(H) = \{(0,k) \mid k \in \mathbb{F}_q^+\}$.

Let $S = \Sp(V\mid \omega)^F$ be the symplectic group determined by the form $\omega$. We have a natural action of $S$ on $H$ given by its action on $V$, in particular this action fixes pointwise $Z$. We denote by $\SHei(V^F \mid \omega)$ and $SH$ the semidirect product $S \ltimes H$. Now assume $\eta \in \Irr(Z)$ is a non-trivial linear character of the centre. There exists a unique irreducible character $\zeta_{\eta} \in \Irr(H)$ whose restriction to $Z$ is a multiple of $\eta$. It is supported only on $Z$ and $\zeta_{\eta}(1) = [H:Z]^{1/2}$, see \cite[Lem.~1.2]{Ger}. The unicity implies this character is invariant under $S$.

\begin{thm}[{}{G\'erardin, \cite[Thm.~2.4]{Ger}}]\label{thm:weil-extension}
Recall our assumption that $p$, hence $q$, is odd. There exists an extension $\tilde{\zeta}_{\eta} \in \Irr(SH)$ of the character $\zeta_{\eta} \in \Irr(H)$. If we stipulate that $\overline{\tilde{\zeta}_{\eta}} \neq \tilde{\zeta}_{\eta}$ when $(\dim(V),q) = (2,3)$ then this extension is unique. Furthermore, if $t \in S$ is semisimple then for any two non-trivial irreducible characters $\eta,\eta' \in \Irr(Z)$ we have $\tilde{\zeta}_{\eta}(t) = \tilde{\zeta}_{\eta'}(t)$.
\end{thm}

\begin{proof}
The statement concerning semisimple elements is implied by \cite[Thm.~2.4]{Ger} but here is a direct proof. Let $\bC = \Sp(V \mid \omega)\cdot Z(\GL(V))$ be the conformal symplectic group and let $C = \bC^F$ be the fixed points under the natural Frobenius endomorphism induced from that on $V$. We have a surjective homomorphism $\mu : C \to \mathbb{F}_q^{\times}$ such that $\omega(gx,gy) = \mu(g)\omega(x,y)$.

The group $C$ acts automorphically on $H = V \times \mathbb{F}_q^+$ by setting $g\cdot (v,k) = (gv, \mu(g)k)$. This clearly extends the action of $S \lhd C$ on $H$ so $SH \lhd C \ltimes H$. We define an action of $\mathbb{F}_q^{\times}$ on $\Irr(\mathbb{F}_q^+)$ given by $(\xi\eta)(k) = \eta(\xi k)$ for all $k \in \mathbb{F}_q^+$ and $\xi \in \mathbb{F}_q^{\times}$. Now assume $\eta \in \Irr(Z)$ is non-trivial then for any $g \in C$ we have ${}^g\zeta_{\eta}$ is an irreducible character of $H$ whose restriction to $Z$ is a multiple of $\mu(g)\eta$. Hence ${}^g\zeta_{\eta} = \zeta_{\mu(g)\eta}$ and it follows that ${}^g\tilde{\zeta}_{\eta} = \tilde{\zeta}_{\mu(g)\eta}$.

If $t \in S$ is semisimple and $g \in C$ then ${}^gt \in S$ is $S$-conjugate to $t$. Indeed, from the definition of $\bC$ it follows that ${}^gt$ and $t$ are $\Sp(V\mid\omega)$-conjugate. However as $t$ is semisimple its centraliser in $\Sp(V\mid\omega)$ is connected because the symplectic group is simply connected, therefore $t$ and ${}^gt$ must be $S = \Sp(V\mid\omega)^F$-conjugate. The statement now follows because, in the action defined above, $\mathbb{F}_q^{\times}$ acts transitively on the non-trivial characters of $\mathbb{F}_q^+$.
\end{proof}

\begin{definition}
The unique extension $\tilde{\zeta}_{\eta}$ of $\zeta_{\eta}$ specified in \cref{thm:weil-extension} will be called the \emph{Weil extension}. A representation of $S = \Sp(V\mid \omega)$ with character $\Res_S^{SH}(\tilde{\zeta}_{\eta})$ is called a (reducible) \emph{Weil representation} of $S$.
\end{definition}

\subsection{Kawanaka Characters for Admissible Pairs}\label{ssec:kawnakadef}
\begin{assumption}
From now, until the end of this section we assume that $p \neq 2$. We fix a unipotent element $u \in \mathcal{U}(\bG)^F$, an admissible pair $(A,\lambda)$ for $u$, and a set $\{u_a = {}^{g_a}u \mid a \in A\}$ of admissible representatives as in \cref{lem:admissiblerep}.
\end{assumption}

Let $\bL = \bL(\lambda) = C_{\bG}(\lambda)$ and $L = \bL^F$. For any $a \in A$ the group $C_A(a) \leqslant C_L(u_a)$ normalises $U(\lambda,-1)$. Moreover, by \cref{eq:conj-eta-char} we see that $C_A(a)$ fixes the character $\zeta_{u_a}^G := \zeta_{u_a,\lambda}^G$ under the natural conjugation action. We wish to define an extension of $\zeta_{u_a}^G$ to the group $C_A(a)U(\lambda,-1) = C_A(a) \ltimes U(\lambda,-1)$, which is a semidirect product as $C_A(a)$ is contained in the Levi complement $\bL$. Following Kawanaka we get such an extension by using the (reducible) Weil representation of the symplectic group.

Let $V = \lie{g}(\lambda,-1)$ and let $H = \Hei(V^F \mid \omega_a)$ be the Heisenberg group and $S = \Sp(V \mid \omega_a)^F$ the symplectic group, as in \cref{subsec:weil}. Here $\omega_a$ is the form defined as in \cref{ssec:gggcdef} with respect to the element $u_a$. We denote by $SH$ the semidirect product $S \ltimes H$. The endomorphism of $H$
$$(v,k) \longmapsto (v,\Tr_{\mathbb{F}_q/\mathbb{F}_p}(k))$$
has central kernel $K$ and $H/K$ is an extra special $p$-group of exponent $p$. 

It follows from \cref{eq:ind-char-values} that $\Ker(\eta_{u_a}^G) = \Ker(\zeta_{u_a}^G)$ and, furthermore, the quotient group $Q := U(\lambda,-1)/\Ker(\zeta_{u_a}^G)$ is either a cyclic group of order $p$ or an extraspecial $p$-group of exponent $p$, see \cref{lem:rad-extraspecial}. Assume $Q$ is extraspecial. The group $C_A(a) \leq C_L(u_a)$, through its action on $U(\lambda,-1)$, fixes the character $\zeta_{u_a}^G$ so it acts on the quotient $Q$. Clearly this action preserves $Z(Q)$ and, in fact, it pointwise fixes $Z(Q)$ since it fixes $\eta_{u_a}^G$. Hence we have a homomorphism $C_A(a) \to \Aut^*(Q)$, in the notation of \cref{subsec:strut-rad}.

As in the proof of \cref{lem:is-subgrp}, the Springer morphism induces a $C_A(a)$-equivariant group isomorphism $\overline{\pi} : \overline{Q} \simto V^F$, where $\overline{Q} = Q/Z(Q)$. We can identify $Z(Q)$ with $Z(H/K)$ via the isomorphism $z \mapsto (0,\gamma(z))+K$, where $\gamma : Z(Q) \to \mathbb{F}_p$ is as in \cref{subsec:strut-rad}. With this identification we have $\overline{\pi}$ intertwines the forms $\omega_Q$ and $\omega_a$ as in \cref{lem:extra-iso-lift}, see \cref{eq:omegaQ}.

Certainly $C_A(a)$ stabilises the form $\omega_a$ so we have a homomorphism $C_A(a) \to S$. We have a natural homomorphism $S \to \Aut^*(H/K)$ and so by composition we obtain a homomorphism $C_A(a) \to \Aut^*(H/K)$. By \cref{lem:extra-iso-lift} the isomorphism $\overline{\pi}$ lifts to a $C_A(a)$-equivariant isomorphism $\pi : Q \simto H/K$. This then lifts, trivially, to an isomorphism $C_A(a)\ltimes Q \to C_A(a)\ltimes H/K$. Recapitulating we have the following commutative diagram of homomorphisms
\begin{center}
\begin{tikzcd}
C_A(a)\ltimes U(\lambda,-1) \ar[r,twoheadrightarrow] & C_A(a) \ltimes Q  \ar[r,"\sim"]  & C_A(a) \ltimes H/K\ar[r,hook] &  S \ltimes H/K& SH \ar[l,twoheadrightarrow] \\
U(\lambda,-1) \ar[r,twoheadrightarrow] \ar[u,hook] &  Q   \ar[r,"\sim"]   \ar[u,hook]&  H/K    \ar[u,hook]& H  \ar[ur,hook]\ar[l,twoheadrightarrow] 
\end{tikzcd}
\end{center}

By deflation followed by inflation along the bottom row of the previous diagram, we may view $\zeta_{u_a}^G$ as an irreducible character of $H$ whose restriction to $Z(H)$ is a multiple of a non-trivial irreducible character, namely the inflation of $\eta_{u_a}^G$. After \cref{thm:weil-extension} this can be extended to a unique character of $SH$, the Weil extension, which clearly has $K$ in its kernel. Deflating, restricting and inflating 
along the top row of the diagram, we obtain an extension of $\zeta_{u_a}^G$ to $C_A(a)\ltimes U(\lambda,-1)$ which we call the \emph{Weil extension}. We formalise our discussion in the following definition.

\begin{definition}\label{def:zeta-ext}
For any $a \in A$ we define an extension $\tilde{\zeta}_{u_a}^G \in \Irr(C_A(a) \ltimes U(\lambda,-1))$ of $\zeta_{u_a}^G$ as follows:
\begin{itemize}
	\item if $\zeta_{u_a}^G$ is linear then $\tilde{\zeta}_{u_a}^G$ is the unique extension whose restriction to $C_A(a)$ is a multiple of the trivial character,
	\item if $\zeta_{u_a}^G$ is not linear then we take $\tilde{\zeta}_{u_a}^G$ to be the Weil extension, as defined above.
\end{itemize}
\end{definition}

\begin{rem}
In defining the Weil extension above we made a choice of a $B$-invariant isomorphism $Q \simto H/K$ lifting $\overline{Q} \simto V^F$. As stated in \cref{lem:extra-iso-lift} any other choice is obtained by composing with an element of $\Inn(Q)$ centralised by $C_A(a)$. This is realised by an inner automorphism of $C_A(a)\ltimes Q$ so the extension is independent of the choice we made.
\end{rem}

The purpose of defining the extension in this way is that we may state the following concerning its values.

\begin{lem}[G\'erardin]\label{lem:weil-signs}
There exists a class function $\varepsilon \in \Class(A)$ such that for each $a \in A$ and $t \in C_A(a)$ the following hold:
\begin{enumerate}
	\item $\varepsilon(t) \in \{\pm1\}$,
	\item $\tilde{\zeta}_{u_a}^G(t) = \varepsilon(t)q^{\dim(\lie{g}_t(\lambda,-1))/2}$.
\end{enumerate}
\end{lem}

\begin{proof}
If $V = \lie{g}(\lambda,-1) = \{0\}$ then this is clear as $\zeta_u^G$ is linear and the function $\varepsilon$ is defined by simply taking $\varepsilon(t) = 1$ for all $t \in A$.

We now consider the case where $V \neq \{0\}$. In this case none of the characters $\zeta_{u_a}^G$ are linear. Consider first the case where $a = 1$. It is shown by G\'erardin in \cite[Cor.~4.8.1]{Ger} that there exists a sign $\varepsilon(t)$ such that
\begin{equation*}
\tilde{\zeta}_{u}^G(t) = \varepsilon(t)q^{\dim(\lie{g}_t(\lambda,-1))/2}.
\end{equation*}
This gives a function $\varepsilon : A \to \{\pm1\}$. We see that $\varepsilon$ is a class function because if $t_1,t_2 \in A$ are conjugate in $A$ then they are conjugate in $C_L(u)$.

We now consider the case where $a \neq 1$. Let us denote by $\pi : \bL \to \GL(V)$ the natural map so that $\pi(s)v = s\cdot v$ for all $s \in \bL$ and $v \in V$. Recall that $u_a = {}^{g_a}u$ where $g_a \in C_{\bL}^{\circ}(C_A(a))$ satisfies $g_a^{-1}F(g_a) = a$. As $a \in A \leqslant C_L(u)$ we have $\pi(a) \in \Sp(V\mid \omega)$, where $\omega = \omega_1$. We pick an element $h \in \Sp(V\mid \omega)$ such that $\pi(a) = h^{-1}F(h)$ then $\pi(g_a)h^{-1} \in \GL(V)^F$. Let $\phi = h\pi(g_a^{-1}) \in \GL(V)^F$. From the definitions we see that for all $x,y \in V$ we have
\begin{equation*}
\omega_a(x,y) = \omega(g_a^{-1}\cdot x,g_a^{-1}\cdot y) = \omega(\phi x,\phi y)
\end{equation*}
where the last equality follows because $h \in \Sp(V \mid \omega)$. Thus we obtain an isomorphism $\SHei(V^F \mid \omega_a) \to \SHei(V^F \mid \omega)$ given by $s \mapsto {}^{\phi}s = \phi s \phi^{-1}$.

By transport of structure we obtain a character ${}^{\phi}\tilde{\zeta}_{u_a}^G$ on $\SHei(V^F\mid\omega)$. The restriction of this character to $Z(\Hei(V^F \mid \omega))$ is clearly a multiple of a non-trivial irreducible character. Hence the restriction of ${}^{\phi}\tilde{\zeta}_{u_a}^G$ is the character of a (reducible) Weil representation of $\Sp(V \mid \omega)^F$. As $t \in C_A(a)$ is semisimple it follows from \cref{thm:weil-extension} that ${}^{\phi}\tilde{\zeta}_{u_a}^G(t) = \tilde{\zeta}_u^G(t)$. However, as ${}^{g_a}t = t$ we have that ${}^{\phi}t = {}^ht$ so ${}^{\phi}t$ and $t$ are $\Sp(V\mid\omega)$-conjugate elements. They must, therefore, be $\Sp(V\mid\omega)^F$-conjugate so ${}^{\phi}\tilde{\zeta}_{u_a}^G(t) = {}^{\phi}\tilde{\zeta}_{u_a}^G({}^{\phi}t) = \tilde{\zeta}_{u_a}^G(t)$. This shows that $\tilde{\zeta}_u^G(t) = \tilde{\zeta}_{u_a}^G(t)$ for all $t \in C_A(a)$.
\end{proof}

\begin{definition}
We will call the class function $\varepsilon \in \Class(A)$ defined in \cref{lem:weil-signs} the \emph{Weil-sign} character of $A$.
\end{definition}

As we define it here, the Weil-sign character of $C_A(a)$ is not necessarily a character of $C_A(a)$, only a class function. In almost all of the situations we consider in this paper it will be a genuine character. However, we will not need this fact here. With all of this in hand we are now ready to define Kawanaka characters.

\begin{definition}\label{def:kaw-chars}
Assume $a \in A$ and $\psi \in \Irr(C_A(a))$ and let $\tilde{\zeta}_{u_a}^G \in \Irr(C_A(a)\cdot U(\lambda,-1))$ be the extension of $\zeta_{u_a}^G \in \Irr(U(\lambda,-1))$ defined in \cref{def:zeta-ext}. We define the \emph{Kawanaka character} associated to the pair $(a,\psi)$ to be
\begin{equation*}
K_{(a,\psi)}^G :=  \Ind_{C_A(a)\ltimes U(\lambda,-1)}^G \left(\tilde\zeta_{u_a}^G \otimes \Inf_{C_A(a)}^{C_A(a)\ltimes U(\lambda,-1)} \psi \right).
\end{equation*}
When the ambient finite group is clear we will denote it simply by $K_{(a,\psi)}$.
\end{definition}

\begin{rem}\label{rem:kaw-chars-def}
Assume $a,b \in A$ then we have ${}^b\tilde{\zeta}_{u_a}^G = \tilde{\zeta}_{u_{bab^{-1}}}^G$. Indeed, the restriction of ${}^b\tilde{\zeta}_{u_a}^G$ to $U(\lambda,-1)$ coincides with ${}^b\zeta_{u_a}^G = \zeta_{bu_ab^{-1}}^G = \zeta_{u_{bab^{-1}}}^G$ by \cref{eq:conj-zeta-char,lem:admissiblerep}. Hence the equality follows from the unicity of the Weil extension. This implies immediately that for any $\psi \in \Irr(C_A(a))$ we have $K_{({}^ba,{}^b\psi)} = K_{(a,\psi)}$.

We note in addition that it is immediately clear from the definition that the corresponding GGGCs
\begin{equation}\label{eq:GGGCs-sum-Kaw}
\Gamma_{u_a} = \sum_{\psi \in \Irr(C_A(a))} \psi(1)K_{(a,\psi)}
\end{equation}
are a sum of Kawanaka characters. This follows simply by decomposing the regular representation of $C_A(a)$.
\end{rem}

\subsection{Character Formula}
We now obtain a character formula for the values of Kawanaka characters on mixed classes in terms of GGGCs. A formula of this kind was stated, without proof, by Kawanaka \cite[Lem.~2.3.5]{Kaw86}. A similar formula already appears in \cite[Prop.~3.5]{Sho06} and \cite[Satz 3.2.11]{Wings}. As this may be needed in different situations in the future we state the first part purely in the context of arbitrary finite groups.

\begin{lem}\label{lem:ind-formula}
Assume $G$ is a finite group, $p > 0$ is a prime, $U \leqslant G$ is a $p$-subgroup of $G$, and $B \leqslant N_G(U)$ is a $p'$-group. Let $H = BU$ and suppose $\gamma = \Ind_H^G(\chi \otimes \Inf_B^H\psi )$ for some class functions $\chi \in \Class(H)$ and $\psi \in \Class(B)$. Moreover, let $sv = vs \in G$ with $s\in G$ a $p'$-element and $v \in G$ a $p$-element.
\begin{enumerate}
	\item If $s$ is not $G$-conjugate to an element of $B$ then $\gamma(sv) = 0$.
	\item If $s \in B$ then for each $t \in B$ which is $G$-conjugate to $s$ choose an element $x_t \in G$ such that ${}^{x_t}s = t$ and set $v_t = {}^{x_t}v$ then
    \begin{equation*}
    \gamma(sv) = \frac{1}{|B|}\sum_{\substack{t \in B\\ t\sim_G s}} \psi(t)\Ind_{C_U(t)}^{C_G(t)}(t\cdot \chi)(v_t)
    \end{equation*}
    where $t\cdot\chi \in \Class(C_U(t))$ is the function defined by $(t\cdot\chi)(g) = \chi(tg)$ for any $g \in C_U(t)$.
\end{enumerate}
\end{lem}

\begin{proof}
By definition we have
\begin{equation*}
\gamma(sv) = \frac{1}{|B| |U|} \sum_{\begin{subarray}{c} x \in G  \\ {}^x (sv) \in BU \end{subarray}} \chi({}^x (sv)) \psi({}^x (sv))
\end{equation*}
where we identify $\psi$ with the inflation. Since $B$ is a $p'$-group the decomposition ${}^x(sv) =({}^x s) ({}^xv)$ is the decomposition of an element into its $p'$-part and $p$-part. The $p$-element ${}^x v$ must lie in the (unique) Sylow $p$-subgroup $U$ of $BU$. Let $b \in B$ be the projection of the $p'$-element ${}^x s$ on $B$ so that $({}^x s) U = b U$. The cyclic groups generated by $b$ and by ${}^x s$ are complements of $U$ in $\langle b,U \rangle$. By the Schur--Zassenhaus theorem they must be conjugate, therefore ${}^x s\in {}^y B$ for some $y \in U$. This shows in particular that (i) holds.

Now assume $s \in B$. If ${}^x s \in {}^y B$ and ${}^x s \in {}^{y'} B$ for $y,y' \in U$ then $t = {}^{y^{-1}x} s$ and $t' = {}^{{y'}^{-1}x} s$ are elements of $B$ which are conjugate under $U$, hence we can write $t' = ztz^{-1}$ for some $z \in U$. But $t't^{-1} =  z(tz^{-1}t^{-1})$ forces $t = t'$ since $B$ normalises $U$. This proves that $y$ and $y'$ differ by left multiplication by an element of $C_U({}^xs)$. We get
\begin{align*}
\gamma(sv) &= \frac{1}{|B| |U|} \sum_{\begin{subarray}{c} (x,y) \in G \times U \\ {}^{y^{-1}x} s \in B \\ {}^xv \in U \end{subarray}} \frac{1}{|C_U({}^xs)|} \chi({}^x (sv)) \psi({}^x (sv)) \\
&= \frac{1}{|B|} \sum_{\begin{subarray}{c} x \in G\\ {}^{x} s \in B \\ {}^xv \in U \end{subarray}} \frac{1}{|C_U({}^xs)|} \chi({}^x(sv)) \psi({}^x s)
\end{align*}
where the last equality is obtained by translating $x$ to $yx$, and using the fact that $\psi$ is trivial on $U$. For each $t \in B$ with $t\sim_G s$ we fix an element $x_t \in G$ such that ${}^{x_t} s = t$ and we set $v_t = {}^{x_t} v$. Then ${}^x s = t$ if and only if $x \in C_G(t) x_t$ and in that case ${}^{xx_t^{-1}} {v_t} = {}^x v \in C_G(t)$ since $v \in C_G(s)$. Replacing $x$ by $xx_t$ and letting $x$ run over $C_G(t)$ we get
\begin{align*}
\gamma(sv) &= \frac{1}{|B|} \sum_{\begin{subarray}{c} t \in B \\ t \sim_G s \end{subarray}} \psi(t) \frac{1}{|C_U(t)|} \sum_{\begin{subarray}{c} x \in C_G(t) \\ {}^xv_t \in C_U(t) \end{subarray}}  \chi(t({}^xv_t))\\
&= \frac{1}{|B|} \sum_{\begin{subarray}{c} t \in B \\ t \sim_G s \end{subarray}} \psi(t)\Ind_{C_U(t)}^{C_G(t)}(t\cdot \chi)(v_t).\qedhere
\end{align*}
\end{proof}

\begin{prop}\label{lem:kawanaka-values}
Assume $a \in A$ and $\psi \in \Irr(C_A(a))$. If $sv = vs \in G$ is an element with $s \in G$ semisimple and $v \in G$ unipotent.
\begin{enumerate}
 \item If $s$ is not $G$-conjugate to an element of $C_A(a)$ then $K_{(a,\psi)}(sv) = 0$.
 \item If $s \in C_A(a)$ then for each $t \in C_A(a)$ which is $G$-conjugate to $s$, choose an element $x_t \in G$ such that ${}^{x_t} s =t$, and set $v_t = {}^{x_t} v$ then
 \begin{equation*}
 K_{(a,\psi)}(sv) = \frac{1}{|C_A(a)|} \sum_{\begin{subarray}{c} t \in C_A(a) \\ t \, \sim_{G} s \end{subarray}} \psi(t)\varepsilon(t) \Gamma_{u_a}^{C_G(t)}(v_t)
 \end{equation*}
 where $\varepsilon$ is the Weil-sign character as in \cref{lem:weil-signs} and $\Gamma_{u_a}^{C_G(t)} :=  \Ind_{C_G^\circ(t)}^{C_G(t)} (\Gamma_{u_a}^{C_G^\circ(t)})$ is the induction of the GGGC associated to the unipotent element $u_a$ in the connected reductive group $C_\bfG^\circ(t)$.
\end{enumerate}
\end{prop}

\begin{proof}
We let $\bU = \bU(\lambda,-1)$, $U = U(\lambda,-1)$, and $B = C_A(a)$. After \cref{lem:ind-formula} we see that (i) holds and we need only show (ii). We consider the formula in (ii) of \cref{lem:ind-formula} evaluated at $sv = vs$ with $s \in C_A(a)$. Let $t \in C_A(a)$ and assume $x_t \in G$ is such that $t = {}^{x_t}s$ then we set $v_t = {}^{x_t}v$. We need to show that
\begin{equation}\label{eq:first-equality}
\Ind_{C_U(t)}^{C_G(t)}(t\cdot \tilde{\zeta}_{u_a}^G)(v_t) = \varepsilon(t)\Gamma_{u_a}^{C_G(t)}(v_t).
\end{equation}
Let $\bG_t = C_{\bG}^{\circ}(t)$, $G_t = \bG_t^F$, and $\lie{g}_t = \Lie(\bG_t)$, as in \cref{sec:cent-s/s-elt}. Calculating the restriction using \cref{cor:restriction-of-eta} we see that
\begin{align*}
\Res_{C_U(t)}^{C_A(a)\cdot U}(\tilde{\zeta}_{u_a}^G) &= \Res_{C_U(t)}^{U}(\zeta_{u_a}^G) = q^{(\dim(\lie{g}(\lambda,-1))-\dim(\lie{g}_t(\lambda,-1)))/2}\zeta_{u_a}^{G_t}
\end{align*}
is a multiple of an irreducible character of $C_U(t) = U_{G_t}(\lambda,-1)$.

As $\langle t\rangle\cdot C_U(t)$ is a direct product it follows that every irreducible constituent of the restriction $\Res_{\langle t\rangle \cdot C_U(t)}^{C_A(a)\cdot U}(\tilde{\zeta}_{u_a}^G)$ is of the form $\psi \otimes \zeta_{u_a}^{G_t}$ with $\psi \in \Irr(\langle t\rangle)$. Hence
\begin{equation*}
\Res_{\langle t\rangle \cdot C_U(t)}^{C_A(a)\cdot U}(\tilde{\zeta}_{u_a}^G) = \chi \otimes \zeta_{u_a}^{G_t}
\end{equation*}
where $\chi$ is a (not necessarily irreducible) character of $\langle t\rangle$. After \cref{lem:weil-signs} we see that
\begin{equation*}
\chi(t) = \zeta_{u_a}^{G_t}(1)^{-1}\tilde{\zeta}_{u_a}^G(t) = \varepsilon(t)
\end{equation*}
is given by the Weil-sign character. Thus, we get that
\begin{equation*}
\Ind_{C_U(t)}^{C_G(t)}(t\cdot \tilde{\zeta}_{u_a}^G)(v_t) = \varepsilon(t)\Ind_{C_U(t)}^{C_G(t)}(\zeta_{u_a}^{G_t})(v_t)
\end{equation*}
from which the statement follows.
\end{proof}

\section{Fourier Transform of Kawanaka Characters}\label{sec:fourier}

Our goal in this section is to determine the multiplicity of a unipotent character in a given Kawanaka character.

\subsection{Lusztig's Classification of Unipotent Characters}\label{ssec:unipotentchar}
If $\mathcal{G}$ is a finite group then the irreducible representations of the Drinfeld double of $\mathcal{G}$ are parametrised by the set of orbits
\begin{equation*}
\mathscr{M}(\mathcal{G}) = \{(a,\psi) \mid a \in \mathcal{G}, \psi\in\Irr(C_{\mathcal{G}}(a)) \}/\mathcal{G}.
\end{equation*}
Here the group $\mathcal{G}$ acts by simultaneous conjugation and the orbit of a pair $(a,\psi)$ is denoted by $[a,\psi]$. Following Lusztig \cite[\S4]{Lu84}, we associate to any two pairs $[a,\phi], [b,\psi] \in \mathscr{M}(\mathcal{G})$ the Fourier coefficient
\begin{equation*}
\{[b,\phi],[a,\psi]\} = \sum_{\substack{x \in A \\ {}^x a \in C_A(b)}} \frac{\phi({}^xa) \overline{\psi(b^x)}}{|C_A(a)| |C_A(b)|}.
\end{equation*}
It is easily checked that the definition does not depend upon the choice of representatives for the equivalence classes $[a,\phi]$ and $[b,\psi]$.

More generally, assume $F : \mathcal{G} \to \mathcal{G}$ is an automorphism then we can consider the coset $\mathcal{G}F \subseteq \mathcal{G}\rtimes \langle F\rangle$. The group $\mathcal{G}$ acts on this coset by conjugation and we can consider the set
\begin{equation*}
\mathscr{M}(\mathcal{G},F) = \{(aF,\phi) \mid a \in \mathcal{G}, \phi\in\Irr(C_{\mathcal{G}}(aF)) \}/\mathcal{G},
\end{equation*}
where again $\mathcal{G}$ acts on the pairs by simultaneous conjugation and the orbit of a pair $(aF,\phi)$ is denoted by $[aF,\phi]$.

Recall from \cref{sec:introduction} that to each family $\mathscr{F} \in \Fam(W)$ there is a corresponding unipotent conjugacy class $\mathcal{O}_{\mathscr{F}} \subseteq \bG$, related via the Springer correspondence. A unipotent class of the form $\mathcal{O}_{\mathscr{F}}$ is said to be \emph{special}. For any family $\mathscr{F} \in \Fam(W)$ Lusztig has defined a quotient $\bar{A}_{\mathscr{F}}$ of the component group $A_{\bG}(u)$, with $u \in \mathcal{O}_{\mathscr{F}}$, known as the \emph{canonical quotient}. Recall that we have a dual family $\mathscr{F}^* = \mathscr{F}\otimes\sgn_W$ given by tensoring with the sign character. Following \cite[\S13]{Lu84} we have the set $\Uch(\mathscr{F}^*)$ is parameterised by $\mathscr{M}(\bar A_{\mathscr{F}},F)$ where $\bar A_{\mathscr{F}}$ is Lusztig's canonical quotient, see also \cite[Thm.~0.4]{Lu14}.

When $F$ acts trivially on $\bar A_{\mathscr{F}}$ then $C_{\bar A_{\mathscr{F}}}(aF) = C_{\bar A_{\mathscr{F}}}(a)$ and we have a bijection $\mathscr{M}(\bar A_{\mathscr{F}},F) \to \mathscr{M}(\bar A_{\mathscr{F}})$ given by $[aF,\phi] \longmapsto [a,\phi]$. When $\bfG/Z(\bG)$ is simple and $Z(\bG)$ is connected we can choose $u \in \mathcal{O}_{\mathscr{F}}^F$ such that $F$ acts trivially on $A_\bfG(u)$, and therefore on $\bar A_{\mathscr{F}}$, see \cite[Prop.~2.4]{Tay13}. Since the canonical quotient depends only on $W$ and $F$, this shows that $F$ acts trivially on $\bar A_{\mathscr{F}}$ whenever $\bfG$ is simple. 
 
\subsection{Fourier Transform}\label{ssec:fourierdef}
\begin{assumption}
We fix a unipotent element $u \in \mathcal{U}(\bG)^F$ and an admissible pair $(A,\lambda)$ as in \cref{def:admissiblesplitting}. Moreover, we assume that the Kawnaka characters are defined for $\Cl_{\bG}(u)$, with respect to $(A,\lambda)$, as in \cref{def:kaw-chars}.
\end{assumption}

For each orbit $[a,\psi] \in \mathscr{M}(A)$ we denote by $K_{[a,\psi]}$ the Kawanaka character $K_{(a,\psi)}$. This is well defined by \cref{rem:kaw-chars-def}. Using the Fourier coefficient defined in \cref{ssec:unipotentchar} we can define the Fourier transform of Kawanaka characters as follows: given $[b,\phi] \in \mathscr{M}(A)$, we set
\begin{align*}
F_{[b,\phi]} &:= \sum_{[a,\psi] \in \mathscr{M}(A)} \{[b,\phi],[a,\psi]\} K_{[a,\psi]}.
\end{align*}
We will also use the following equivalent expression for $F_{[b,\phi]}$, namely
\begin{equation}\label{eq:fourier-trans-kawanaka-alt}
F_{[b,\phi]} = \frac{1}{|A|}\sum_{x \in A}\sum_{a \in C_A(b)} \sum_{\psi \in \Irr(C_A(a^x))} \frac{\phi(a) \overline{\psi(b^x)}}{|C_A(b)|}K_{(a^x,\psi)}.
\end{equation}

Considering the special case where $b=1$ we get the following expression for the Fourier transform
\begin{align*}
F_{[1,\phi]} &= \frac{1}{|A|}\sum_{x \in A}\sum_{a \in A}\sum_{\psi \in \Irr(C_A(a^x))} \frac{\phi(a) \overline{\psi(1)}}{|A|} K_{(a^x,\psi)}\\
&= \frac{1}{|A|} \sum_{a \in A} \phi(a) \Big(\sum_{\psi \in \Irr(C_A(a))} \psi(1) K_{(a,\psi)} \Big)\\
& = \frac{1}{|A|} \sum_{a \in A} \phi(a) \Gamma_{u_a}
\end{align*}
where $\Gamma_{u_a}$ is the generalised Gelfand-Graev character (GGGC) of $G$ associated with $u_a$, see \cref{ssec:gggcdef,ssec:kawnakadef}. Note, that for the second equality we performed the change of variables $a \mapsto a^x$ and used the fact that $\phi({}^xa) = \phi(a)$ as $\phi \in \Irr(A)$.

The linear combination $F_{[1,\phi]}$ is often referred to as the \emph{Mellin transform} of the generalised Gelfand-Graev characters, studied for example in \cite{Lu92,Ge97,DLM}. It is unipotently supported and vanishes on many conjugacy classes. The following proposition generalises this observation to the other Fourier transforms.

\begin{proposition}\label{prop:fourier-values}
Recall our choice of unipotent element $u \in \mathcal{U}(\bG)^F$ and admissible pair $(A,\lambda)$. Consider a pair $[b,\phi] \in \mathscr{M}(A)$ and an element $sv = vs \in G$ with $s$ semisimple and $v$ unipotent. Then:
\begin{enumerate}
 \item $F_{[b,\phi]}(sv)=0$ if $s$ is not $G$-conjugate to $b$,
 \item for $s=b$ we have
$$F_{[b,\phi]} (bv) = \frac{\varepsilon(b)}{|C_A(b)|} \sum_{a \in C_A(b)}  \phi(a) \Gamma_{u_a}^{C_G(b)}(v),$$
where $\varepsilon$ is the Weil-sign character of $A$.
\end{enumerate}
\end{proposition}

\begin{proof}Let $sv$ be the Jordan decomposition of an element of $G$. Given
$[a,\psi] \in \mathscr{M}(A)$, the values of $K_{[a,\psi]}$ at $sv$ can be computed from \cref{lem:kawanaka-values}. Consequently we have, using \cref{eq:fourier-trans-kawanaka-alt}, that
\begin{align*}
F_{[b,\phi]}(sv) &= \frac{1}{|A|}\sum_{x \in A}\sum_{a \in C_A(b)}\sum_{\psi \in \Irr(C_A(a^x))} \frac{\phi(a) \overline{\psi(b^x)}}{|C_A(a^x)||C_A(b)|}\sum_{\substack{t \in C_A(a^x)\\ t \sim_G s}} \psi(t)\varepsilon(t)\Gamma_{u_{a^x}}^{C_G(t)}(v_t)\\
&= \frac{1}{|A|}\sum_{x \in A}\sum_{a \in C_A(b)}\sum_{\substack{t \in C_A(a^x)\\ t \sim_G s}} \frac{\phi(a)\varepsilon(t)\Gamma_{u_{a^x}}^{C_G(t)}(v_t)}{|C_A(a^x)||C_A(b)|} \left(\sum_{\psi \in \Irr(C_A(a^x))}\overline{\psi(b^x)}\psi(t)\right),
\end{align*}
where $v_t = {}^{x_t} v$ for some $x_t$ such that ${}^{x_t} s= t$.
From the orthogonality relations the sum $\sum_{\psi \in \Irr(C_A(a^x))} \overline{\psi(b^x)} \psi(t)$ is zero unless $t$ and $b^x$ are conjugate under $C_A(a^x)$, in which case it equals $|C_{C_A(a^x)}(t)|$. In particular $F_{[b,\phi]} (sv) = 0$ if $s$ is not $G$-conjugate to $b$, which proves~(i).

Let us now consider the case where $s=b$. In this case the previous equality becomes
\begin{equation*}
F_{[b,\phi]} (bv) = \frac{1}{|A|}\sum_{x \in A}\sum_{a \in C_A(b)}\sum_{\substack{t \in C_A(a^x) \\ t \sim_{C_A(a^x)} b^x}} \frac{ |C_{C_A(a^x)}(t)|}{|C_A(a^x)||C_A(b)| } \phi(a)\varepsilon(t) \Gamma_{u_{a^x}}^{C_G(t)} (v_t).
\end{equation*}
Note that we have $C_A(a^x) = C_A(a)^x$ and that if $r \in C_A(a)$ then $r^x \sim_{C_A(a^x)} b^x$ if and only if $r \sim_{C_A(a)} b$. Now if $t = r^x \in C_A(a^x)$ and $y \in C_A(a)$ is an element such that ${}^yb = r$ then we have ${}^{x^{-1}y}b = r^x = t$. Hence, if $t = r^x \in C_A(a^x)$ contributes to the sum above then we may assume that $x_t = x^{-1}y_r$ where $y_r \in C_A(a)$ satisfies ${}^{y_r}b = r$. Putting this together the previous equality becomes
\begin{equation*}
F_{[b,\phi]} (bv) = \frac{1}{|A|}\sum_{x \in A}\sum_{a \in C_A(b)}\sum_{\substack{r \in C_A(a) \\ r \sim_{C_A(a)} b}} \frac{ |C_{C_A(a)}(r)|}{|C_A(a)||C_A(b)| } \phi(a)\varepsilon(r) {}^{y_r^{-1}x}\Gamma_{u_{x^{-1}ax}}^{C_G(r^x)} (v).
\end{equation*}
Note that $\varepsilon$ is a class function on $A$, see \cref{lem:weil-signs}.

Recall from \cref{lem:admissiblerep} that $u_{x^{-1}ax} = x^{-1}u_ax$ for any $x,a \in A$ so by (ii) of \cref{lem:induction-GGGC} we get that
\begin{equation*}
\Gamma_{u_{x^{-1}ax}}^{C_G(r^x)} = \Gamma_{x^{-1}u_ax}^{C_G(r^x)} = {}^{x^{-1}}\Gamma_{u_a}^{C_G(r)}.
\end{equation*}
An entirely analogous argument shows that ${}^{y_r^{-1}}\Gamma_{u_a}^{C_G(r)} = \Gamma_{u_a}^{C_G(b)}$ because $y_r \in C_A(a)$ and ${}^{y_r}b = r$. Therefore, we get that
\begin{align*}
F_{[b,\phi]} (bv) &= \frac{1}{|A|}\sum_{x \in A}\sum_{a \in C_A(b)}\sum_{\substack{r \in C_A(a) \\ r \sim_{C_A(a)} b}} \frac{ |C_{C_A(a)}(b)|}{|C_A(a)||C_A(b)| } \phi(a)\varepsilon(b) \Gamma_{u_a}^{C_G(b)} (v)\\
&= \frac{\varepsilon(b)}{|C_A(b)|}\sum_{a \in C_A(b)} \phi(a)\Gamma_{u_a}^{C_G(b)} (v).
\end{align*}
This proves (ii).
\end{proof}

\subsection{Translation by Central Elements}\label{subsec:translation}

It will be convenient to formulate a version of \cref{prop:fourier-values} in terms of the Mellin transforms of GGGCs. For that purpose we introduce the translation operator. 

\begin{definition}
Given a central function $f$ on a finite group $H$ and $z \in Z(H)$ we denote by $z \cdot f$ the class function on $H$ obtained by translation by $z$. More precisely, $z \cdot f$ satisfies $ (z \cdot f) (h) = f(zh)$ for all $g \in G$. 
\end{definition}

\begin{rmk}
Since multiplication by $z$ is an $H$-equivariant bijection of $H$ the corresponding translation by $z$ is an isometry of $\Class(H)$, so if $f, f' \in \Class(H)$ then $\langle z\cdot f , f' \rangle_H = \langle f , z^{-1}\cdot f' \rangle_H$.
\end{rmk}

If we define the following class function on $C_G^\circ(b) := C_{\bG}^{\circ}(b)^F$
$$\Gamma_{(b,\phi)} = \frac{1}{|C_A(b)|} \sum_{a \in C_A(b)}\phi(a) \Gamma_{u_a}^{C_G^\circ(b)}$$
then \cref{prop:fourier-values} can be stated as follows.

\begin{corollary}\label{cor:fourier-values}
Under the assumptions of \cref{prop:fourier-values}
$$ F_{[b,\phi]} = \varepsilon(b)\Ind_{C_G^\circ(b)}^G \big( b^{-1} \cdot \Gamma_{(b,\phi)} \big).$$
\end{corollary}

\begin{proof} Since GGGCs are unipotently supported, the class function $b^{-1} \cdot \Gamma_{(b,\phi)}$ can take non-zero values only at elements of the form  $b v$ where $v \in \mathcal{U}(C_{\bG}^\circ(b))^F$ is unipotent. Furthermore, if $x \in G$ is such that ${}^x (bv) = bv'=v'b$ with $v'$ unipotent then $x \in C_G(b)$, which shows that 
$$\Ind_{C_G^\circ(b)}^G \big( b^{-1} \cdot \Gamma_{(b,\phi)} \big)(bv) = \frac{1}{|C_G^\circ(b)|} \sum_{x \in C_G(b)}  \Gamma_{(b,\phi)} ({}^x v) = (\Ind_{C_G^\circ(b)}^{C_G(b)} \Gamma_{(b,\phi)})(v).$$
The result follows from \cref{prop:fourier-values}.
\end{proof}

When $H:=\bfH^F$ is a finite reductive group, the translation operator is compatible with Alvis--Curtis duality $D_\bfH$. We refer the reader to \cite[Chap.~8]{DiMibook} for a definition of this operator.

\begin{lemma}\label{lem:translationandduality} Let $\bfH$ be a connected reductive group with Frobenius endomorphism $F$.
If $z \in Z(\bfH)^F = Z(\bH^F)$ then $z \cdot D_\bfH(f) = D_\bfH(z\cdot f)$ for every class function $f$ on $H$.
\end{lemma}

\begin{proof} Since $z$ is central, it is contained in any maximal torus of $\bfH$, therefore it lies in the centre of any standard Levi subgroup $\bfL$ of $\bfH$. But the translation by any element of the centre of $\bfL$ commutes with the operator $R_{L}^{H} \circ {}^*R_{L}^{H}$ on class functions, see \cite[10.2, 10.3]{Bon06}. Therefore the translation by $z$ commutes with $D_\bfH$. 
\end{proof}

We will be particularly interested in translations by semisimple elements $s \in G$ in their centraliser $C_G^\circ(s)$ or $C_G(s)$ (note that $s \in C_\bfG^\circ(s)$ by \cite[Prop.~2.5]{DiMibook}). Combining the previous results we can compute the scalar product of the Alvis--Curtis dual of $F_{[b,\phi]}$ with any class function on $G$ in terms of duals of GGGCs of $C_G^\circ(b)$.

\begin{corollary}\label{cor:scalarofkawanaka}
Under the assumptions of \cref{prop:fourier-values}, given any class function $f \in \Class(G)$ we have
$$\begin{aligned}
\big\langle D_\bfG (F_{[b,\phi]}) , f\big\rangle_G &\, =  \pm \varepsilon(b)\big\langle D_{C_\bfG^\circ(b)}(\Gamma_{(b,\phi)}),b \cdot \Res_{C_G^\circ(b)}^{G}(f) \big\rangle_{C_G^\circ(b)} \\
&\, =  \pm \frac{\varepsilon(b)}{|C_A(b)|} \sum_{a \in C_A(b)}\phi(a)  \big\langle D_{C_\bfG^\circ(b)}(\Gamma_{u_a}^{C_G^\circ(b)}),b \cdot \Res_{C_G^\circ(b)}^{G}(f) \big\rangle_{C_G^\circ(b)}.\end{aligned}$$

\end{corollary}

\begin{proof}Let $f$ be a class function on $G$. Using adjunction and \cref{cor:fourier-values} we get
$$ \big\langle D_\bfG (F_{[b,\phi]}) , f\big\rangle_G  = \varepsilon(b)\big\langle \Gamma_{[b,\phi]} , b \cdot \Res_{C_G^\circ(b)}^{G}(D_\bG(f)) \rangle_{C_G^\circ(b)}.$$
Since  $\Gamma_{[b,\phi]}$ is unipotently supported, we are only interested in the value of the class function $\Res_{C_G^\circ(b)}^{G}(D_\bG(f))$ at elements of the form $b^{-1} v$ where $v \in \mathcal{U}(C_{\bG}^\circ(b))^F = \mathcal{U}(C_{\bG}^\circ(b^{-1}))^F$. By \cite[Cor.~8.16]{DiMibook} we have in that case
$$\Res_{C_G^\circ(b)}^{G} D_\bG(f) (b^{-1} v) = \pm D_{C_\bG^\circ(b)} (\Res_{C_G^\circ(b)}^{G} f) (b^{-1} v)$$
and the result follows from \cref{lem:translationandduality}.
\end{proof}

\subsection{Restriction of Character Sheaves}\label{ssec:charsheaves}

We want to use \cref{cor:scalarofkawanaka} when $f$ is the characteristic function of a character sheaf. To this end we recall recent results by Lusztig on the restriction of such sheaves \cite{Lu15}.

Let us denote by $\USh(\bG)$ the set of (isomorphism classes of) unipotent character sheaves on $\bG$. We have a partition
\begin{equation*}
\USh(\bG) = \bigsqcup_{\mathscr{F} \in \Fam(W)} \USh(\mathscr{F})
\end{equation*}
indexed by the families of $W$, see \cite[III, Cor.~16.7]{Lu86}. The character sheaves in $\USh(\mathscr{F})$ have a \emph{unipotent support}, which turns out to be the class $\mathcal{O}_{\mathscr{F}}$. The intrinsic properties characterising the class $\mathcal{O}_{\mathscr{F}}$ as the unipotent support are as follows.

Given a conjugacy class $\mathcal{O}$ of $\bfG$, we denote by $\mathcal{O}_{\uni}$ the unipotent conjugacy class consisting of the unipotent parts of the elements in $\mathcal{O}$. It is shown in \cite[Thm.~10.7]{Lu92}, see also \cite[Thm.~13.8]{Tay16} for the extension to good characteristic, that the following two properties hold:
\begin{itemize}
 \item if $\mathcal{F} \in \USh(\mathscr{F})$ and $\mathcal{O}$ is a conjugacy class then $\mathcal{F}_{|\mathcal{O}} = 0$ if $\dim \mathcal{O}_{\uni} > \dim \mathcal{O}_{\mathscr{F}}$ or if $\dim \mathcal{O}_{\uni} = \dim  \mathcal{O}_{\mathscr{F}}$ and $\mathcal{O}_{\uni} \neq \mathcal{O}_{\mathscr{F}}$;
 \item there exists a conjugacy class $\mathcal{O} \subseteq \bG$, with $\mathcal{O}_{\uni} = \mathcal{O}_{\mathscr{F}}$, and a unipotent character sheaf $\mathcal{F}' \in \USh(\mathscr{F})$ such that $\mathcal{F}'_{|\mathcal{O}} \neq 0$.
\end{itemize}

A priori it is not clear from the definition that given a unipotent character sheaf $\mathcal{F} \in \USh(\mathscr{F})$  there exists a class $\mathcal{O} \subseteq \bG$ such that $\mathcal{O}_{\uni} = \mathcal{O}_\mathscr{F}$ and $\mathcal{F}_{|\mathcal{O}} \neq 0$. However, this ambiguity has recently been resolved by Lusztig \cite{Lu15}, giving a geometric interpretation of the parametrisation of $F$-stable unipotent character sheaves which parallels that of unipotent characters. We recall his result here.

Assume $\mathscr{F} \in \Fam(W)^F$ is a family and let $u \in \mathcal{O}_\mathscr{F}$ be a unipotent element. Recall that $\bar A_{\mathscr{F}}$ denotes Lusztig's canonical quotient of $A_\bG(u)$. Let $\mathcal{O}$ be a conjugacy class with $\mathcal{O}_{\uni} = \mathcal{O}_{\mathscr{F}}$ and let $\mathcal{F}\in \USh(\mathscr{F}) $ be a character sheaf such that $\mathcal{F}_{|\mathcal{O}} \neq 0$. Pick an element $b \in C_{\bG}(u)$ such that $\mathcal{O} = \Cl_\bG(bu)$. This semisimple element $b$ is unique up to conjugation by $C_\bfG(u)$. Moreover, if $\mathcal{O}' = \Cl_\bG(b'u)$ is another class with $b' \in C_\bG(u)$ such that $\mathcal{F}_{|\mathcal{O}'} \neq 0$ then the images of $b$ and $b'$ in $\bar A_{\mathscr{F}}$ are conjugate under $\bar A_{\mathscr{F}}$.

As in \cref{subsec:translation} one can consider the translation by $b$ on $C_\bG^\circ(b)$. We will denote by $b^*$ the pull-back of $C_\bG(b)$-equivariant sheaves on $C_\bG^\circ(b)$ along this map. If $\mathcal{C} = \Cl_{C_\bG(b)}(u)$ then it is shown in \cite[\S 1.7]{Lu15} that
\begin{equation}\label{eq:restriction-of-sheaf}
b^*( \mathcal{F}_{|b\mathcal{C}}) \simeq \mathcal{E}[-\dim \mathcal{O} - \dim Z(\bL)]
\end{equation}
for some (not-necessarily irreducible) $C_\bfG(b)$-equivariant local system $\mathcal{E}$ on $\mathcal{C}$. Here $\bfL$ is a Levi subgroup of $\bG$ attached to the cuspidal support of $\mathcal{F}$.

Denote by $\bar b$ the image of $b$ in $\bar A_{\mathscr{F}}$. It is explained in \cite[\S 2.3]{Lu15} that the local system $\mathcal{E}$ is obtained from an irreducible representation $\psi$ of $C_{\bar A_{\mathscr{F}}}(\bar b)$ through the natural map $A_{C_\bfG(b)}(u) = A_\bG(bu) \longrightarrow C_{\bar A_{\mathscr{F}}}(\bar b)$.
Note that $(\bar b,\psi)$ is well-defined up to $\bar A_{\mathscr{F}}$-conjugation, and in fact it characterises the character sheaf $\mathcal{F}$, see \cite[Thm.~2.4]{Lu15}.

\begin{theorem}[Lusztig]\label{thm:lusztig}
The map $\mathcal{F} \mapsto [\bar b,\psi]$ induces a bijection between $\USh(\mathscr{F})$ and the elements of $\mathscr{M}(\bar A_{\mathscr{F}})$.
\end{theorem}

Let us draw the consequences of this result on the values of the characteristic function of $\mathcal{F}$. We assume now that: $\mathscr{F} \in \Fam(W)^F$ is $F$-stable, $u \in \mathcal{O}_{\mathscr{F}}^F$ is $F$-fixed, and that $F$ acts trivially on $\bar A_{\mathscr{F}}$. This last assumption is satisfied, for instance, when $\bfG$ is simple, see \cite[Prop.~2.4]{Tay13}. With these assumptions it follows from \cref{thm:lusztig} that every character sheaf with unipotent support $\mathcal{O}_{\mathscr{F}}$, i.e., any character sheaf contained in $\USh(\mathscr{F})$, is $F$-stable. The choice of an isomorphism between $\mathcal{F}$ and $F^*\mathcal{F}$ defines a class function $\chi_{\mathcal{F}}$ on $G$.

Since $\mathcal{F}$ is a simple perverse sheaf, this class function is well-defined up to a scalar. We can normalise it so that if $\mathcal{F} = \mathcal{F}_{[\bar b,\psi]}$ is the character sheaf corresponding to $[\bar b,\psi] \in \mathscr{M}(\bar A_{\mathscr{F}})$ then for every $x \in C_\bfG^\circ(b)$ such that ${}^x(bu)$ is $F$-stable, there exists $d_{\mathcal{F}} \in \frac{1}{2} \mathbb{Z}$ such that we have 
\begin{equation}\label{eq:valuecharsheaf}
\chi_\mathcal{F}\big({}^x(bu)\big) = \chi_\mathcal{F}\big(b ({}^xu)\big) = q^{d_{\mathcal{F}}}\psi\big(\, \overline{x^{-1}F(x)}\, \big)
\end{equation}
where $\overline{x^{-1}F(x)}$ is the image of $x^{-1}F(x)$ in $C_{\bar A}(\bar b)$. Indeed, $x^{-1}F(x)$ centralises both $u$ and $b$, therefore its image in $\bar A_{\mathscr{F}}$ centralises $\bar b$, the image of $b$. Furthermore, using the normalisation in \cite[V, \S 25.1]{Lu86} and \cref{eq:restriction-of-sheaf} the value of $d_\mathcal{F}$ is given by
\begin{equation}\label{eq:valueofd}
d_\mathcal{F} = \frac{1}{2} \big(\dim\bG- \dim \Cl_\bG(bu) -\dim Z(\bL) \big) = \frac{1}{2} \big(\dim C_\bG(bu)-\dim Z(\bL) \big).   
\end{equation}

Finally, let us note that in \cite[\S1.7]{Lu15} it is also shown that the irreducible summands of the local system $\mathcal{E}$ in \cref{eq:restriction-of-sheaf}, considered as a $C_\bG^\circ(b)$-equivariant local system, lie in a Springer series attached to a Levi subgroup of the form $C_{\bL^x}^{\circ}(b)$ of $C_{\bG}^\circ(b)$ where $x \in \bG$ is such that $b$ is an isolated element of $\bL^x$. Consequently $Z^{\circ}(\bL^x) = Z^{\circ}( C_{\bL^x}^{\circ}(b))$ so that 
\begin{equation}\label{eq:valueofdbis}
d_\mathcal{F} = \frac{1}{2} \big(\dim C_\bG(bu)-\dim Z( C_{\bL^x}^\circ(b)) \big).   
\end{equation}

\subsection{Projection on the Family} 
\begin{assumption}
From now until the end of this section we fix an $F$-stable family $\mathscr{F} \in \Fam(W)^F$ and a unipotent element $u \in \mathcal{O}_{\mathscr{F}}^F$. We assume that $(A,\lambda)$ is an admissible covering for the canonical quotient $\bar{A}_{\mathscr{F}}$ of $A_{\bG}(u)$, see \cref{def:admissiblesplitting}, and that the Kawanaka characters are defined for $\mathcal{O}_{\mathscr{F}}$, with respect to $(A,\lambda)$.
\end{assumption}

Associated to the $F$-stable unipotent class $\mathcal{O}_{\mathscr{F}}$ we get from \cref{ssec:charsheaves} a family of class functions on $G$ parametrised by $\mathscr{M}(\bar A_{\mathscr{F}})$, given by the characteristic functions of the unipotent character sheaves with unipotent support $\mathcal{O}_{\mathscr{F}}$. On the other hand, we have constructed in \cref{ssec:fourierdef}  a family $\{F_{[b,\Phi]}\}_{[b,\Phi]\in \mathscr{M}(A)}$ of class functions on $G$ parametrised by $\mathscr{M}(A)$ whenever $(A,\lambda)$ is an admissible pair for $u \in \mathcal{O}_{\mathscr{F}}^F$.

Our next goal is to establish the following proposition, which computes the scalar products between these two families. For this we will need certain properties of the Alvis--Curtis dual $D_{\bG}(\Gamma_u)$ of a GGGC. This introduces some restrictions on the size of the field $q$ when $Z(\bG)$ is not connected. We recall that there exists a bound $q_0(\bG)$, depending only on the root system of $\bG$, such that if $q > q_0(\bG)$ then the main results of \cite{Lu90} hold.

If $Z(\bG)$ is connected then by work of Shoji \cite{Sho96} one may take $q_0(\bG) = 1$. In the following we will use properties of GGGCs stated in \cite[\S2]{Ge97}. As mentioned in \cite{Ge97} these properties are known to hold when $q > q_0(\bG)$ and $p$ is good by results of Lusztig \cite{Lu92}, with extensions to good characteristic provided in \cite{Tay16}.

\begin{proposition}\label{prop:fouriermultiplicities}
Recall that $(A,\lambda)$ is an admissible covering for $\bar{A}_{\mathscr{F}}$. For any $a \in A$ we denote by $\bar{a} \in \bar{A}_{\mathscr{F}}$ the image of $a$ under the map $A \to A_{\bG}(u) \to \bar{A}_{\mathscr{F}}$. To $[a,\psi] \in \mathscr{M}(\bar A_{\mathscr{F}})$ and $[b,\phi] \in \mathscr{M}(A)$ we associate:
\begin{itemize}
 \item the characteristic function $\chi_{[a,\psi]}$ of the unipotent character sheaf associated to $[a,\psi]$, normalised so that \eqref{eq:valuecharsheaf} holds,
 \item the Fourier transform $F_{[b,\phi]}$ defined in \cref{ssec:fourierdef}.
\end{itemize}
If $q > q_0(\bG)$ then the following hold: 
\begin{enumerate}
 \item if $a$ and $\bar b$ are not conjugate under $\bar A_{\mathscr{F}}$, then $\langle F_{[b,\phi]}, D_\bfG(\chi_{[a,\psi]}) \rangle_G = 0$,
 \item if $a = \bar b$ then 
 $$ \langle F_{[b,\phi]}, D_\bfG(\chi_{[a,\psi]}) \rangle_{G} = \varepsilon(b)\zeta \langle \phi, \widetilde \psi \rangle_{C_A(b)}$$
 for some root of unity $\zeta$, where $\widetilde \psi$ is the irreducible character of $C_A(b)$ obtained by inflation of $\psi$ through the map $C_A(b) \twoheadrightarrow C_{\bar A}(\bar b) = C_{\bar A}(a) $. 
\end{enumerate} 
\end{proposition}

\begin{proof}
By \cref{cor:scalarofkawanaka}, the scalar product between $F_{[b,\phi]}$ and $D_\bfG(\chi_{[a,\psi]})$ is given by
$$\langle F_{[b,\phi]}, D_\bfG(\chi_{[a,\psi]}) \rangle_{G} = \pm \varepsilon(b)\big\langle D_{C_\bfG^\circ(b)}(\Gamma_{(b,\phi)}), b \cdot \Res_{C_G^\circ(b)}^{G} \chi_{[a,\psi]} \big\rangle_{C_G^\circ(b)}.$$
With our assumption that $q > q_0(\bG)$ it is known that the class function $D_{C_\bfG^\circ(b)}(\Gamma_{(b,\phi)})$ is unipotently supported and 
$D_{C_\bfG^\circ(b)}(\Gamma_{(b,\phi)})(v) \neq 0$ forces the $C_\bfG^\circ(b)$-conjugacy class of $u$ to lie in the closure of the $C_\bfG^\circ(b)$-conjugacy class of $v$, see \cite[2.4(c)]{Ge97}. On the other hand, by definition of the unipotent support of character sheaves, if  $v \in C_G^\circ(b)$ is such that $\chi_{[a,\psi]}(bv) \neq 0$ then $\dim \Cl_\bG(v) < \dim \Cl_\bG(u)$ or $\Cl_\bG(u) = \Cl_\bG(v)$.

Therefore, for the product $D_{C_\bfG^\circ(b)}(\Gamma_{(b,\phi)})(v) \overline{\chi_{[a,\psi]}(bv)}$ to be non-zero we need to have that $\Cl_\bG(u) = \Cl_\bG(v)$, and that the $C_\bfG^\circ(b)$-conjugacy class of $u$ lies in the closure of the $C_\bfG^\circ(b)$-conjugacy class of $v$. By \cref{thm:fusion-unip-classes}, this forces the $C_\bfG^\circ(b)$-conjugacy classes of $u$ and $v$ to coincide. Consequently, 
\begin{equation}\label{eq:firsteq}
\langle F_{[b,\phi]}, D_\bfG(\chi_{[a,\psi]}) \rangle_{G} = \pm \frac{\varepsilon(b)}{|C_G^\circ(b)|} \sum_{\begin{subarray}{c}v \in \mathcal{U}(C_\bG^\circ(b))^F \\[3pt] v \sim_{C_\bfG^\circ(b)} u \end{subarray}} D_{C_\bfG^\circ(b)}(\Gamma_{(b,\phi)})(v) \overline{\chi_{[a,\psi]}(bv)}.
\end{equation}
Now for $v$ conjugate to $u$ under $C_\bfG^\circ(b)$ the $\bfG$-conjugacy classes of $bv$ and $bu$ coincide. Therefore from \cref{thm:lusztig} we deduce that $\chi_{[a,\psi]}(bv) = 0$ unless the image $\bar b$ of $b$ in $\bar A_{\mathscr{F}}$ is conjugate to $a$. This proves (i).

Now assume that $a = \bar b$.  Let us consider the map
$$ \beta : A_{C_\bfG^\circ(b)}(u) \to C_{\bar A_{\mathscr{F}}}(a).$$
If $c \in C_A(b)$ then $b \in C_A(c)$ and by \cref{en:K2} we have $c \in C_\bG^\circ(C_A(c)) \leqslant C_\bG^\circ(b)$, which shows that $C_A(b) \leqslant C_\bG^\circ(b)$. Now by \cref{rmk:K3-and-centralisers} the map $C_A(b) \to C_{\bar A_{\mathscr{F}}}(\bar b) = C_{\bar A_{\mathscr{F}}}(a)$ is surjective. Since $A \leqslant C_\bG(u)$ then $C_A(b) \leqslant C_{C_\bG^\circ(b)}(u)$ and we deduce that the map $\beta$ is surjective.

We denote by $\beta^* \psi$ the irreducible character of $A_{C_\bfG^\circ(b)}(u)$ obtained from $\psi$ by inflation along $\beta$. Following \cite[(2.2)]{Ge97}, we define the following class function on $C_G^\circ(b)$:
$$Y_{(u,\beta^*\psi)}(g) = \left\{ \begin{array}{ll} \psi \big(\, \overline{x^{-1} F(x)}\,\big) & \text{if $g = xux^{-1}$ for some $x \in C_\bfG^\circ(b)$}, \\
0 & \text{otherwise.} \end{array} \right.$$
where $\overline{x^{-1} F(x)}$ is the image of $x^{-1} F(x)$ in $C_{\bar A_{\mathscr{F}}}(\bar b) = C_{\bar A_{\mathscr{F}}}(a)$. Note that $F$ acts trivially on $\bar A_{\mathscr{F}}$ so that one can extend $\beta^* \psi$ trivially to $A_{C_\bfG^\circ(b)}(u) \rtimes \langle F\rangle$, in which case $Y_{(u,\beta^*\psi)}$ agrees with the definition given in \cite[(2.2)]{Ge97}. The class function $Y_{(u,\beta^*\psi)}$ is also the (normalised) characteristic function of the irreducible local system corresponding to $\beta^* \psi$.

From the equations  \eqref{eq:valuecharsheaf} and \eqref{eq:firsteq} we deduce that
\begin{equation}\label{eq:second}
\langle F_{[b,\phi]}, D_\bfG(\chi_{[a,\psi]}) \rangle_{G} = \pm \varepsilon(b)q^{d_{\mathcal{F}}} \langle D_{C_\bfG^\circ(b)}(\Gamma_{(b,\phi)}), Y_{(u,\beta^* \psi)} \rangle_{C_G^\circ(b)},
\end{equation}
where $\mathcal{F}$ is the character sheaf associated to $[a,\psi]$. By \cite[(2.4) and (2.3.c)]{Ge97}, there is a root of unity $\zeta$ and a half-integer $d$ (both depending on the local system associated to $(u,\beta^*\psi)$) such that
for every $c \in C_A(b)$ we have
$$\
\langle D_{C_\bfG^\circ(b)}(\Gamma_{u_c}^{C_G^\circ(b)}), Y_{(u,\beta^* \psi)} \rangle_{C_G^\circ(b)}  = \zeta q^{-d} \overline{Y_{(u,\beta^* \psi)}(u_c)}= \zeta q^{-d} \overline{\psi(\overline c)}.$$
By \cite[(2.2.b)]{Ge97} the explicit value of $d$ is given by 
$$ d= \frac{1}{2} (\dim C_\bG(su) - \dim Z(\bM))$$
for $\bM$ a Levi subgroup of  $ C_\bG^\circ(b)$ attached to the cuspidal support of the pair $(u,\beta^* \psi)$ in the generalised Springer correspondence. We actually have $d = d_{\mathcal{F}}$ by \cref{eq:valueofdbis}.

Using the definition of $\Gamma_{(b,\phi)}$ in terms of the various generalised Gelfand-Graev characters $\Gamma_{u_c}^{C_G^\circ(b)}$ for $c \in C_A(b)$ we get
$$ \begin{aligned} \langle D_{C_\bfG^\circ(b)}(\Gamma_{[b,\phi]}), Y_{(u,\beta^*\psi)} \rangle_{C_G^\circ(b)} 
& \, = \frac{1}{|C_A(b)|} \sum_{c\in C_A(b)} \phi(c) \langle D_{C_\bfG^\circ(b)}(\Gamma_{u_c}^{C_G^\circ(b)}), Y_{(u,\beta^*\psi)}  \rangle_{C_G^\circ(b)} \\ 
& \, = \frac{\zeta q^{-d}}{|C_A(b)|} \sum_{c\in C_A(b)} \phi(c)   \overline{\psi (\bar c)} \\
& \, =  \zeta q^{-d} \langle \phi, \widetilde \psi \rangle_{C_A(b)} \end{aligned}$$
where $\widetilde \psi$ is the inflation of $\psi$ through the map $C_A(b) \twoheadrightarrow C_{\bar A_{\mathscr{F}}}(\bar b) = C_{\bar A_{\mathscr{F}}}(a) $. Then (ii) follows from \eqref{eq:second}.
\end{proof}

\subsection{Kawanaka's Conjecture}
In the previous section we computed the multiplicity of the characteristic functions of character sheaves with unipotent support $\mathcal{O}_{\mathscr{F}}$ in the Alvis--Curtis dual of the Fourier transforms of Kawanaka characters. We now translate this result to deduce the multiplicity of the unipotent characters with unipotent support $\mathcal{O}_{\mathscr{F}}$ in the Alvis--Curtis dual of the Kawanaka characters.

As was mentioned in \cref{ssec:unipotentchar} we have an involutive permutation of the families given by tensoring with the sign character of $W$. At the level of unipotent classes this corresponds to Spaltenstein duality, see \cite[Cor.~3.5]{BaVo85} and \cite[Chp.~3]{SpBk}. On unipotent characters we have an involutive permutation induced by Alvis--Curtis duality. For any $F$-stable family $\mathscr{F} \in \Fam(W)^F$ we have a bijection $\Uch(\mathscr{F}) \to \Uch(\mathscr{F}^*)$ given by $\rho \mapsto \pm D_{\bG}(\rho)$, where the sign is the unique choice making $\pm D_{\bG}(\rho)$ a character.

As explained in \cite[13.1.3]{Lu84} the irreducible characters in $\Uch(\mathscr{F}^*)$ are parameterised by $\mathscr{M}(\bar{A}_{\mathscr{F}})$. The following theorem proves the conjecture stated by Kawanaka in \cite[2.4.5]{Ka87} under the assumption that an admissible covering of Lusztig's canonical quotient exists.

\begin{theorem}\label{thm:kawanakaconj}
Recall that $\mathscr{F} \in \Fam(W)^F$ is a family and $(A,\lambda)$ is an admissible covering of Lusztig's canonical quotient $\bar{A}_{\mathscr{F}}$ of $A_{\bG}(u)$ with $u \in \mathcal{O}_{\mathscr{F}}^F$. We will additionally assume that either:
\begin{enumerate}
 \item $A$ is abelian,
 \item or $A \cong \bar A_{\mathscr{F}}$.
\end{enumerate}
If $q > q_0(\bG)$ then given $[a,\psi] \in \mathscr{M}(A)$, the character $K_{[a,\psi]}$ has at most one unipotent constituent in $\Uch(\mathscr{F}^*)$ and it occurs with multiplicity at most one. Furthermore, every unipotent character in $\Uch(\mathscr{F}^*)$ occurs in some $K_{[a,\psi]}$.
\end{theorem}

\begin{proof} We start by recalling from \cref{ssec:kawnakadef} that the generalised Gelfand-Graev character associated to $u_a$ is 
$$\Gamma_{u_a}  = \sum_{\psi \in \Irr(C_A(a))} \psi(1) K_{[a,\psi]}.$$
Since every unipotent character of $\mathscr{F}$ occurs in $\Gamma_{u_a}$ for some $a \in A$ by \cite[Prop.~15.4]{Tay16}, we deduce that every unipotent character of $\mathscr{F}$ occurs in some $K_{[a,\psi]}$. 

Now let $\widehat K_{[a,\psi]}$, resp., $\widehat F_{[b,\phi]}$, be the projection of $K_{[a,\psi]}$, resp., $F_{[b,\phi]}$, on the space spanned by the unipotent characters in $\Uch(\mathscr{F}^*)$. It is also the span of the Alvis--Curtis dual of the characteristic functions of character sheaves with unipotent support $\Cl_\bG(u)$. Recall from \cref{def:admissiblesplitting} that $Z$ denotes the kernel of the surjective map $A \twoheadrightarrow \bar A$. Then \cref{prop:fouriermultiplicities} shows that given $[b,\phi] \in \mathscr{M}( A)$ we have  
\begin{equation}\label{eq:hatF}
\widehat F_{[b,\phi]} = \left\{ \begin{array}{ll} 0 & \text{if $Z \not\subset \Ker(\phi)$}, \\
\varepsilon(b)\zeta_{[b,\phi]} D_G(\chi_{[\bar b,\bar\phi]}) & \text{otherwise}, \end{array}\right.
\end{equation}
where $\zeta_{[b,\phi]}$ is some root of unity and $\bar \phi$ is the irreducible character of $C_{\bar A}(\bar b) \simeq C_A(b)/Z$ induced by $\phi$ (see \cref{rmk:K3-and-centralisers} for the latter isomorphism). 

Let us first assume that $\bar A = A$, i.e., that $Z=1$. Then by \cref{eq:hatF} the family  $\{\widehat F_{[b,\phi]}\}_{[b,\phi] \in \mathscr{M}(A)}$ is an orthonormal family of class functions. Since the Fourier transform on $\mathscr{M}(A)$ is a unitary involution, we get
$$\langle \widehat K_{[a,\psi]} , \widehat K_{[a,\psi]} \rangle_G = \langle \widehat F_{[a,\psi]} , \widehat F_{[a,\psi]} \rangle_G = 1$$
which shows that $\widehat K_{[a,\psi]}$ is an irreducible unipotent character.

Now assume that $A$ is abelian. In that case the class function  $K_{[a,\psi]}$ can be written
$$\begin{aligned} 
 \widehat K_{[a,\psi]} & =  \frac{1}{|A|}\sum_{(b, \phi)\in \mathscr{M}(A)} \psi(b) \overline{\phi(a)}  \widehat F_{[b,\phi]} \\ 
& = \frac{1}{|A|}\sum_{\begin{subarray}{c} (b, \phi) \in \mathscr{M}(A) \\ 
Z \subset \Ker(\phi) \end{subarray}} \psi(b) \overline{\phi(a)}  \varepsilon(b)\zeta_{[b,\phi]} D_G(\chi_{[\bar b,\bar\phi]}).
\end{aligned}
$$
We deduce that the inner product of $\widehat K_{[a,\psi]}$ with itself is given by
$$\begin{aligned} 
 \langle \widehat K_{[a,\psi]} , \widehat K_{[a,\psi]} \rangle_G & = \frac{1}{|A|^2}\sum_{\begin{subarray}{c} (b, \phi), (b', \phi') \in \mathscr{M}(A)  \ \\ 
Z \subset \Ker(\phi) \cap  \Ker(\phi') \\
(\bar b,\bar \phi) = (\bar b', \bar \phi') \end{subarray}} \psi(b)\overline{\psi(b')}\overline{\phi(a)}\phi'(a)  \varepsilon(b)\varepsilon(b')\zeta_{[b,\phi]} \zeta_{[b',\phi']}^{-1}\\
& =  \frac{1}{|A|^2}\sum_{\begin{subarray}{c} [b, \phi] \in \mathscr{M}(A)  \ \\
z \in Z \\ 
Z \subset \Ker(\phi) \end{subarray}} \psi(z)  \varepsilon(b)\varepsilon(bz)\zeta_{[b,\phi]} \zeta_{[bz,\phi]}^{-1}\\ \end{aligned}$$
where we have used that the irreducible characters of $A$ are linear since $A$ is abelian. Each term in the previous sum is a root of unity, and there are $|A| |A/Z| |Z| = |A|^2$ terms in the sum. Therefore 
 $\langle \widehat K_{[a,\psi]} , \widehat K_{[a,\psi]} \rangle_G \leq 1$, and at most one unipotent character occurs in  $\widehat K_{[a,\psi]}$, and it occurs with multiplicity at most one.
\end{proof}

\begin{rmk} If one can prove that  $\zeta_{[b,\phi]} =  \zeta_{[bz,\phi]}$ and $\varepsilon(b)=\varepsilon(bz)$ for every $z \in Z$ then one does not need the extra assumption on $A$ given in \cref{thm:kawanakaconj}. Furthermore in that case one can show that $\widehat K_{[a,\psi]}$ is either zero if $Z \nsubseteq \Ker(\psi)$ or is a single unipotent character otherwise. 
\end{rmk}

\part{Admissible coverings for simple groups}

\section{Admissible Coverings for Classical Groups}\label{ssec:symplectic}
\begin{assumption}
In this section we assume that $\Char(K) = p \neq 2$.
\end{assumption}

\subsection{Notation for classical groups}\label{subsec:notation-classical}
Assume $\mathbb{F}$ is a field with $\Char(\mathbb{F}) \neq 2$ and $V$ is a finite dimensional $\mathbb{F}$-vector space then we denote by $\GL(V)$ the general linear group of invertible $\mathbb{F}$-linear transformations $V \to V$. As usual $\SL(V) = \{f \in \GL(V) \mid \det(f) = 1\} \leqslant \GL(V)$ denotes the special linear group. If $g \in \GL(V)$ and $\eta \in \mathbb{F}^{\times}$ then we denote by $V_{\eta}(g) = \{v \in V \mid gv = \eta v\}$ the $\eta$-eigenspace of $g$.

We have an action of $\GL(V)$ on the bilinear forms $\mathcal{B} : V \times V \to \mathbb{F}$ defined by ${}^g\mathcal{B}(v,w) = \mathcal{B}(gv,gw)$ for all $g \in \GL(V)$ and $v,w \in V$. The stabiliser of $\mathcal{B}$ under this action is denoted by $\GL(V\mid\mathcal{B})$. We also set $\SL(V\mid\mathcal{B}) = \SL(V) \cap \GL(V\mid\mathcal{B})$. If $\mathcal{B}$ is a non-degenerate symmetric bilinear form then we also write $\Or(V) = \Or(V\mid\mathcal{B}) = \GL(V\mid\mathcal{B})$ and $\SO(V) = \SO(V\mid\mathcal{B}) = \SL(V\mid\mathcal{B})$ for the corresponding orthogonal and special orthogonal groups. If $\mathcal{B}$ is a non-degenerate alternating bilinear form then we denote by $\Sp(V) = \Sp(V \mid\mathcal{B}) = \SL(V \mid\mathcal{B}) = \GL(V \mid\mathcal{B})$ the corresponding symplectic group.

\subsection{Partitions and unipotent classes of classical groups}\label{subsec:combinatorics}
Recall that a \emph{partition} is a (possibly empty) finite sequence $(\mu_1,\dots,\mu_r)$ of positive integers $\mu_1 \geqslant \cdots \geqslant \mu_r > 0$. If $\mu$ is empty then we define $|\mu| = 0$ otherwise we define $|\mu| = \mu_1+\cdots+\mu_r$. If $N \geqslant 0$ is an integer then we denote by $\mathcal{P}(N)$ the set of all \emph{partitions of $N$}, i.e., the set of all partitions $\mu$ satisfying $|\mu|=N$. If $V$ has dimension $N$ then the unipotent classes of $\GL(V)$ are labelled by  $\mathcal{P}(N)$ according to the Jordan type of the unipotent elements.

The class corresponding to $\mu \in \mathcal{P}(N)$ will be denoted by $\mathcal{O}_\mu$. For example, we have $\mathcal{O}_{(1^N)}$ is the trivial class $\{1\}$ and $\mathcal{O}_{(N)}$ is the regular unipotent  class. We recall that there is a natural partial order $\unlhd$ on the set $\mathcal{P}(n)$ given by the \emph{dominance} ordering. Under the above bijection we have $\unlhd$ coincides with the partial order $\preceq$ on unipotent classes introduced in \cref{sec:introduction}. In other words,
$$ \mu \unlhd \nu \iff \mathcal{O}_\mu \subseteq  \overline{\mathcal{O}_\nu}.$$

Given a partition $\mu$ of $N$ and $m \in \{1,\ldots,N\}$ we will write $r_m(\mu) := \#\{1 \leqslant j \leqslant r \mid \mu_j = m\}$ for the number of parts of $\mu$ equal to $m$. 
If the partition in question is clear then we simply write $r_m$ instead of $r_m(\mu)$, so that
$\mu$ can be written $\mu = (N^{r_N}, \dots ,2^{r_2}, 1^{r_1})$. We denote by $\mathcal{P}_{\varepsilon}(N) \subseteq \mathcal{P}(N)$ the set of partitions $\mu$ for which $r_m \equiv 0 \pmod{2}$ for all $1 \leqslant m \leqslant N$ such that $(-1)^m = \varepsilon$. The partition $\mu \in \mathcal{P}_1(N)$ is said to be \emph{very even} if each part $\mu_i$ is even.

Assume $V$ is endowed with a non-degenerate alternating bilinear form. For any $\mu \in \mathcal{P}_{-1}(N)$ we have $\Sp(V) \cap \mathcal{O}_{\mu} \neq \emptyset$. When $\mathbb{F}$ is algebraically closed the map $\mu \mapsto \Sp(V) \cap \mathcal{O}_{\mu}$ gives a bijection between $\mathcal{P}_{-1}(N)$ and the unipotent classes of $\Sp(V)$. We will denote again by $\mathcal{O}_{\mu}$ the intersection $\Sp(V) \cap \mathcal{O}_{\mu}$ when $\mathbb{F}$ is algebraically closed. Assuming now $V$ is endowed with a non-degenerate symmetric bilinear form then the same statements hold if we replace $\mathcal{P}_{-1}(N)$ with $\mathcal{P}_1(N)$ and $\Sp(V)$ with $\Or(V)$. We again denote by $\mathcal{O}_{\mu}$ the intersection.

Now consider the case of $\SO(V)$. Still assuming $\mathbb{F}$ is algebraically closed we have the intersection $\mathcal{O}_{\mu} := \Or(V) \cap \mathcal{O}_{\mu}$ is a single $\SO(V)$ class unless $\mu$ is very even in which case $\mathcal{O}_{\mu} =\mathcal{O}_{\mu}^+ \sqcup \mathcal{O}_{\mu}^-$ is a disjoint union of two $\SO(V)$ conjugacy classes. We will not need to differentiate between these two classes here. To unify the notation we set $\mathcal{O}_{\mu}^+ = \mathcal{O}_{\mu}^- = \mathcal{O}_{\mu}$ when $\mu$ is not very even.

\subsection{Lifting component groups}\label{subsec:component-groups}

The group $\Sp(V)$ has no torsion primes. Therefore by \cite[Thm.~II.5.8]{SpSt} every finite subgroup of commuting semisimple elements lies in a maximal torus of $\Sp(V)$. An extra assumption is needed for orthogonal groups.

\begin{lem}\label{lem:con-comp-cent}
Assume $\mathcal{B}$ is a non-degenerate symmetric bilinear form on $V$ and let $\bG = \SL(V\mid \mathcal{B}) = \SO(V)$. Furthermore, we assume $\bA = \langle a_1,\dots,a_k\rangle \leqslant \bG$ is a subgroup generated by commuting semisimple elements $a_i$ of $\bG$ satisfying $a_i^2 = \Id_V$. If $V_{-1}(a_i)\cap V_{-1}(a_j) = \{0\}$ for all $i \neq j$ then there exists a torus $\bS \leqslant \bG$ containing $\bA$.
\end{lem}

\begin{proof}
Let $\overline{\bG} = \Or(V)$. By assumption we have a basis $(v_1,\dots,v_n)$ of $V$ such that each $v_i$ is an eigenvector of each $a_j$. If $W = V_1(a_1)\cap \cdots \cap V_1(a_k)$ then our assumptions imply that we have an orthogonal decomposition
\begin{equation*}
V = W \oplus V_{-1}(a_1) \oplus \cdots \oplus V_{-1}(a_k).
\end{equation*}
In particular the restriction of $\mathcal{B}$ to each of these summands is non-degenerate.
Clearly any element $f \in C_{\overline{\bG}}(\bA)$ must preserve this subspace decomposition. So we have
\begin{equation*}
C_{\overline{\bG}}(\bA) = \Or(W) \times \Or(V_{-1}(a_1)) \times \cdots \times \Or(V_{-1}(a_k))
\end{equation*}
and $C_{\bG}^{\circ}(\bA) = C_{\overline{\bG}}^{\circ}(\bA)$ is a corresponding direct product of special orthogonal groups. Now $a_i|_{V_{-1}(a_j)} = (-1)^{\delta_{i,j}}\Id_{V_{-1}(a_j)} \in Z(\SO(V_{-1}(a_j)))$ because $V_{-1}(a_j)$ has even dimension as $a_j \in \SO(V)$. We may therefore pick any maximal torus $\bS$ in $C_{\bG}^\circ(\bA)$.
\end{proof}

\subsection{Nilpotent Orbits}\label{sec:nil-orbit-classical}
For the rest of this section we will assume that $V_0$ is an $\mathbb{F}_q$-vector space of dimension $N$ and $\mathcal{B} : V_0 \times V_0 \to \mathbb{F}_q$ is a non-degenerate bilinear form for which there exists a sign $\varepsilon \in \{\pm 1\}$ such that
\begin{equation}\label{eq:alt-sym-cond}
\mathcal{B}(v,w) = \varepsilon\mathcal{B}(w,v)\qquad\text{for all }v,w \in V_0.
\end{equation}
In other words, $\mathcal{B}$ is alternating if $\varepsilon = -1$ and symmetric if $\varepsilon = 1$. Extending scalars we obtain a $K$-vector space $V = K \otimes_{\mathbb{F}_q} V_0$ and a non-degenerate bilinear form $\mathcal{B} : V \times V \to K$. We set $\bG = \SL(V\mid\mathcal{B})$ and $\overline{\bG} = \GL(V\mid\mathcal{B})$.

The vector space $V$ has a natural Frobenius endomorphism $F : V \to V$ satisfying $F(k\otimes v) = k^q\otimes v$ on simple tensors. We have the fixed point space $V^F$ is an $\mathbb{F}_q$-vector space naturally identified with $V_0$. From $F$ we obtain Frobenius endomorphisms $F :\bG \to \bG$ and $F : \overline{\bG} \to \overline{\bG}$ given by precomposing with the Frobenius on $V$. It is clear that we have $\bG^F \cong \SL(V_0 \mid \mathcal{B})$ and $\overline{\bG}^F \cong \GL(V_0 \mid \mathcal{B})$.

We fix a Springer isomorphism $\phi_{\spr} : \mathcal{U}(\bG) \to \mathcal{N}(\bG)$ which gives a bijection $\mathcal{U}(\bG)/\bG \to \mathcal{N}(\bG)/\bG$, see \cref{thm:springer-morphism}. Recall that we have a bijection $\mathcal{P}_{\varepsilon}(N) \to \mathcal{U}(\bG)/\bG$ defined as in \cref{subsec:combinatorics}. If $\mu \in \mathcal{P}_{\varepsilon}(N)$ and $e \in \phi_{\spr}(\mathcal{O}_{\mu}^{\pm})$ then the Jordan form of $e$ is also described by $\mu$. This follows from \cref{thm:springer-morphism} as there is at least one Springer isomorphism for which this is true. For instance, one can take a Cayley map as in \cite[Thm.~III.3.14]{SpSt}.

Following \cite{Jan04} we describe how to get a representative of the nilpotent orbit $\phi_{\spr}(\mathcal{O}_{\mu})$. Let $\mu = (\mu_1 \geqslant \cdots \geqslant \mu_r > 0)$. We fix a basis
\begin{equation*}
\mathcal{V} = (v_{s,i}^{(m)} \mid 1 \leqslant m \leqslant N\text{ with }r_m \neq 0\text{ and } 1 \leqslant s \leqslant r_m, 1 \leqslant i \leqslant m)
\end{equation*}
of $V$ and set $V_s^{(m)} = \Span_K\{v_{s,i}^{(m)} \mid 1 \leqslant i \leqslant m\}$ and $V^{(m)} = \bigoplus_{s=1}^{r_m} V_s^{(m)}$. If $r_m = 0$ then we implicitly assume that $V_s^{(m)} = V^{(m)} = \{0\}$. We define a bijection $\bar{\phantom{x}} : \mathcal{V} \to \mathcal{V}$ by setting
\begin{equation*}
\bar{v}_{s,i}^{(m)} = \begin{cases}
v_{s,m+1-i}^{(m)} &\text{if }(-1)^m = -\varepsilon\\
v_{r_m+1-s,m+1-i}^{(m)} &\text{if }(-1)^m = \varepsilon.
\end{cases}
\end{equation*}
Note we have $r_m$ is even if $(-1)^m = \varepsilon$. Following \cite[1.11]{Jan04} we may choose the basis $\mathcal{V}$ such that
\begin{equation}\label{eq:basis-prop}
\mathcal{B}(v_{s,i}^{(m)},\bar{v}_{t,j}^{(n)}) = \varepsilon\mathcal{B}(\bar{v}_{t,j}^{(n)},v_{s,i}^{(m)}) = (-1)^{i+1}\delta_{i,j}\delta_{s,t}\delta_{m,n}
\end{equation}
where the indices run over all possible relevant values.

Now, let $e_s^{(m)} \in \lie{gl}(V_s^{(m)})$ be the nilpotent element defined such that
\begin{equation*}
e_s^{(m)}v_{s,i}^{(m)} = \begin{cases}
v_{s,i+1}^{(m)} &\text{if }1 \leqslant i < m\\
0 &\text{if }i = m.
\end{cases}
\end{equation*}
Up to permuting the basis elements, the matrix of $e_s^{(m)}$ is a single Jordan block of size $m$. Let $\lie{g} = \Lie(\bG)$. We have $e_{\mu} = \sum_{m=1}^N\sum_{s=1}^{r_m} e_s^{(m)} \in \lie{g}$ is a nilpotent element whose Jordan normal form is parameterised by $\mu$, see \cite[1.11]{Jan04}.

We now define some elements that will be used below in the construction of admissible splittings. 
Let $\mathcal{I} = \{1\leqslant k \leqslant N \mid r_k \neq 0$ and $(-1)^k =-\varepsilon\}$. For each $k \in \mathcal{I}$ and  $1 \leqslant s \leqslant r_k$ the bilinear form  $\mathcal{B}$ restricts to a non-degenerate bilinear form on $V_s^{(k)}$ by \cref{eq:basis-prop} (with the same $\varepsilon$). Therefore we can form the involution  $a_s^{(k)} \in \prod_{m \in \mathcal{I}}\prod_{t=1}^{r_m} \GL(V_t^{(m)} \mid \mathcal{B}) \leqslant \overline{\bG}$  defined such that
\begin{equation*}
a^{(k)}_s|_{V_t^{(m)}} = (-1)^{\delta_{s,t}\delta_{k,m}}\Id_{V_t^{(m)}}
\end{equation*}
for all $m \in \mathcal{I}$ and $1 \leqslant t \leqslant r_m$. It is clear from the definition that $a_s^{(k)} \in C_{\overline{\bG}}(e_{\mu})$ and $\det(a_s^{(k)}) = (-1)^k$ so $a_s^{(k)} \in \bG$ if and only if $k$ is even, i.e., if and only if $\varepsilon = -1$. Note that these involutions pairwise commute. 

\begin{rem}\label{rem:F-fixed-rep}
As the $\mathbb{F}_q$-subspace $V^F \subseteq V$ contains a basis of $V$ there exists an element $g \in \GL(V)$ such that $g\mathcal{V} \subseteq V^F$ so the bilinear form ${}^g\mathcal{B}$ is $F$-fixed. This implies $h = g^{-1}F(g) \in \overline{\bG} = \GL(V \mid \mathcal{B})$. We have a corresponding Frobenius endomorphism $F' = h\circ F : V \to V$ and the nilpotent element $e_{\mu} \in \lie{g}^{F'}$ is $F'$-fixed.

As $\overline{\bG}^{\circ} = \bG$ we see that by replacing $g$ by $gx$, with $x \in \overline{\bG}$, we may replace $h$ by any element of the coset $h\bG$. With this remark we see that when $\varepsilon = -1$ we may assume $h=1$ because $\overline{\bG} = \bG$. This amounts to saying we may assume $\mathcal{V} \subseteq V^F$. Assume now $\varepsilon = 1$ and $\mu$ is not very even then there exists an odd integer $1 \leqslant k \leqslant N$ with $r_k(\mu) \neq 0$. As $k$ is odd the element $a_1^{(k)} \in C_{\overline{\bG}}(e_{\mu})$ is not contained in $\bG$ so we may assume that either $h = 1$ or $h = a_1^{(k)}$. In particular, as $h \in C_{\overline{\bG}}(e_{\mu})$ we get that $e_{\mu} \in \lie{g}^F$ and each $a_s^{(k)}$ is $F$-stable.
\end{rem}

\subsection{Centraliser in the Levi Factor}\label{sec:admissible-split-classical}
We continue with the setup in \cref{sec:nil-orbit-classical}. Following \cite[3.3, 3.4]{Jan04} we let $\mathcal{V}_i = (v_{s,j}^{(m)} \mid i = 2j-1-m)$, for $i \in \mathbb{Z}$, be part of the basis $\mathcal{V}$ and set $V(i) = \Span_K\mathcal{V}_i$. We then get a grading $V = \bigoplus_{i \in \mathbb{Z}} V(i)$ of $V$ and a corresponding grading $\lie{gl}(V) = \bigoplus_{i \in \mathbb{Z}} \lie{gl}(V,i)$ such that $e_{\mu} \in \lie{gl}(V,2)$. In turn, we get a grading $\lie{g} = \bigoplus_{i \in \mathbb{Z}} \lie{g}(i)$ by setting $\lie{g}(i) = \lie{gl}(V,i) \cap \lie{g}$. It follows from \cite[3.5, 5.4]{Jan04} and \cref{lem:assoc-dynkin-cocharacters} that there exists a cocharacter $\lambda \in \mathcal{D}_{e_\mu}(\bG)^F$ such that $\lie{g}(\lambda,i) = \lie{g}(i)$ for all $i \in \mathbb{Z}$, see also \cref{rem:F-fixed-rep}. If we set
\begin{equation*}
\overline{\bL}(\lambda) = \{g \in \overline{\bG} \mid g(V(i)) = V(i)\text{ for all }i \in \mathbb{Z}\}
\end{equation*}
then $\bL(\lambda) = \overline{\bL}(\lambda) \cap \bG$, see \cite[3.11]{Jan04}.

We now aim to describe more concretely the group $\bL(\lambda)$. By definition, if $g \in \bL(\lambda)$ then we get an element $g_i = g|_{V(i)} \in \GL(V(i))$ for any $i \in \mathbb{Z}$. If $V(i) \neq \{0\}$ then we have $\mathcal{B}(V(i),V(-i)) \neq 0$, see \cite[3.4(1)]{Jan04}, but $\mathcal{B}(V(i),V(i)) = \mathcal{B}(V(-i),V(-i)) = 0$. It follows that $\mathcal{B}$ restricts to a non-degenerate bilinear form on $V(i)\oplus V(-i)$ for any $i > 0$ and we have $g_{-i} = {}^{\text{tr}}g_i^{-1}$ is the inverse transpose of $g_i$, see \cite[4.5(1)]{Jan04}.

With this we see that an element $g \in \bL(\lambda)$ is uniquely determined by its restrictions $g|_{V(i)}$ for all $i \geqslant 0$. In other words, we have an isomorphism
\begin{equation}\label{eq:L-iso}
\begin{aligned}
\bL(\lambda) &\to \SL(V(0)\mid\mathcal{B}) \times \prod_{i > 0} \GL(V(i))\\
g &\mapsto (g|_{V(0)},g|_{V(1)},g|_{V(2)},\dots),
\end{aligned}
\end{equation}
see \cite[4.5]{Jan04}.

Given $1 \leq m \leq N$, we now let $W_m = \bigoplus_{s=1}^{r_m} Kv_{s,1}^{(m)}$ then certainly $\dim(W_m) = r_m$. As in \cite[3.7(4)]{Jan04} the space $W_m$ can be endowed with a non-degenerate bilinear form $\mathcal{B}_m : W_m\times W_m \to K$ defined by setting $\mathcal{B}_m(v,w) = \mathcal{B}(v,e_{\mu}^{m-1}w)$. According to \cite[3.8, Prop.~2]{Jan04} the map
\begin{equation}\label{eq:Ce-iso}
\begin{aligned}
C_{\overline{\bL}(\lambda)}(e_\mu) &\to \prod_{m=1}^N \GL(W_m\mid\mathcal{B}_m)\\
f &\mapsto (f|_{W_1},f|_{W_2},\dots,f|_{W_N})
\end{aligned}
\end{equation}
is an isomorphism of algebraic groups and we have
\begin{equation}\label{eq:Ce-iso-types}
\GL(W_m\mid\mathcal{B}_m) = \begin{cases}
\Or(W_m) &\text{if }(-1)^m = -\varepsilon,\\
\Sp(W_m) &\text{if }(-1)^m = \varepsilon.
\end{cases}
\end{equation}

Recall that $\mathcal{I} = \{1\leqslant k \leqslant N \mid r_k \neq 0$ and $(-1)^k =-\varepsilon\}$. The group $B = \langle a_s^{(k)} \mid k \in \mathcal{I}$ and $1 \leqslant s \leqslant r_k \rangle \leqslant C_{\overline{\bG}}(e_{\mu})$ is an elementary abelian $2$-group. It is clear from the definition that $a_s^{(k)}(V(i)) = V(i)$ for all $i \in \mathbb{Z}$ so $B \leqslant C_{\overline{\bL}(\lambda)}(e_{\mu})$. Moreover, the natural surjective map $C_{\overline{\bL}(\lambda)}(e_{\mu}) \to A_{\overline{\bL}(\lambda)}(e_{\mu}) = C_{\overline{\bL}(\lambda)}(e_{\mu})/C_{\overline{\bL}(\lambda)}^{\circ}(e_{\mu})$ restricts to a surjective map $B \to A_{\overline{\bL}(\lambda)}(e_{\mu})$. If $\bar{a}_s^{(k)}$ denotes the image of $a_s^{(k)}$ in $A_{\overline{\bL}(\lambda)}(e_{\mu})$ then we have $\bar{a}_s^{(k)} = \bar{a}_t^{(k)}$ for any $1 \leqslant s,t \leqslant r_k$. From this we see that the subgroup $A = \langle a_1^{(k)} \mid k \in \mathcal{I}\rangle \leqslant B$ maps isomorphically onto $A_{\overline{\bL}(\lambda)}(e_{\mu})$ by \cref{eq:Ce-iso,eq:Ce-iso-types}.

\begin{prop}\label{prop:admiss-split-symp}
Assume $\mathcal{B}$ is an alternating bilinear form so that $\bG = \Sp(V) = \SL(V\mid\mathcal{B})$ is a symplectic group. Any unipotent class $\mathcal{O}_{\mu} \in \mathcal{U}(\bG)/\bG$, with $\mu \in \mathcal{P}_{-1}(N)$ is $F$-stable. Moreover, there exists en element $u \in \mathcal{O}_{\mu}^F$, a finite $2$-group $A \leqslant C_{\bG}(u)^F$, and a cocharacter $\lambda \in \mathcal{D}_u(\bG)^F$, such that $(A,\lambda)$ is an admissible covering of $A_{\bG}(u)$.
\end{prop}

\begin{proof}
By \cref{rem:F-fixed-rep} the nilpotent element $e_{\mu}$ is $F$-fixed hence so is $u = \phi_{\spr}^{-1}(e_{\mu}) \in \mathcal{O}_{\mu}$. In particular $\mathcal{O}_{\mu}$ is $F$-stable. We take $A = \langle a_1^{(k)} \mid k \in \mathcal{I} \rangle \leqslant C_{\bL(\lambda)}(e_{\mu})^F = C_{\bL(\lambda)}(u)^F$, and $\lambda \in \mathcal{D}_u(\bG)^F$ to be as in the beginning of the section. Obviously $(A,\lambda)$ satisfies \cref{en:K0}, \cref{en:K1}, and \cref{en:K3}. We need only to check that \cref{en:K2} holds.

With this in mind let $\bS_i \leqslant \GL(V(i))$ be the diagonal maximal torus defined with respect to the basis $\mathcal{V}_i$. As $k \in \mathcal{I}$ is even we have $a_1^{(k)}|_{V(0)} = \Id_{V(0)}$ because any basis element $v_{s,j}^{(m)} \in \mathcal{V}_0$ has $m$ odd. It is therefore clear that the isomorphism in \cref{eq:L-iso} maps $\bA$ into the torus $\bS_1\times\bS_2\times\cdots$. That \cref{en:K2} holds follows from \cref{rem:in-max-torus}.
\end{proof}

If $\bG$ is a special orthogonal group then it is false, in general, that for a unipotent element $u \in \mathcal{U}(\bG)^F$ one can find an admissible covering for $A_{\bG}(u)$. The issue is showing that \cref{en:K2} holds, see \cref{lem:con-comp-cent}. As an example, consider the partition $\mu = (5,3,1)$. However, what we show now is that for special unipotent classes one can find an admissible covering for Lusztig's canonical quotient, which will be sufficient for our needs.

\begin{prop}\label{prop:admissible-splitting-ortho}
Assume $\mathcal{B}$ is a symmetric bilinear form so that $\bG = \SO(V) = \SL(V\mid\mathcal{B})$ is a special orthogonal group. Let $\mathcal{O}_{\mu} \in \mathcal{U}(\bG)/\bG$ be an $F$-stable \emph{special} unipotent class, with $\mu \in \mathcal{P}_1(N)$. Then there exists an element $u \in \mathcal{O}_{\mu}^F$, a finite $2$-group $A \leqslant C_{\bG}(u)^F$, and a cocharacter $\lambda \in \mathcal{D}_u(\bG)^F$ such that the pair $(A,\lambda)$ is an admissible covering of Lusztig's canonical quotient $\bar{A}$ of $A_{\bG}(u)$.
\end{prop}

\begin{proof}
If $\mu$ is a very even partition, i.e., $\mathcal{I} = \emptyset$, then the statement is trivial because $A_{\bG}(u)$ is trivial. Hence we will assume that $\mu$ is not very even so $\mathcal{I} \neq \emptyset$ and we may drop the $\pm$ superscripts from $\mathcal{O}_{\mu}^{\pm}$.

As above, $u = \phi_{\spr}^{-1}(e_{\mu} )\in \mathcal{O}_{\mu}^F$ is $F$-fixed. We take $\lambda \in \mathcal{D}_u(\bG)^F$ to be as in \cref{sec:admissible-split-classical}. Following \cite{So01} we write $\mathcal{I} = \mathcal{I}_{\mathrm{odd}} \sqcup \mathcal{I}_{\mathrm{ev}}$ where $\mathcal{I}_{\mathrm{odd}} = \{k \in \mathcal{I} \mid r_k(\mu) \equiv 1\pmod{2}\}$ and similarly $\mathcal{I}_{\mathrm{ev}} = \{k \in \mathcal{I} \mid r_k(\mu) \equiv 0\pmod{2}\}$. As in the proof of \cite[Thm.~6]{So01} we write $\mathcal{I}_{\mathrm{odd}} = \{j_1,\dots,j_s\}$ with $j_s \geqslant \cdots \geqslant j_1$ and for each $1 \leqslant m < s$ we let $\mathcal{I}^{(m)} = \{k \in \mathcal{I} \mid j_{m+1} > k > j_m\}$. Note that, by definition we have $\mathcal{I}^{(m)} \subseteq \mathcal{I}_{\mathrm{ev}}$.

It is shown in the proof of \cite[Thm.~6]{So01} that for any odd $1 \leqslant m < t$ and $i \in \mathcal{I}^{(m)}$ we have $\bar{a}_1^{(i)}\bar{a}_1^{(j_m)}$ and $\bar{a}_1^{(j_{m+1})}\bar{a}_1^{(j_m)}$ are in the kernel of the natural map $A_{\bG}(u) \to \bar{A}_{\mathscr{F}}$. Now let us write $\mathcal{I} = \{k_1,\dots,k_t\}$ with $k_t > \cdots > k_1$. For $1 \leqslant n < t$ we set
\begin{equation*}
c_n = \begin{cases}
a_1^{(k_{n+1})}a_1^{(k_n)} &\text{if }k_n \in \mathcal{I}_{\mathrm{odd}}\\
a_1^{(k_{n+1})}a_2^{(k_n)} &\text{if }k_n \in \mathcal{I}_{\mathrm{ev}},
\end{cases}
\end{equation*}
where we note that if $k_n \in \mathcal{I}_{\mathrm{ev}}$ then we necessarily have $r_{k_n}(\mu) > 1$. Let $\widetilde{A} = \langle c_n \mid 1\leqslant n < t \rangle \leqslant C_{\bG}(u)^F$. This is an elementary abelian $2$-group that maps isomorphically onto $A_{\bG}(u)$. Assume $1 \leqslant m < n < t$ then we have $V_{-1}(c_m) \cap V_{-1}(c_n) = \{0\}$ unless $n = m+1$ and $k_n \in \mathcal{I}_{\mathrm{odd}}$.

We claim that for any $1 < n < t$ with $k_n \in \mathcal{I}_{\mathrm{odd}}$ we have either the image of $c_n$ or $c_{n-1}$ in $A_{\bG}(u)$ is in the kernel of the natural map $A_{\bG}(u) \to \bar{A}_{\mathscr{F}}$. Write $k_n = j_m$ with $1 \leqslant m \leqslant t$. Either $k_{n+1} \in \mathcal{I}^{(m)}$ or $k_{n+1} = j_{m+1}$ so if $m$ is odd we have $\bar{a}_1^{(k_{n+1})}\bar{a}_1^{(k_n)}$ is in the kernel. Similarly, either $k_{n-1} \in \mathcal{I}^{(m-1)}$ or $k_{n-1} = j_{m-1}$. If $m$ is even then $\bar{a}_1^{(k_n)}\bar{a}_1^{(k_{n-1})}$ is in the kernel because if $k_{n-1} \in \mathcal{I}^{(m-1)}$ we have
\begin{equation*}
\bar{a}_1^{(k_n)}\bar{a}_1^{(k_{n-1})} = (\bar{a}_1^{(j_m)}\bar{a}_1^{(j_{m-1})})(\bar{a}_1^{(k_{n-1})}\bar{a}_1^{(j_{m-1})}).
\end{equation*}

Now for each $1 < n < t$ with $k_n \in \mathcal{I}_{\mathrm{odd}}$ we remove one of the generators $c_{n-1}$ or $c_n$ from $\widetilde{A}$ which is contained in the kernel of the natural map $A_{\bG}(u) \to \bar{A}_{\mathscr{F}}$. Let $A \leqslant \widetilde{A}$ be generated by what is left then we have a subgroup whose generators, and therefore their restrictions to $V(0)$, satisfy the assumptions of \cref{lem:con-comp-cent}. Consequently there exists a torus $\bS_0 \leqslant \SL(V(0) \mid \mathcal{B})$ containing $a_{|V(0)}$ for every $a \in A$. We may now conclude as in \cref{prop:admiss-split-symp} that \cref{en:K2} holds. Finally, as in the symplectic case it is obvious that $(A,\lambda)$ satisfies \cref{en:K0}, \cref{en:K1}, and \cref{en:K3}. 
\end{proof}

\section{Admissible Coverings for Exceptional Groups}\label{ssec:exc}
\begin{figure}[t]
\definecolor{rnode}{gray}{0.85}
\begin{center}
\begin{tabular}{rl>{\hspace{1cm}}rl}
$\tilde{\B}_n$ & 
\begin{tikzpicture}[baseline,scale=0.75]
\begin{scope}[xshift={-3.3cm}]
\coordinate (0) at (140:1.1);
\coordinate (1) at (220:1.1);
\end{scope}
\coordinate (2) at (-3.3,0);
\coordinate (3) at (-2.2,0);
\coordinate (4) at (-1.1,0);
\coordinate (5) at (0,0);
\coordinate (6) at (1.1,0);

\draw (2) -- (3);
\draw[dashed] (3) -- (4);
\draw (4) -- (5);
\draw (0) -- (2) -- (1);

\draw (0,0.075) -- (1.1,0.075);
\draw (0,-0.075) -- (1.1,-0.075);

\draw (0.75,0) -- (0.35,0.25);
\draw (0.75,0) -- (0.35,-0.25);

\foreach \a in {1,...,6} {
\filldraw (\a) circle (0.12cm);
}

\filldraw[fill=white] (0) circle (0.12cm);

\node[left=0.15cm] at (0) {\tiny $0$};
\node[left=0.15cm] at (1) {\tiny $1$};
\node[below=0.15cm] at (2) {\tiny $2$};
\node[below=0.15cm] at (5) {\tiny $n{-}1$};
\node[below=0.15cm] at (6) {\tiny $n$};
\end{tikzpicture}
&
$\tilde{\C}_n$ & 
\begin{tikzpicture}[baseline,scale=0.75]
\coordinate (0) at (-4.4,0);
\coordinate (1) at (-3.3,0);
\coordinate (2) at (-2.2,0);
\coordinate (3) at (-1.1,0);
\coordinate (4) at (0,0);
\coordinate (5) at (1.1,0);

\draw (-4.4,0.075) -- (-3.3,0.075);
\draw (-4.4,-0.075) -- (-3.3,-0.075);
\draw (1) -- (2);
\draw (3) -- (4);
\draw (0,0.075) -- (1.1,0.075);
\draw (0,-0.075) -- (1.1,-0.075);

\draw (-3.65,0) -- (-4.05,0.25);
\draw (-3.65,0) -- (-4.05,-0.25);

\draw (0.35,0) -- (0.75,0.25);
\draw (0.35,0) -- (0.75,-0.25);

\draw[dashed] (2) -- (3);

\foreach \a in {1,...,5} {
\filldraw (\a) circle (0.12cm);
}

\filldraw[fill=white] (0) circle (0.12cm);

\node[below=0.15cm] at (0) {\tiny $0$};
\node[below=0.15cm] at (1) {\tiny $1$};
\node[below=0.15cm] at (2) {\tiny $2$};
\node[below=0.15cm] at (4) {\tiny $n{-}1$};
\node[below=0.15cm] at (5) {\tiny $n$};

\end{tikzpicture}\\

$\tilde{\D}_n$ & 
\begin{tikzpicture}[baseline,scale=0.75]
\begin{scope}[xshift={-3.3cm}]
\coordinate (0) at (140:1.1);
\coordinate (1) at (220:1.1);
\end{scope}
\coordinate (2) at (-3.3,0);
\coordinate (3) at (-2.2,0);
\coordinate (4) at (-1.1,0);
\coordinate (5) at (0,0);
\coordinate (6) at (40:1.1);
\coordinate (7) at (320:1.1);

\draw (2) -- (3);
\draw[dashed] (3) -- (4);
\draw (4) -- (5);
\draw (7) -- (5) -- (6);
\draw (0) -- (2) -- (1);

\foreach \a in {1,...,7} {
\filldraw (\a) circle (0.12cm);
}

\filldraw[fill=white] (0) circle (0.12cm);

\node[left=0.15cm] at (0) {\tiny $0$};
\node[left=0.15cm] at (1) {\tiny $1$};
\node[below=0.15cm] at (2) {\tiny $2$};
\node[below=0.15cm] at (5) {\tiny $n{-}2$};
\node[right=0.15cm] at (6) {\tiny $n{-}1$};
\node[right=0.15cm] at (7) {\tiny $n$};

\end{tikzpicture}
&
$\tilde{\F}_4$
&
\begin{tikzpicture}[baseline,scale=0.75]
\node[above=0.15cm] at (0,0) {\tiny $0$};
\node[above=0.15cm] at (1.1,0) {\tiny $1$};
\node[above=0.15cm] at (2.2,0) {\tiny $2$};
\node[above=0.15cm] at (3.3,0) {\tiny $3$};
\node[above=0.15cm] at (4.4,0) {\tiny $4$};

\draw (0,0) -- (1.1,0) -- (2.2,0);
\draw (3.3,0) -- (4.4,0);
\draw (2.2,0.075) -- (3.3,0.075);
\draw (2.2,-0.075) -- (3.3,-0.075);

\draw (2.95,0) -- (2.55,0.25);
\draw (2.95,0) -- (2.55,-0.25);

\filldraw[fill=white] (0,0) circle (0.12cm);
\filldraw[fill=red] (1.1,0) circle (0.12cm);
\filldraw (2.2,0) circle (0.12cm);
\filldraw (3.3,0) circle (0.12cm);
\filldraw[fill=red] (4.4,0) circle (0.12cm);
\end{tikzpicture}\\[0.8cm]

$\tilde{\E}_7$ & 
\begin{tikzpicture}[baseline,scale=0.75]
\draw (0,0.55) -- (1.1,0.55) -- (2.2,0.55) -- (3.3,0.55) -- (4.4,0.55) -- (5.5,0.55) -- (6.6,0.55);
\draw (3.3,0.5) -- (3.3,-0.5);

\filldraw[fill=white] (0,0.55) circle (0.12cm);
\filldraw[fill=red] (1.1,0.55) circle (0.12cm);
\filldraw (2.2,0.55) circle (0.12cm);
\filldraw (3.3,0.55) circle (0.12cm);
\filldraw (4.4,0.55) circle (0.12cm);
\filldraw (5.5,0.55) circle (0.12cm);
\filldraw (6.6,0.55) circle (0.12cm);
\filldraw[fill=red] (3.3,-0.55) circle (0.12cm);

\node[above=0.15cm] at (0,0.55) {\tiny $0$};
\node[above=0.15cm] at (1.1,0.55) {\tiny $1$};
\node[above=0.15cm] at (2.2,0.55) {\tiny $3$};
\node[above=0.15cm] at (3.3,0.55) {\tiny $4$};
\node[above=0.15cm] at (4.4,0.55) {\tiny $5$};
\node[above=0.15cm] at (5.5,0.55) {\tiny $6$};
\node[above=0.15cm] at (6.6,0.55) {\tiny $7$};
\node[right=0.15cm] at (3.3,-0.55) {\tiny $2$};
\end{tikzpicture}
&
$\tilde{\G}_2$
&
\begin{tikzpicture}[baseline,scale=0.75]
\node[above=0.15cm] at (0,0) {\tiny $1$};
\node[above=0.15cm] at (1.1,0) {\tiny $2$};
\node[above=0.15cm] at (2.2,0) {\tiny $0$};

\draw (0,0) -- (1.1,0) -- (2.2,0);
\draw (0,0.075) -- (1.1,0.075);
\draw (0,-0.075) -- (1.1,-0.075);

\draw (0.35,0) -- (0.75,0.25);
\draw (0.35,0) -- (0.75,-0.25);

\filldraw (0,0) circle (0.12cm);
\filldraw[fill=red] (1.1,0) circle (0.12cm); 
\filldraw[fill=white] (2.2,0) circle (0.12cm); 
\end{tikzpicture}\\[0.8cm]

$\tilde{\E}_8$ & 
\begin{tikzpicture}[baseline,scale=0.75]
\draw (0,0.55) -- (1.1,0.55) -- (2.2,0.55) -- (3.3,0.55) -- (4.4,0.55) -- (5.5,0.55) -- (6.6,0.55) -- (7.7,0.55);
\draw (2.2,0.5) -- (2.2,-0.5);

\filldraw[fill=red] (0,0.55) circle (0.12cm);
\filldraw (1.1,0.55) circle (0.12cm);
\filldraw (2.2,0.55) circle (0.12cm);
\filldraw (3.3,0.55) circle (0.12cm);
\filldraw (4.4,0.55) circle (0.12cm);
\filldraw (5.5,0.55) circle (0.12cm);
\filldraw[fill=red] (6.6,0.55) circle (0.12cm);
\filldraw[fill=white] (7.7,0.55) circle (0.12cm);
\filldraw (2.2,-0.55) circle (0.12cm);

\node[above=0.15cm] at (0,0.55) {\tiny $1$};
\node[above=0.15cm] at (1.1,0.55) {\tiny $3$};
\node[above=0.15cm] at (2.2,0.55) {\tiny $4$};
\node[above=0.15cm] at (3.3,0.55) {\tiny $5$};
\node[above=0.15cm] at (4.4,0.55) {\tiny $6$};
\node[above=0.15cm] at (5.5,0.55) {\tiny $7$};
\node[above=0.15cm] at (6.6,0.55) {\tiny $8$};
\node[above=0.15cm] at (7.7,0.55) {\tiny $0$};
\node[right=0.15cm] at (2.2,-0.55) {\tiny $2$};
\end{tikzpicture}
&
$\tilde{\E}_6$
&
\begin{tikzpicture}[baseline,scale=0.75]
\draw (0,0.55) -- (1.1,0.55) -- (2.2,0.55) -- (3.3,0.55) -- (4.4,0.55);
\draw (2.2,0.5) -- (2.2,-0.5) -- (2.2,-1.65);

\filldraw (0,0.55) circle (0.12cm);
\filldraw (1.1,0.55) circle (0.12cm);
\filldraw (2.2,0.55) circle (0.12cm);
\filldraw (3.3,0.55) circle (0.12cm);
\filldraw (4.4,0.55) circle (0.12cm);
\filldraw[fill=red] (2.2,-0.55) circle (0.12cm);
\filldraw[fill=white] (2.2,-1.65) circle (0.12cm);

\node[above=0.15cm] at (0,0.55) {\tiny $1$};
\node[above=0.15cm] at (1.1,0.55) {\tiny $3$};
\node[above=0.15cm] at (2.2,0.55) {\tiny $4$};
\node[above=0.15cm] at (3.3,0.55) {\tiny $5$};
\node[above=0.15cm] at (4.4,0.55) {\tiny $6$};
\node[right=0.15cm] at (2.2,-0.55) {\tiny $2$};
\node[right=0.15cm] at (2.2,-1.65) {\tiny $0$};
\end{tikzpicture}
\end{tabular}
\end{center}
\caption{Labelled affine Dynkin diagrams.}
\label{fig:affine-dynkin}
\end{figure}
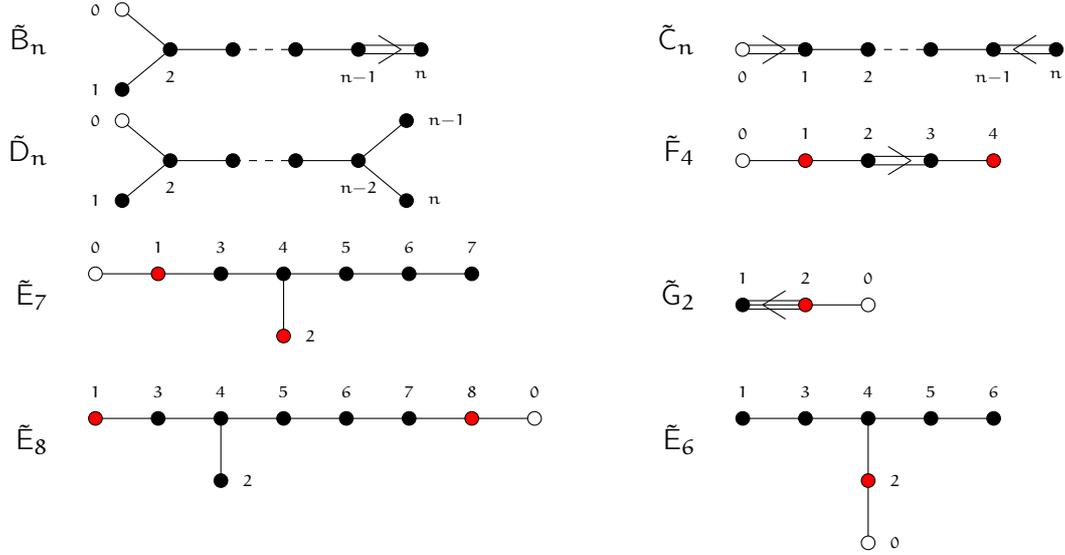

\begin{thm}\label{thm:exceptional-grps}
Assume $\bG$ is an adjoint simple group and $F : \bG \to \bG$ is a Frobenius endomorphism so that either $\bG$ is of exceptional type or $(\bG,F)$ has type ${}^3\D_4$. For each $F$-stable family $\mathscr{F} \in \Fam(W)^F$ there exists a unipotent element $u \in \mathcal{O}_{\mathscr{F}}^F$ and an admissible covering $(A,\lambda)$ of Lusztig's canonical quotient $\bar{A}_{\mathscr{F}}$ of $A_{\bG}(u)$ such that $|A|$ is a product of bad primes for $\bG$.
\end{thm}

The goal of this section is to prove \cref{thm:exceptional-grps}. If $\bar{A}_{\mathscr{F}}$ is the canonical quotient attached to the family $\mathscr{F} \in \Fam(W)$ then it is well known that, when $\bG$ is of exceptional type, we have $\bar{A}_{\mathscr{F}} \cong \mathfrak{S}_m$ is a symmetric group with $m \in \{1,2,3,4,5\}$. When $G$ is Steinberg's triality group there is only one non-trivial $\bar{A}_{\mathscr{F}}$ and it is isomorphic to $\mathfrak{S}_2$. Of course, if $\bar{A}_{\mathscr{F}}$ is trivial then we may simply take $A$ to be trivial. Hence, we will only treat the cases where $\bar{A}_{\mathscr{F}}$ is non-trivial. We note here that, as a consequence of our proof, we establish the following which may be of independent interest. Over $\mathbb{C}$ this was shown by Fu--Juteau--Levy--Sommers \cite[Prop.~6.1]{FJLS17}.

\begin{lem}
Recall that $p$ is assumed to be good for $\bG$. Let $\bG$ be an adjoint simple group, $u \in \bG$ a unipotent element, and $\lambda \in \mathcal{D}_u(\bG)$ a Dynkin cocharacter, then there exists a subgroup $A \leqslant C_{\bL(\lambda)}(u)$ such that $C_{\bG}(u) = C_{\bG}^{\circ}(u) \rtimes A$ unless $\bG$ is of type $\E_7$ or $\E_8$ and $u$ is contained in one of the four classes $\mathcal{O}_{\mathscr{F}}$ occurring in \cref{tab:S_2-no-lift-to-inv}. Moreover, when such a complement exists we may assume it is $F$-stable if $u \in \bG^F$ and $\lambda \in \mathcal{D}_u(\bG)^F$.
\end{lem}

\begin{proof}
If $\bG$ is exceptional and $u$ is distinguished then this is clear as $C_{\bG}^{\circ}(u)$ is unipotent and $A_{\bG}(u)$ is a $p'$-group. If $u$ is contained in the class $\D_4(a_1)$ or $\D_4(a_1){+}\A_1$ then this follows from the calculations in \cref{subsec:D4(a1),subsec:D4(a1)A1}. Finally the cases where $A_{\bG}(u) \cong \mathfrak{S}_2$ are covered by the arguments in \cref{sec:inv-lift,sec:inv-not-lift}. Whilst not all classes with $A_{\bG}(u) \cong \mathfrak{S}_2$ are listed in \cref{tab:S_2-lift-to-inv,tab:S_2-no-lift-to-inv} the remaining cases are treated with identical calculations performed in \Chevie, which we omit.

Now consider the case where $\bG$ is classical. Let $\pi : \tilde{\bG} \to \bG$ be an isogeny from a classical group $\Sp(V)$ or $\SO(V)$. This defines a bijection $\mathcal{U}(\tilde{\bG}) \to \mathcal{U}(\bG)$ and for any $\tilde{u} \in \mathcal{U}(\tilde{\bG})$ we have $C_{\tilde{\bG}}(\tilde{u}) = \pi^{-1}(C_{\bG}(u))$, where $u = \pi(\tilde{u})$. We now appeal to the notation of \cref{ssec:symplectic}. Writing $\mathcal{I} = \{k_1,\dots,k_m\}$ then we let $\tilde{A} = \langle a_1^{(k_1)},\dots, a_1^{(k_m)}\rangle$ if $\tilde{\bG} = \Sp(V)$ and $\tilde{A} = \langle a_1^{(k_1)}a_1^{(k_2)},\dots, a_1^{(k_{m-1})}a_1^{(k_m)}\rangle$ if $\tilde{\bG} = \SO(V)$. It is clear that $\tilde{A}$ is an elementary abelian $2$-group and $C_{\tilde{\bG}}(\tilde{u}) = C_{\tilde{\bG}}^{\circ}(\tilde{u}) \rtimes \tilde{A}$.

Let $B = \pi(\tilde{A})$ then $C_{\bG}(u) = C_{\bG}^{\circ}(u)B$ but it may happen that $C_{\bG}^{\circ}(u) \cap B \neq \{1\}$. However, as $B$ is elementary abelian the subgroup $C_{\bG}^{\circ}(u) \cap B \leqslant B$ has a complement, say $A$, in $B$. The group $A$ then gives a complement.
\end{proof}

In \cref{tab:S_2-lift-to-inv,tab:S_2-no-lift-to-inv,tab:families-with-A-nonabelian} below we gather various information regarding Lusztig's canonical quotient in adjoint exceptional groups. This information can be extracted from the data gathered in \cite[Chp.~13]{Car}, or alternatively using \Chevie{}. Indeed, in \cite{Car} one finds all the sets of families $\Fam(W)$ with the corresponding special character in each family. From the explicit description of the Springer correspondence one can then compute the canonical quotient directly from the definition given in \cite[\S13.1]{Lu84}, see also the tables in \cite{So01}.

Proving \cref{thm:exceptional-grps} will involve some detailed calculations in exceptional groups. For this we introduce some notation relating to the Steinberg presentation of $\bG$. Let $\mathcal{R}(\bG,\bT_0,\bB_0) = (X,\Phi,\Delta,\widecheck{X},\widecheck{\Phi},\widecheck{\Delta})$ be the based root datum of $\bG$ defined as in \cref{subsec:reductive-groups}. As in \cite[Prop.~8.1.1]{Spr09} we choose for each root $\alpha \in \Phi$ a closed embedding $x_{\alpha} : K^+ \to \bG$, of the additive group of the field, such that ${}^tx_{\alpha}(c) = x_{\alpha}(\alpha(t)c)$ for all $t \in \bT_0$ and $c\in K$. The image of the differential $d_1x_{\alpha} : \Lie(K^+) \to \Lie(\bG)$ is the corresponding $1$-dimensional root space. We set $e_{\alpha} := d_1x_{\alpha}(1)$.

We also have an element $n_{\alpha} = x_{\alpha}(1)x_{-\alpha}(-1)x_{\alpha}(1) \in N_{\bG}(\bT_0)$ whose image in the Weyl group $W = N_{\bG}(\bT_0)/\bT_0$, denoted by $s_{\alpha}$, is the reflection of $\alpha$ in the natural action of $W$ on $\mathbb{Q} \otimes_{\mathbb{Z}} X$. If $w \in W$ is any element of the Weyl group then we set $n_w := n_{\alpha_1}\cdots n_{\alpha_r}$ where $w = s_{\alpha_1}\cdots s_{\alpha_r}$ is a reduced expression for the element. The element $n_w$ does not depend upon the choice of reduced expression used to define it (see \cite[Prop.~4.3]{Tit66} or \cite[8.3.3, 9.3.2]{Spr09}).

We set $\widecheck{V} = \mathbb{Q} \otimes \widecheck{X}$. For any simple root $\alpha \in \Delta$ we denote by $\widecheck{\omega}_{\alpha} \in \widecheck{V}$ the fundamental dominant coweight corresponding to $\alpha$ so that we have $\langle \alpha,\widecheck{\omega}_{\beta}\rangle = \delta_{\alpha,\beta}$. In what follows we will assume, unless otherwise specified, that $\bG$ is adjoint simple. In that case $\widecheck{X}$ has a basis, as a $\mathbb{Z}$-module, given by $\{\widecheck{\omega}_{\alpha} \mid \alpha \in \Delta\}$. We fix a homomorphism $\iota : \mathbb{Q}^+  \to K^{\times}$ of abelian groups such that $\mathbb{Z} \subseteq \Ker(\iota)$ and $\Ker(\iota)/\mathbb{Z} \leqslant \mathbb{Q}/\mathbb{Z}$ is the $p$-torsion subgroup of $\mathbb{Q/Z}$. If $\gamma \in \widecheck{X}(\bT_0)$ then we consider this as a homomorphism $\mathbb{Q}^+ \to \bT_0$ by composing with $\iota$. For example, we write $\widecheck{\omega}_{\alpha}(\frac{1}{2})$ instead of $\widecheck{\omega}_{\alpha}(\iota(\frac{1}{2}))$.

As $\bG$ is simple we have $\Phi$ is irreducible and, unless specified otherwise, we denote by $\alpha_0 \in \Phi$ the negative of the highest root. We take the simple roots of $\bG$ to be $\Delta = \{\alpha_1,\dots,\alpha_n\}$ with the ordering of the simple roots taken to match the labelling of the Dynkin diagrams in \cref{fig:affine-dynkin}. We let $\tilde{\Delta} = \Delta\cup\{\alpha_0\}$ be the roots of the extended Dynkin diagram.

The elements $e_{\alpha} \in \Lie(\bG)$ for $\alpha \in \Phi$ defined above may, and will, be chosen to form  part of a Chevalley basis for the Lie algebra. The structure constants for the Lie algebra depend upon a series of signs chosen on so-called extraspecial pairs. We follow a standard algorithm to produce the extraspecial pairs and on each extraspecial pair we choose the sign to be $1$. To compute the relevant structure constants, and related information, we use the algorithms defined in \cite[\S3]{CoMuTa04}. We defer to \cite{CoMuTa04} for the details.

Given two roots $\alpha,\beta \in \Phi$ there exists a sign $\eta_{\alpha,\beta} \in \{\pm1\}$ such that ${}^{n_{\alpha}}x_{\beta}(c) = x_{s_{\alpha}\beta}(\eta_{\alpha,\beta}c)$ for all $c \in K$. There is a standard algorithm for calculating these signs as described in \cite[\S3]{CoMuTa04}. If $w \in W$ then there also exists a sign $\eta_{w,\alpha}$ such that ${}^{n_w}x_{\alpha}(c) = x_{w\alpha}(\eta_{w,\alpha}c)$ for all $c \in K$. If $w = s_{\alpha_r} \cdots s_{\alpha_1}$ is a fixed reduced expression for $w$ then we have
\begin{equation*}
\eta_{w,\alpha} = \eta_{\alpha_1,\alpha}\eta_{\alpha_2,s_{\alpha_1}\alpha} \cdots \eta_{\alpha_r,s_{\alpha_{r-1}}\cdots s_{\alpha_1}\alpha}.
\end{equation*}
Hence $\eta_{w,\alpha}$ can be computed from the signs $\eta_{\alpha,\beta}$ defined above.

\subsection{Quasi-semisimple elements}
In this section we recall some recent results of \cite{DiMi18} concerning quasi-semisimple elements. Let $\Sigma \subseteq \Phi$ be a closed subsystem with positive system $\Sigma^+ \subseteq \Sigma$ and $\bT_0 \leqslant \bM \leqslant \bG$ the corresponding subsystem subgroup. We will assume that $\sigma \in N_{\bG}(\bT_0)$ is such that ${}^{\sigma}\Sigma^+ = \Sigma^+$ so certainly ${}^{\sigma}\Sigma = \Sigma$ and ${}^{\sigma}\bM = \bM$. Moreover, if $\Pi \subseteq \Sigma^+$ is the unique simple system contained in $\Sigma^+$ then ${}^{\sigma}\Pi = \Pi$. Moreover, the element $\sigma$ is a \emph{quasi-semisimple} element of the group generated by the coset $\bM\sigma$.

Let $\Sigma/\langle \sigma\rangle$ be the set of orbits of $\langle\sigma\rangle$ acting on $\Sigma$. To each orbit $\mathcal{O} \in \Sigma/\langle \sigma\rangle$ we associate vectors $\beta_{\mathcal{O}}^* = \sum_{\alpha \in \mathcal{O}} \alpha \in V$ and $\beta_{\mathcal{O}} = |\mathcal{O}|^{-1}\beta_{\mathcal{O}}^*$. Similarly we associate vectors $\widecheck{\beta}_{\mathcal{O}}^* = \sum_{\alpha \in \mathcal{O}} \widecheck{\alpha} \in \widecheck{V}$ and $\widecheck{\beta}_{\mathcal{O}} = |\mathcal{O}|^{-1}\widecheck{\beta}_{\mathcal{O}}^*$. Furthermore, we denote by $C_{\sigma,\mathcal{O}} \in K^{\times}$ the unique element such that ${}^{\sigma^{|\mathcal{O}|}}x_{\alpha}(k) = x_{\alpha}(C_{\sigma,\mathcal{O}}k)$ for any $\alpha \in \mathcal{O}$ and $k \in K$.

Following \cite{DiMi18} we say the orbit $\mathcal{O}$ is \emph{special} if there exist two roots $\alpha,\beta \in \mathcal{O}$ such that $\alpha+\beta \in \Sigma$ is a root. The orbit of $\alpha+\beta$ is then said to be \emph{cospecial}. We set $s(\mathcal{O}) = 2$ if the orbit is special and $s(\mathcal{O}) = 1$ otherwise.

\begin{prop}[{}{see, \cite[Prop.~1.11]{DiMi18}}]\label{prop:root-sys-cent-sig}
The group $C_{\bM}^{\circ}(\sigma)$ is a connected reductive algebraic group with maximal torus $C_{\bT_0}^{\circ}(\sigma)$ and Borel subgroup $C_{\bB_0}^{\circ}(\sigma)$. Let $\Sigma_{\sigma}$ be the roots of $C_{\bM}^{\circ}(\sigma)$ with respect to $C_{\bT_0}^{\circ}(\sigma)$ and let $\Pi_{\sigma} \subseteq \Sigma_{\sigma}^+ \subseteq \Sigma_{\sigma}$ be the simple and positive roots determined by $C_{\bB_0}^{\circ}(\sigma)$. If $p \neq 2$ then we have:
\begin{enumerate}
	\item $\Sigma_{\sigma} = \{\beta_{\mathcal{O}} \mid \mathcal{O} \in \Sigma/\langle \sigma\rangle\text{ and }C_{\sigma,\mathcal{O}} = 1\}$,
	\item $\Sigma_{\sigma}^+ = \{\beta_{\mathcal{O}} \mid \mathcal{O} \in \Sigma^+/\langle \sigma\rangle\text{ and }C_{\sigma,\mathcal{O}} = 1\}$,
	\item $\Pi_{\sigma} = \{\beta_{\mathcal{O}} \mid \mathcal{O} \in \Pi/\langle \sigma\rangle\text{ and }C_{\sigma,\mathcal{O}} = 1\} \cup \{2\beta_{\mathcal{O}} \mid \mathcal{O} \in \Pi/\langle \sigma\rangle\text{ is special and }C_{\sigma,\mathcal{O}} = -1\}$.
\end{enumerate}
Moreover, we have $s(\mathcal{O})\widecheck{\beta}_{\mathcal{O}}^* \in \widecheck{X}(C_{\bT_0}^{\circ}(\sigma))$ is the coroot corresponding to $\beta_{\mathcal{O}} \in \Sigma_{\sigma}$. We let $\widecheck{\Pi}_{\sigma} \subseteq \widecheck{\Sigma}_{\sigma}^+ \subseteq \widecheck{\Sigma}_{\sigma}$ be the corresponding sets of coroots.
\end{prop}

\begin{rem}
Assume $\mathcal{O}_1 \in \Sigma/\langle \sigma\rangle$ is a special orbit then $|\mathcal{O}_1| = 2i$ is even and $\alpha+{}^{\sigma^i}\alpha \in \Sigma$ is a root. Moreover, $\{\alpha,{}^{\sigma^i}\alpha\}$ is contained in an irreducible component of type $\A_{2k}$. If $\mathcal{O}_2$ is the cospecial orbit containing $\alpha+{}^{\sigma^i}\alpha$ then $\beta_{\mathcal{O}_2} = 2\beta_{\mathcal{O}_1}$. However, by the calculation in \cite[8.2]{St68} we have $C_{\sigma,\mathcal{O}_1} = -C_{\sigma,\mathcal{O}_2}$ so exactly one of $\beta_{\mathcal{O}_1}$ or $\beta_{\mathcal{O}_2}$ is contained in $\Sigma_{\sigma}$.
\end{rem}

Let $W_{\bM}^{\circ}(\sigma)$ be the Weyl group $N_{C_{\bM}^{\circ}(\sigma)}(C_{\bT_0}^{\circ}(\sigma))/C_{\bT_0}^{\circ}(\sigma)$ of $C_{\bM}^{\circ}(\sigma)$ with respect to $C_{\bT_0}^{\circ}(\sigma)$. We may naturally identify $W_{\bM}^{\circ}(\sigma)$ with a subgroup of the centraliser $W_{\bG}(\sigma) = \{w \in W_{\bG}(\bT_0) \mid {}^{\sigma}w = w\}$. Following \cite{DiMi18} we say $\sigma$ is \emph{quasi-central} if $W_{\bM}^{\circ}(\sigma) = W_{\bG}(\sigma)$. If $\sigma$ is quasi-central then $C_{\sigma,\mathcal{O}} \in \{\pm 1\}$ for all orbits $\mathcal{O} \in \Sigma/\langle \sigma\rangle$ and $C_{\sigma,\mathcal{O}} = 1$ unless the orbit $\mathcal{O}$ is special. We will need the following.

\begin{lem}[{}{see, \cite[(1.24)]{DiMi18}}]
Assume $\sigma$ is quasi-central and $t \in C_{\bT_0}(\sigma)$ then we have $\Sigma_{t\sigma} = \{\beta_{\mathcal{O}} \mid \mathcal{O} \in \Sigma/\langle \sigma\rangle\text{ and }\beta_{\mathcal{O}}^*(t) = C_{\sigma,\mathcal{O}}\}$.
\end{lem}

\begin{rem}
If there are no special orbits then the condition that $\beta_{\mathcal{O}}^*(t) = C_{\sigma,\mathcal{O}}$ is simply the condition that $\beta_{\mathcal{O}}^*(t) = 1$. However, even in this case the root system $\Sigma_{t\sigma}$ is not necessarily a closed subsystem of $\Sigma_{\sigma}$.
\end{rem}

Let us now give an example of these concepts which will be used later to deal with the unipotent class $\D_4(a_1)$ in $\E_6$.

\begin{exmp}\label{exmp:qs-elts-3D4}
Assume $\bG$ is any simple group of type $\E_6$ and let $\bM \leqslant \bG$ be the standard Levi subgroup of type $\D_4$ with basis of simple roots given by $\Pi = \{2,3,4,5\}$. The group $N_W(\Pi) = \{w \in W \mid {}^w\Pi = \Pi\} = \langle s_{\Pi,\alpha_1}, s_{\Pi,\alpha_6}\rangle$ is isomorphic to the symmetric group $\mathfrak{S}_3$, where $s_{\Pi,\alpha} = w_{\Pi\cup\{\alpha\}}w_{\Pi}$ is the product of longest elements. The element $w = s_{\Pi,\alpha_6}s_{\Pi,\alpha_1}$ is of order $3$ and acts on $\Pi$ as the cyclic permutation $(\alpha_2,\alpha_5,\alpha_3)$.

Let $\sigma = n_w$, an element of order $3$, then $\Sigma^+/\langle\sigma\rangle$ contains six orbits
\begin{align*}
\mathcal{O}_1 &= \{1000,0100,0001\} & \mathcal{O}_3 &= \{1010,0110,0011\} & \mathcal{O}_5 &= \{1111\},\\
\mathcal{O}_2 &= \{0010\} & \mathcal{O}_4 &= \{1110,1011,0111\} & \mathcal{O}_6 &= \{1121\}
\end{align*}
where roots are written in Bourbaki convention with respect to $\Pi$. To ease notation we set $\beta_i = \beta_{\mathcal{O}_i}$, $\beta_i^* = \beta_{\mathcal{O}_i}^*$, \dots, etc. The set $\Sigma/\langle \sigma\rangle$ contains no special orbits, hence also no cospecial orbits. As $\sigma$ has odd order it follows that $C_{\sigma,\mathcal{O}} = 1$ for all orbits $\mathcal{O} \in \Sigma/\langle \sigma\rangle$ so $\sigma$ is quasi-central.

The root system $\Sigma_{\sigma}$ is of type $\G_2$ with simple system $\Pi_{\sigma} = \{\beta_1,\beta_2\}$. Here the root $\beta_1$ is short and the root $\beta_2$ is long. Now consider the element $t = \widecheck{\omega}_4(\frac{1}{3}) \in C_{\bT_0}(\sigma)$. Note that $\widecheck{\omega}_4 \in \mathbb{Z}\widecheck{\Phi}$ is contained in the coroot lattice so $t$ is an element of order $3$ and we have
\begin{equation*}
\Sigma_{t\sigma}^+ = \{\beta_{\mathcal{O}} \mid \mathcal{O} \in \Sigma/\langle\sigma\rangle\text{ and } \langle \beta_{\mathcal{O}}^*, \widecheck{\omega}_4\rangle \in 3\mathbb{Z}\} = \{\beta_1,\beta_3,\beta_4\}
\end{equation*}
are the positive roots of a root system of type $\A_2$ with simple system $\{\beta_1,\beta_3\}$. Note that the sum $\beta_1+\beta_4 = 3\beta_1+\beta_2 \in \Sigma_{\sigma}$ is a root not contained in $\Sigma_{t\sigma}$ so this subsystem of $\Sigma_{\sigma}$ is not closed in $\Sigma_{\sigma}$.
\end{exmp}

In the following sections we will need to calculate the fusion map $\mathcal{U}(C_{\bM}^{\circ}(\sigma)) \rightsquigarrow_{\bM} \mathcal{U}(\bM)$, as in \cref{sec:fusion-unip-cls}. Note that $C_{\bM}^{\circ}(\sigma)$ is not a maximal rank subgroup of $\bM$ so the results in \cref{sec:fusion-unip-cls} do not directly apply. However, the fusion map can be calculated with the following result of Fowler--R\"ohrle which is an analogue of \cref{thm:fowler-rohrle} for the group $C_{\bM}^{\circ}(\sigma)$.

\begin{thm}[{}{Fowler--R\"ohrle, \cite[Cor.~3.23]{FowRohr08}}]\label{thm:fowler-rohrle-qss}
Assume $p$ is a good prime for $\bG$ and $\sigma \in \bG$ is a semisimple element. If $u \in \mathcal{U}(C_{\bM}^{\circ}(\sigma))$ is a unipotent element then $\mathcal{D}_u(C_{\bM}^{\circ}(\sigma)) = \mathcal{D}_u(\bM,C_{\bM}^{\circ}(\sigma))$.
\end{thm}

With this result we may now use the same algorithm given in \cref{alg:fusion} to calculate the fusion map $\mathcal{U}(C_{\bM}^{\circ}(\sigma)) \rightsquigarrow_{\bM} \mathcal{U}(\bM)$. We note that, as before, to calculate the fusion map we need only that the inclusion $\mathcal{D}_u(C_{\bM}^{\circ}(\sigma)) \subseteq \mathcal{D}_u(\bG)$ holds. Moreover, as $\bM$ has maximal rank inside $\bG$ another application of \cref{alg:fusion} gives the fusion map $\mathcal{U}(\bM) \rightsquigarrow_{\bG} \mathcal{U}(\bG)$. So we may in fact use the algorithm to calculate directly the map $\mathcal{U}(C_{\bM}^{\circ}(\sigma)) \rightsquigarrow_{\bG} \mathcal{U}(\bG)$. We give the following example to illustrate the point.

\begin{exmp}\label{exmp:fusion-3D4}
We use the notation of \cref{exmp:qs-elts-3D4}. Note that we have
\begin{align*}
\widecheck{\beta}_1^* &= 2(\widecheck{\omega}_2+\widecheck{\omega}_3+\widecheck{\omega}_5) - 3\widecheck{\omega}_4,\\
\widecheck{\beta}_2^* &= -(\widecheck{\omega}_2+\widecheck{\omega}_3+\widecheck{\omega}_5) + 2\widecheck{\omega}_4.
\end{align*}
Let $\widecheck{\pi}_i$ be the fundamental dominant coweight of $\Sigma_{\sigma}$ corresponding to $\beta_i$. We then have $\widecheck{\pi}_1 = 2\widecheck{\beta}_1^*+3\widecheck{\beta}_2^* = \widecheck{\omega}_2+\widecheck{\omega}_3+\widecheck{\omega}_5$ and $\widecheck{\pi}_2 = \widecheck{\beta}_1^*+2\widecheck{\beta}_2^* = \widecheck{\omega}_4$. As one might expect we find that, in terms of weighted Dynkin diagrams, the fusion map $\mathcal{U}(C_{\bM}^{\circ}(\sigma)) \rightsquigarrow_{\bM} \mathcal{U}(\bM)$ is given as follows.
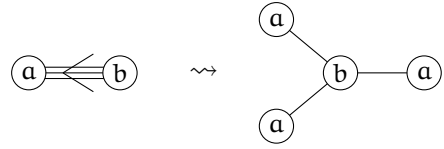
\begin{center}
\begin{tikzpicture}[baseline,
vertex/.style={inner sep=0pt,minimum size=4.5mm,draw,circle,fill=white,align=center}]
\draw (0,0) -- (1.1,0);
\draw (0,0.075) -- (1.1,0.075);
\draw (0,-0.075) -- (1.1,-0.075);

\draw (0.35,0) -- (0.75,0.25);
\draw (0.35,0) -- (0.75,-0.25);

\node[vertex] at (-0.1,0) {\small $a$};
\node[vertex] at (1.1,0) {\small $b$};

\node at (2.2,0) {$\rightsquigarrow$};

\begin{scope}[xshift={4cm}]
\coordinate (0) at (140:1.1);
\coordinate (1) at (220:1.1);
\coordinate (2) at (0:1.1);
\coordinate (3) at (0,0);
\end{scope}

\draw (0) -- (3) -- (1);
\draw (3) -- (2);

\node[vertex] at (0) {\small $a$};
\node[vertex] at (1) {\small $a$};
\node[vertex] at (2) {\small $a$};
\node[vertex] at (3) {\small $b$};
\end{tikzpicture}
\end{center}
Recall that $\beta_1$ is the short root. In particular, there is no non-trivial fusion and $C_{\bM}^{\circ}(\sigma)$ meets every $\sigma$-invariant unipotent class of $\bM$ except the subregular class.

Now consider the group $C_{\bM}^{\circ}(t\sigma)$ of type $\A_2$ with simple system $\{\beta_1,\beta_3\}$. Note that $\widecheck{\beta}_3^* = \widecheck{\beta}_1^*+3\widecheck{\beta}_2^*$. The fundamental dominant coweights of $\Sigma_{t\sigma}$ are
\begin{align*}
\frac{1}{3}(2\widecheck{\beta}_1^*+\widecheck{\beta}_3^*) &= \widecheck{\beta}_1^*+\widecheck{\beta}_2^* = (\widecheck{\omega}_2+\widecheck{\omega}_3+\widecheck{\omega}_5) - \widecheck{\omega}_4,\\
\frac{1}{3}(\widecheck{\beta}_1^*+2\widecheck{\beta}_3^*) &= \widecheck{\beta}_1^*+2\widecheck{\beta}_2^* = \widecheck{\omega}_4.
\end{align*}
Thus the fusion in terms of weighted Dynkin diagrams is given as follows
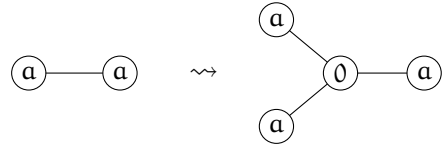
\begin{center}
\begin{tikzpicture}[baseline,
vertex/.style={inner sep=0pt,minimum size=4.5mm,draw,circle,fill=white,align=center}]
\draw (0,0) -- (1.1,0);

\node[vertex] at (-0.1,0) {\small $a$};
\node[vertex] at (1.1,0) {\small $a$};

\node at (2.2,0) {$\rightsquigarrow$};

\begin{scope}[xshift={4cm}]
\coordinate (0) at (140:1.1);
\coordinate (1) at (220:1.1);
\coordinate (2) at (0:1.1);
\coordinate (3) at (0,0);
\end{scope}

\draw (0) -- (3) -- (1);
\draw (3) -- (2);

\node[vertex] at (0) {\small $a$};
\node[vertex] at (1) {\small $a$};
\node[vertex] at (2) {\small $a$};
\node[vertex] at (3) {\small $0$};
\end{tikzpicture}
\end{center}
Hence, we see that $C_{\bM}^{\circ}(t\sigma)$ meets the subregular class of $\bM$.
\end{exmp}

\subsection{Families with $\bar{A}_{\mathscr{F}} \cong \mathfrak{S}_2$ lifting to an involution}\label{sec:inv-lift}
\begin{table}[th]
\centering
\begin{tabular}{*{7}{>{$}c<{$}}}
\toprule
\bG & \chi_{\mathscr{F}} & \mathcal{O}_{\mathscr{F}} & \mathcal{O}_{\mathscr{F}^*} & \Cl_{\bL_J}(u) & J\\
\midrule
\D_4 & \phi_{2,1} & 3^21^2 & 3^21^2 & 4\A_1 & \{0,1,3,4\}\\
\midrule
\F_4
& \phi_{4,13}  & \tilde{\A}_1   & \F_4(a_1) & \A_1{+}\tilde{\A}_1 & \{0,2\}\\
& \phi_{4,1} & \F_4(a_1) & \tilde{\A}_1   & \B_4               & \{0,1,2,3\}\\
\midrule
\E_6
& \phi_{30,15}  & \A_2 & \E_6(a_3)      & 4\A_1       & \{0,1,4,6\}\\
& \phi_{30,3} & \E_6(a_3)      & \A_2  & \A_5{+}\A_1 & \{0,1,3,4,5,6\}\\
\midrule
\E_7
& \phi_{56,30}   & \A_2       & \E_7(a_3)       & 4\A_1        & \{0,3,5,7\}\\
& \phi_{120,25}  & \A_2{+}\A_1       & \E_6(a_1)  & 5\A_1        & \{0,2,3,5,7\}\\
& \phi_{405,15}  & \D_4(a_1){+}\A_1  & \E_6(a_3) & \A_3{+}3\A_1 & \{0,2,3,5,6,7\}\\
& \phi_{420,13} & \A_4       &  \D_5(a_1)      & 2\A_3        & \{0,1,3,5,6,7\}\\
& \phi_{420,10} & \D_5(a_1)            & \A_4  & \D_4{+}2\A_1 & \{0,2,3,4,5,7\}\\
& \phi_{405,8} & \E_6(a_3) & \D_4(a_1){+}\A_1  & \A_5{+}\A_1  & \{0,3,4,5,6,7\}\\
& \phi_{120,4} & \E_6(a_1)      & \A_2{+}\A_1  & \A_7         & \{0,1,3,4,5,6,7\}\\
& \phi_{56,3}  & \E_7(a_3)            & \A_2  & \D_6{+}\A_1  & \{0,2,3,4,5,6,7\}\\
\midrule
\E_8
& \phi_{112,63}   & \A_2   & \E_8(a_3)            & 4\A_1 & \{0,2,5,7\}\\
& \phi_{210,52}   & \A_2{+}\A_1   & \E_8(a_4)      & 5\A_1 & \{0,2,3,5,7\}\\
& \phi_{700,42}   & 2\A_2   & \E_8(a_5)           & \A_2{+}4\A_1 & \{0,1,2,3,5,7\}\\
& \phi_{2268,30} & \A_4   & \E_7(a_3)            & 2\A_3 & \{0,2,4,5,7,8\}\\
& \phi_{2240,28} & \D_4(a_1){+}\A_2   & \E_8(b_6) & \A_3{+}\A_2{+}2\A_1 & \{0,1,2,3,5,6,7\}\\
& \phi_{4200,24} & \A_4{+}2A_1   &  \D_7(a_2)    & \D_4(a_1){+}\A_3 & \{0,2,3,-4,5,7,8\}\\
& \phi_{2800,25} & \D_5(a_1)   & \E_6(a_1)       & \D_4{+}2\A_1 & \{0,2,3,4,5,7\}\\
& \phi_{5600,21} & \E_6(a_3)   & \D_6(a_1)       & \A_5{+}\A_1 & \{0,2,4,5,6,7\}\\
& \phi_{5600,15} & \D_6(a_1)   & \E_6(a_3)       & \D_5{+}2\A_1 & \{0,1,2,3,4,5,7\}\\
& \phi_{2800,13} & \E_6(a_1)   & \D_5(a_1)       & \A_7 & \{0,2,4,5,6,7,8\}\\
& \phi_{2268,10} & \E_7(a_3)        & \A_4       & \D_6{+}\A_1 & \{0,2,3,4,5,6,7\}\\
& \phi_{210,4}  & \E_8(a_4)  & \A_2{+}\A_1       & \D_8 & \{0,2,3,4,5,6,7,8\}\\
& \phi_{112,3}  & \E_8(a_3)        & \A_2       & \E_7{+}\A_1 & \{0,1,2,3,4,5,6,7\}\\
& \phi_{2240,10} & \E_8(b_6) & \D_4(a_1){+}\A_2   & \D_8(a_3) & \{0,2,3,-4,5,-6,7,-8\}\\
& \phi_{700,6}  & \E_8(a_5)      & 2\A_2       & \D_8(a_1) & \{0,2,3,-4,5,6,7,8\}\\
\bottomrule
\end{tabular}
\caption{Families with $\bar{A}_{\mathscr{F}} \cong \mathfrak{S}_2$ lifting to an involution.}
\label{tab:S_2-lift-to-inv}
\end{table}

Each family $\mathscr{F} \in \Fam(W)$ in \cref{tab:S_2-lift-to-inv} satisfies the property that $\bar{A}_{\mathscr{F}} \cong \mathfrak{S}_2$. In the table we list the unique special character $\chi_{\mathscr{F}} \in \mathscr{F}$, the corresponding unipotent class $\mathcal{O}_{\mathscr{F}}$, and the Spaltenstein dual class $\mathcal{O}_{\mathscr{F}^*}$. Moreover, following \cite{So98} we list for each family a set $J$ which encodes the weighted Dynkin diagram for the class $\Cl_{\bL_J}(u)$ of a distinguished unipotent element $u \in \mathcal{U}(\bL_J)$ in a corresponding maximal rank subgroup $\bL_J \leqslant \bG$.

As in \Chevie{} \cite{Chv} $J$ encodes this information as follows. By taking the absolute value of each integer in $J$ we obtain a subset of the extended Dynkin diagram $\tilde{\Delta}$, via the explicit labelling in \cref{fig:affine-dynkin}, and thus a corresponding maximal rank subgroup $\bL_J$. The weights on the weighted Dynkin diagram of the class $\Cl_{\bL_J}(u)$ are determined as follows. If $i \in J$ is positive then the corresponding weight on the diagram is $2$ and if $i$ is negative then the corresponding weight on the diagram is $0$; note that any distinguished class is even. Hence, in this case, $\Cl_{\bL_J}(u)$ is always the regular class in $\bL_J$ except when $\bG$ is of type $\E_8$ and $\mathcal{O}_{\mathscr{F}}$ is either $\A_4{+}2\A_1$, $\E_8(b_6)$, or $\E_8(a_5)$.

With the subset $J$ in hand it is easy to verify in \Chevie{} that $\Cl_{\bG}(u) = \mathcal{O}_{\mathscr{F}}$. In \cref{fig:affine-dynkin} we have listed the extended Dynkin diagrams for exceptional groups. If $s \in \bT_0$ is an isolated involution in $\bG$ then the Dynkin diagram of $C_{\bG}^{\circ}(s)$ is obtained by removing one of the red nodes from the extended diagram. Indeed, up to conjugacy we can, and will, assume that $s = \widecheck{\omega}_{\alpha}(\frac{1}{2})$, with $\alpha \in \Delta$ a red root.

With this we see that each subset $J$ is contained in a subset $K \subseteq \tilde{\Delta}$ such that $C_{\bG}^{\circ}(s) = \bL_K$ where $s \in \bT_0$ is an isolated involution in $\bG$. We note that there could be more than one choice for $K$. Hence we have $\bL_J \leqslant C_{\bG}^{\circ}(s)$ is a Levi subgroup. In particular, this means that $\bL_J = C_{C_{\bG}(s)}^{\circ}(Z^{\circ}(\bL_J)) = C_{\bG}^{\circ}(sZ^{\circ}(\bL_J))$. By \cite[Prop.~15]{McnSo03} there exists an element $t \in sZ^{\circ}(\bL_J)$ such that $\bL_J = C_{\bG}^{\circ}(t)$, so $\bL_J$ is a pseudo-Levi subgroup of $\bG$ in the parlance of \cite{McnSo03}.

As $s$ is isolated in $\bG$ we must have $Z^{\circ}(C_{\bG}^{\circ}(s))$ is trivial; if not then $C_{\bG}^{\circ}(s)$ is contained in the proper Levi subgroup $C_{\bG}(Z^{\circ}(C_{\bG}^{\circ}(s)))$. Hence $s \not\in Z^{\circ}(C_{\bG}^{\circ}(s))$. As $\bL_J$ is a Levi subgroup of $C_{\bG}^{\circ}(s)$ we have the natural map $Z(C_{\bG}^{\circ}(s)) \to Z(\bL_J)/Z^{\circ}(\bL_J)$ is surjective. Using either \cite[Lem.~33]{McnSo03} or a calculation in \Chevie{} we get that $Z(\bL_J)/Z^{\circ}(\bL_J)$ is not trivial in each case. This means $s \not\in Z^{\circ}(\bL_J)$. It thus follows from \cite[Thm.~1]{McnSo03} that $s \not\in C_{\bG}^{\circ}(u)$.

In each case we have the subset $J \subseteq \tilde{\Delta}$ is stable under the permutation $\tilde{\Delta} \to \tilde{\Delta}$ induced by $F$. Hence, we have $\bL_J$ is $F$-stable. Moreover, the class $\Cl_{\bL_J}(u)$ is the unique unipotent class in $\bL_J$ of its dimension so it is also $F$-stable. Therefore we can assume that $u \in \bL_J^F$ is an $F$-fixed element of its conjugacy class. From the description of $s$ above it is clear that we can assume that $s$ is $F$-fixed.

We now take $A = \langle s\rangle$ and we fix a cocharacter $\lambda \in \mathcal{D}_u(\bL_J)^F$. By \cref{thm:fowler-rohrle} we have $\lambda \in \mathcal{D}_u(\bG)^F$ and $s \in \bL(\lambda) = C_{\bG}(\lambda)$ because $\bL_J \leqslant C_{\bG}(s)$. Thus we have $A \leqslant \bL(\lambda)^F$ and the natural maps $A \to A_{\bL(\lambda)}(u) \to A_{\bG}(u)$ are isomorphisms. As $p$ is good for $\bG$ it is necessarily odd so \cref{en:K1} holds and $|A| = 2$ is divisible only by bad primes. We have \cref{en:K2} holds by \cref{rem:cyclic-centraliser} and \cref{en:K3} holds trivially.

\subsection{Families with $\bar{A}_{\mathscr{F}} \cong \mathfrak{S}_2$ not lifting to an involution}\label{sec:inv-not-lift}
\begin{table}[t]
\centering
\begin{tabular}{*{7}{>{$}c<{$}}}
\toprule
\bG & \chi_{\mathscr{F}} & \mathcal{O}_{\mathscr{F}} & \mathcal{O}_{\mathscr{F}^*} & \Cl_{\bL_J}(u) & J\\
\midrule
\E_7
& \phi_{512,11} & \A_4{+}\A_1 & \A_4{+}\A_1 & \A_1{+}2\A_3 & \{2 \mid 0,1,3 \mid 5,6,7\}\\
\midrule
\E_8
& \phi_{4096,11} & \E_6(a_1){+}\A_1 & \A_4{+}\A_1 & \A_7{+}\A_1 & \{2,4,5,6,7,8,0 \mid 1\}\\
& \phi_{4200,12} & \D_7(a_2) & \A_4{+}2A_1 & \D_5{+}\A_3 & \{1,3,4,5,2\mid 7,8,0\}\\
& \phi_{4096,26} & \A_4{+}\A_1 & \E_6(a_1){+}\A_1 & \A_1{+}2\A_3 & \{1 \mid 0,7,8 \mid 2,4,5\}\\
\bottomrule
\end{tabular}
\caption{Families with $\bar{A}_{\mathscr{F}} \cong \mathfrak{S}_2$ not lifting to an involution.}
\label{tab:S_2-no-lift-to-inv}
\end{table}

There are four families where $A_{\bG}(u) = \bar{A}_{\mathscr{F}} \cong \mathfrak{S}_2$ but all involutions in $C_{\bG}(u)$ are contained in $C_{\bG}^{\circ}(u)$. These families are listed in \cref{tab:S_2-no-lift-to-inv}. In each case we either have $\bG$ is of type $\E_7$ or $\E_8$. For each family we list a subset $J \subseteq \tilde{\Delta}$ of the extended Dynkin diagram of $\bG$. We assume $s = \widecheck{\omega}_{\alpha}(\frac{1}{4})$ with $\alpha \in \tilde{\Delta}\setminus J$ then $s$ is an isolated semisimple element of order $4$ with simple system of roots $\tilde{\Delta}\setminus\{\alpha\}$. Hence $J \subseteq \tilde{\Delta}\setminus\{\alpha\}$ so $\bL_J \leqslant C_{\bG}^{\circ}(s)$ is a Levi subgroup of the centraliser. The root $\alpha$ is uniquely determined except in the case where $(\bG,\mathcal{O}_{\mathscr{F}})$ is of type $(\E_8,\A_4{+}\A_1)$ in which case there are two choices. It does not matter which element we choose.

Now we have ${}^nF(s) = s$ for some $n \in N_{\bG}(\bT_0)$ which can be taken to be $1$ if $q \equiv 1 \pmod{4}$ and a representative of the longest element if $q \equiv -1 \pmod{4}$. Assume $g \in \bG$ satisfies $g^{-1}F(g) = n$ then $F({}^gs) = {}^gs$. If $F'(x) = {}^nF(x)$ then conjugation by $g$ intertwines the Frobenius endomorphisms $F$ and $F'$. It suffices to argue with respect to $F'$ and translate under the conjugation action by $g$.

As the longest element is $-1$ we have $\bL_J$ is $F'$-stable and is a Levi subgroup of $C_{\bG}^{\circ}(s)$. If $u \in \bL_J^{F'}$ is a regular unipotent element and $\lambda \in \mathcal{D}_u(\bL_J)^{F'}$ is a corresponding Dynkin cocharacter then $s \in C_{\bL(\lambda)}(u)$. A calculation in \Chevie{} shows that $Z(\bL_J)/Z^{\circ}(\bL_J)$ is of order $4$ in each case so, arguing as in \cref{sec:inv-lift}, we see that $s \not\in C_{\bL(\lambda)}^{\circ}(u)$. It is clear that $A = \langle s\rangle \leqslant \bL(\lambda)^{F'}$ satisfies \cref{en:K0}, \cref{en:K1}, and \cref{en:K3} of \cref{def:admissiblesplitting}, with respect to $F'$, and \cref{en:K2} is satisfied by \cref{rem:cyclic-centraliser}.

\subsection{Families with $\bar{A}_{\mathscr{F}} \cong \mathfrak{S}_3$ and $\mathcal{O}_{\mathscr{F}}$ distinguished}\label{subsec:S_3-distinguished}
\begin{table}[t]
\centering
\begin{tabular}{*{6}{>{$}c<{$}}}
\toprule
\bG & \chi_{\mathscr{F}} & \mathcal{O}_{\mathscr{F}} & \mathcal{O}_{\mathscr{F}^*} & A_{\bG}(u) = \bar{A}_{\mathscr{F}}\\
\midrule
\G_2 & \phi_{2,1} & \G_2(a_1) & \G_2(a_1) & \mathfrak{S}_3\\
\F_4 & \phi_{12,4} & \F_4(a_3) & \F_4(a_3) & \mathfrak{S}_4\\
\E_6 & \phi_{80,7} & \D_4(a_1) & \D_4(a_1) & \mathfrak{S}_3\\
\midrule
\E_7 & \phi_{315,7} & \D_4(a_1) & \E_7(a_5) & \mathfrak{S}_3\\
     & \phi_{315,16} & \E_7(a_5) & \D_4(a_1) & \mathfrak{S}_3\\
\midrule
\E_8 & \phi_{1400,37} & \D_4(a_1) & \E_8(b_5) & \mathfrak{S}_3\\
& \phi_{1400,32} & \D_4(a_1){+}\A_1 & \E_8(a_6) & \mathfrak{S}_3\\
& \phi_{1400,8}  & \E_8(a_6) & \D_4(a_1){+}\A_1 & \mathfrak{S}_3\\
& \phi_{1400,7} & \E_8(b_5) & \D_4(a_1) & \mathfrak{S}_3\\
& \phi_{4480,16} & \E_8(a_7) & \E_8(a_7) & \mathfrak{S}_5\\
\bottomrule
\end{tabular}
\caption{Families with non-abelian $A_{\bG}(u) \cong \bar{A}_{\mathscr{F}}$.}
\label{tab:families-with-A-nonabelian}
\end{table}

Fix an element $u \in \mathcal{O}_{\mathscr{F}}^F$ and a cocharacter $\lambda \in \mathcal{D}_u(\bG)^F$. We have assumed that the class $\mathcal{O}_{\mathscr{F}}$ is distinguished. By definition this means that $C_{\bG}(u)$ contains no non-trivial torus. After (i) of \cref{prop:bumper-para-levi} we have $C_{\bL(\lambda)}(u)$ is reductive so $C_{\bL(\lambda)}^{\circ}(u) = \{1\}$ as its rank must be $0$. Hence, the natural map $C_{\bL(\lambda)}(u) \to A_{\bG}(u)$ is an isomorphism so \cref{en:K3} holds.

 We claim that, after possibly replacing $(u,\lambda)$ by $({}^gu,{}^g\lambda)$ for some $g \in \bG$, we have $(C_{\bL(\lambda)}(u),\lambda)$ is an admissible covering of $A_{\bG}(u) \cong \bar{A}_{\mathscr{F}}$. Firstly, we arrange that \cref{en:K0} holds. Certainly $C_{\bL(\lambda)}(u)$ is $F$-stable but it may not be the case that $F$ acts trivially on $C_{\bL(\lambda)}(u)$. However, as $C_{\bL(\lambda)}(u) \cong \mathfrak{S}_3$ we have $F$ acts as an inner automorphism, say ${}^aF(b) = b$ for some $a \in C_{\bL(\lambda)}(u)$ and all $b \in C_{\bL(\lambda)}(u)$. Now take $g \in \bG$ to be such that $g^{-1}F(g) = a$ then replacing $(u,\lambda)$ by $({}^gu,{}^g\lambda)$ we have $F$ acts trivially on $C_{\bL(\lambda)}(u)$.

Let $A = C_{\bL(\lambda)}(u)$. As $p$ is a good prime for $\bG$, an exceptional group, we have $p \neq 2,3$ so \cref{en:K1} holds. If $a \in A$ then either $C_{A}(a)$ is cyclic or $C_{A}(a) = A$. If $C_{A}(a) = A$ then $a$ must be the identity so $a \in C_{\bL(\lambda)}^{\circ}(C_{A}(a))$. Appealing to \cref{rem:cyclic-centraliser} in the cyclic case we see that \cref{en:K2} holds.

\subsection{Families with $\bar{A}_{\mathscr{F}} \cong \mathfrak{S}_3$ and $\mathcal{O}_{\mathscr{F}}$ of type $\D_4(a_1)$}\label{subsec:D4(a1)}

After \cref{tab:families-with-A-nonabelian} we see that $\bG$ is of type $\E_n$ with $n \in \{6,7,8\}$. Let us note that regardless of the isogeny type of $\bG$ if $u \in \mathcal{U}(\bG)$ is in the class $\D_4(a_1)$ we have $A_{\bG}(u) \cong \mathfrak{S}_3$. Using the exact same argument as in \cref{subsec:S_3-distinguished} it suffices to show that there exists an element $u \in \mathcal{O}_{\mathscr{F}}^F$, a cocharacter $\lambda \in \mathcal{D}_u(\bG)^F$, and an $F$-stable complement $A \leqslant C_{\bL(\lambda)}(u)$ of $C_{\bL(\lambda)}^{\circ}(u)$.

We first consider the case where $\bG$ is simple of type $\E_6$. Let us adopt the notation of \cref{exmp:qs-elts-3D4}. In particular, we have $\sigma = n_w \in N_{\bG}(\bT_0)$ acts on $\Pi$ as the cycle $(\alpha_2,\alpha_5,\alpha_3)$ and $t = \widecheck{\omega}_4(\frac{1}{3}) \in C_{\bT_0}(\sigma)$ is an element of order $3$ with $C_{\bM}^{\circ}(t\sigma)$ of type $\A_2$. The group $C_{\bM}^{\circ}(t)$ is reductive of type $3\A_1$ with a basis of simple roots being given by $J = \{2,3,5\}$. Let $w' = w_{\Delta}w_J$ and let $a = \widecheck{\alpha}_2(\frac{1}{2})\widecheck{\alpha}_3(\frac{1}{2})\widecheck{\alpha}_4(\frac{1}{2})\widecheck{\alpha}_5(\frac{1}{2})n_{w'}$.

The element $a$ has order $2$ and satisfies ${}^at = t^{-1}$ and ${}^a\sigma = \sigma^{-1}$. In particular, as $t \in C_{\bT_0}(\sigma)$ we have ${}^a(t\sigma) = (t\sigma)^{-1}$ so $a$ normalises $\bK := C_{\bM}^{\circ}(t\sigma)$ and $A = \langle t\sigma, a\rangle$ is isomorphic to $\mathfrak{S}_3$. The element $a$ preserves the roots $\Sigma_{t\sigma}$ and acts as the permutation $(\beta_3,-\beta_4)(-\beta_3,\beta_4)$. Hence, $a$ preserves the simple system $\{\beta_3,-\beta_4\}$ of $\Sigma_{t\sigma}$ and $a$ is a quasi-central element of the coset $\bK a$. Let $u \in \mathcal{U}(C_{\bK}^{\circ}(a))$ be a regular unipotent element and $\lambda \in \mathcal{D}_u(C_{\bK}^{\circ}(a),C_{\bT_0}^{\circ}(a,t\sigma))$ a Dynkin cocharacter. Then $u$ is regular in $\bK$ and subregular in $\bM$ by \cref{exmp:fusion-3D4}. Hence $u$ is contained in the class $\D_4(a_1)$ of $\bG$.

As $u \in \bM$ we have $Z(\bM) \leqslant C_{\bG}(u)$ and certainly $Z(\bM) \leqslant \bL(\lambda)$ because $\lambda \in \widecheck{X}(\bT_0)$. We know that $C_{\bL(\lambda)}^{\circ}(u)$ is a $2$-dimensional torus, see \cite[Table 22.1.3]{LiSe}, so $C_{\bL(\lambda)}^{\circ}(u) = Z^{\circ}(\bM)$. Certainly no element of $A$ centralises $\bM$ so $A \cap C_{\bL(\lambda)}^{\circ}(u) = \{1\}$ and $A$ is a complement of $C_{\bL(\lambda)}^{\circ}(u)$. That $A$ is $F$-stable is clear.

Now assume $\bG$ is a simple group of type $\E_7$ or $\E_8$ and let $\bM \leqslant \bG$ be the standard Levi subgroup of type $\E_6$. We have already shown that there exists an element $u \in \mathcal{U}(\bM)^F$ in the class $\D_4(a_1)$, a cocharacter $\lambda \in \mathcal{D}_u(\bM)^F \subseteq \mathcal{D}_u(\bG)^F$, and an $F$-stable complement $A \leqslant C_{\bL_{\bM}(\lambda)}(u)$ of $C_{\bL_{\bM}(\lambda)}^{\circ}(u)$. By \cref{lem:inj-comp-grp} we have $A \leqslant C_{\bL_{\bG}(\lambda)}(u)$ must also be an $F$-stable complement of $C_{\bL_{\bG}(\lambda)}^{\circ}(u)$ so we are done.

\subsection{The family with $\bar{A}_{\mathscr{F}} \cong \mathfrak{S}_3$ and $\mathcal{O}_{\mathscr{F}}$ of type $\D_4(a_1){+}\A_1$}\label{subsec:D4(a1)A1}
In this case we have $\bG$ is of type $\E_8$. Let $\bM \leqslant \bG$ be the standard Levi subgroup of type $\D_4\A_1$ determined by the set of simple roots $\Pi = \{2,3,4,5 \mid 8\} \subseteq \Delta$. We have $N_W(\Pi)$ is a reflection group of type $\B_3$ with generators $\tilde{s}_1 = s_{\Pi,\alpha_6}$, $\tilde{s}_2 = s_{\Pi,\alpha_1}$, and $\tilde{s}_3 = s_{\Pi,\alpha_6+\alpha_7}$, where $s_{\Pi,\beta}$ denotes the product of longest elements $w_{\Pi\cup\{\beta\}}w_{\Pi}$. The element $w = \tilde{s}_1\tilde{s}_2$, of order $3$, acts on $\Pi$ as the cycle $(\alpha_2,\alpha_5,\alpha_3)$. Moreover, the element $\sigma = n_w$ is a quasi-central element of the coset $\bM \sigma$ of order $3$.

Let $t = \widecheck{\omega}_1(\frac{2}{3})\widecheck{\omega}_4(\frac{1}{3})\widecheck{\omega}_6(\frac{2}{3}) \in C_{\bT_0}(\sigma)$, which is an element of order $3$. If we let $\mathcal{O}_7 = \{\alpha_8\}$ and $\mathcal{O}_1,\dots,\mathcal{O}_6$ be as in \cref{exmp:qs-elts-3D4} then we have $\Sigma^+/\langle \sigma\rangle = \{\mathcal{O}_1,\dots,\mathcal{O}_7\}$, where $\Sigma^+$ are the positive roots of $\bM$ containing $\Pi$. Similarly to \cref{exmp:qs-elts-3D4} we have $C_{\bM}^{\circ}(t\sigma)$ is of type $\A_2\A_1$ with positive roots $\Sigma_{t\sigma} = \{\beta_1,\beta_3,\beta_4,\beta_7\}$.

The centraliser $C_{\bM}(t)$ in $\bM$ is the standard Levi subgroup of $\bG$ of type $4\A_1$ with basis of simple roots given by $J = \{2,3,5,8\} \subseteq \Pi$. We let $w' = w_{\Delta}w_J\tilde{s}_3\tilde{s}_2\tilde{s}_1\tilde{s}_3\tilde{s}_2\tilde{s}_3$ and $a = \widecheck{\omega}_2(\frac{1}{2})\widecheck{\omega}_6(\frac{1}{2})n_{w'}$. Then $a$ is an involution satisfying ${}^at = t^{-1}$ and ${}^a\sigma = \sigma^{-1}$. Hence ${}^a(t\sigma) = (t\sigma)^{-1}$ so $a$ normalises $\bK = C_{\bM}^{\circ}(t\sigma)$ and $A = \langle t\sigma,a\rangle$ is isomorphic to $\mathfrak{S}_3$. As in \cref{subsec:D4(a1)} the element $a$ acts on $\Sigma_{t\sigma}$ as the permutation $(\beta_3,-\beta_4)(-\beta_3,\beta_4)$. Hence, $a$ preserves the simple system $\{\beta_3,-\beta_4,\beta_7\}$ of $\Sigma_{t\sigma}$ and $a$ is a quasi-central element of the coset $\bK a$. If $u \in \mathcal{U}(C_{\bK}^{\circ}(a))$ is a regular unipotent element then $u$ is contained in the class $\D_4(a_1){+}\A_1$ of $\bG$. We fix a cocharacter $\lambda \in \mathcal{D}_u(C_{\bK}^{\circ}(a),C_{\bT_0}^{\circ}(a,t\sigma)) \subseteq \mathcal{D}_u(\bG,\bT_0)$.

By design $A \leqslant C_{\bL(\lambda)}(u)$ is an $F$-stable subgroup of $C_{\bL(\lambda)}(u)$. Arguing as in \cref{subsec:S_3-distinguished} we see that it is sufficient to show that $A \cap C_{\bL(\lambda)}^{\circ}(u)$ is trivial. The centraliser $C_W(\Pi) = \langle \tilde{s}_3, \tilde{s}_3^{\tilde{s}_2}, \tilde{s}_3^{\tilde{s}_2\tilde{s}_1} \rangle \leqslant N_W(\Pi)$ is a reflection group of type $3\A_1$. One easily checks that conjugation gives a faithful permutation action of the group $\langle w,w'\rangle \leqslant W$ on the non-identity elements of $C_W(\Pi)$.

The group $C_W(\Pi)$ is the Weyl group of a subsystem subgroup $\bC \leqslant C_{\bG}(\bM)$ of type $3\A_1$, its roots being those that are orthogonal to all of those in $\Sigma$. As $u \in \bM$ and $\lambda \in \mathcal{D}_u(\bM,\bT_0)$, and as $\bC$ is connected, we must have $\bC \leqslant C_{\bL(\lambda)}^{\circ}(u)$. Moreover, it is well known that equality holds in this case, see \cite[Table 22.1.1]{LiSe}. By the above remarks it follows immediately that $A \cap \bC = \{1\}$ as desired.

\subsection{The family with $\bar{A}_{\mathscr{F}} \cong \mathfrak{S}_4$}\label{subsec:S4}
In this case we have $\bG$ is of type $\F_4$ and $\mathcal{O}_{\mathscr{F}}$ is the class $\F_4(a_3)$. Let $u \in \mathcal{O}_{\mathscr{F}}^F$ and $\lambda \in \mathcal{D}_u(\bG)^F$ be an associated cocharacter. The class $\F_4(a_3)$ is distinguished so we can use the argument of \cref{subsec:S_3-distinguished} to show that, after possibly replacing $(u,\lambda)$ by $({}^gu,{}^g\lambda)$ for some $g \in \bG$, we have $(A,\lambda)$ is an admissible covering for $u$ where $A = C_{\bL(\lambda)}(u)$.

The problem here is showing that \cref{en:K2} holds. As $A = C_{\bL(\lambda)}(u) \cong \mathfrak{S}_4$ we have for any $a \in A$ that $C_{A}(a)$ is either: abelian of order $3$ or $4$, a dihedral group of order $8$, or $A$ itself. If $C_{A}(a) = A$ then $a \in Z(A) = \{1\}$ so certainly $a \in C_{\bL(\lambda)}^{\circ}(C_{A}(a))$. Now, in the abelian case $C_{A}(a)$ is generated by at most two elements. As $\bG$ is simply connected it follows from \cite[Thm.~II.5.8]{SpSt} that $C_{A}(a)$ is contained in a maximal torus of $\bG$ so we can apply \cref{rem:in-max-torus}.

With this it suffices to prove that the property in \cref{en:K2} holds when $C_{A}(a)$ is a dihedral group. All elements of $A$ with $C_{A}(a)$ a dihedral group are conjugate, hence we need only show it for one such element. Note also that as \cref{en:K2} makes no reference to the Frobenius we need not concern ourselves with any questions of $F$-stability.

Consider the isolated element $s = \widecheck{\omega}_3(\frac{1}{4})$ of order $4$. If $\alpha_0$ is the root $1242$ then the set $\Pi = \{1,2,0\mid 4\}$ is a basis of simple roots for the centraliser $\bM = C_{\bG}^{\circ}(s)$, which is a semisimple group of type $\A_3\A_1$. Let $w = w_{\Delta}w_{\Pi}$. The element $\sigma = n_w$ is an involution satisfying ${}^{\sigma}\Pi = \Pi$. The set $\Pi/\langle \sigma\rangle$ is a union of three orbits $(\mathcal{O}_1,\mathcal{O}_2,\mathcal{O}_3) = (\{0,1\},\{2\},\{4\})$. One readily checks that $\sigma$ is a quasi-central element of the coset $\bM\sigma$ and the group $C_{\bM}^{\circ}(\sigma)$ is of type $\C_2\A_1$. Here $\{\beta_{\mathcal{O}_1},\beta_{\mathcal{O}_2}\}$ is a simple system of roots for the $\C_2$ component with $\beta_{\mathcal{O}_1}$ the short root.

Pick a regular unipotent element $u \in C_{\bM}^{\circ}(\sigma)$ and a cocharacter $\lambda \in \mathcal{D}_u(C_{\bM}^{\circ}(\sigma),C_{\bT_0}^{\circ}(\sigma))$. One can check that $u$ is regular in $\bM$ and in the class $\F_4(a_3)$ of $\bG$. As ${}^{\sigma}s = s^{-1}$ we have $H = \langle s,\sigma\rangle$ is a dihedral group of order $8$ with centre $Z(H) = \langle s^2\rangle$. Let $A = C_{\bL(\lambda)}(u)$ then clearly $H \leqslant A$ and $H \leqslant C_{A}(s^2)$. As $s^2$ is not the identity we must have $H = C_{A}(s^2)$. The torus $\bS = \{\widecheck{\omega}_1(-2v)\widecheck{\omega}_3(v) \mid v\in \mathbb{Q}\} \leqslant \bT_0 \leqslant \bL(\lambda)$ is centralised by $s$ and $\sigma$ and contains $s^2$. Hence, we have $s^2 \in \bS \leqslant C_{\bL(\lambda)}^{\circ}(C_{A}(s^2))$.

\begin{rem}
This last fact can be checked by hand and also by \Chevie{} using the command \texttt{MatYPerm}. This gives the matrix representation of the element $w$ acting on the cocharacter lattice $\widecheck{X}$.
\end{rem}

\subsection{The family with $\bar{A}_{\mathscr{F}} \cong \mathfrak{S}_5$}
In this case we have $\bG$ is of type $\E_8$ and $\mathcal{O}_{\mathscr{F}}$ is the class $\E_8(a_7)$. We argue exactly as in \cref{subsec:S4}. As $A \cong \mathfrak{S}_5$ we have for any $a \in A$ that $C_{A}(a)$ is either: abelian of order $4$, $5$, or $6$, dihedral of order $8$ or $12$, or $A$ itself. We need only show that \cref{en:K2} holds when $C_{A}(a)$ is dihedral.

{\bfseries Case $C_{A}(a) \cong \mathrm{Dih}_8$}. Consider the isolated element $s = \widecheck{\omega}_6(\frac{1}{4}) \in \bT_0$ of order $4$. Let $\alpha_0$ be the root $23465421$ then the set $\Pi = \{1,3,4,2,5 \mid 8,7,0\}$ is a basis of simple roots for the centraliser $\bM := C_{\bG}(t) = C_{\bG}^{\circ}(t)$, which is a semisimple group of type $\D_5\A_3$. Let $w = w_{\Delta}w_{\Pi}$ then ${}^w\Pi = \Pi$ and $\Pi/\langle w\rangle$ is a union of the following six orbits
\begin{equation*}
(\mathcal{O}_1,\mathcal{O}_2,\mathcal{O}_3,\mathcal{O}_4,\mathcal{O}_5,\mathcal{O}_6) = (\{1\},\{3\},\{4\},\{2,5\},\{0,8\},\{7\}).
\end{equation*}

The element $\sigma = n_w \in N_{\bG}(\bT_0)$ is a quasi-central element of the coset $\bM\sigma$ such that $\sigma^2 = \widecheck{\omega}_6(\frac{1}{2})$. The group $C_{\bM}^{\circ}(\sigma)$ is a semisimple group of type $\B_4\C_2$. The element $t = \widecheck{\omega}_4(\frac{1}{2})\widecheck{\omega}_6(\frac{1}{4}) \in \bT_0$ is centralised by $\sigma$ so the product $t\sigma$ is an involution. Moreover, the group $C_{\bM}^{\circ}(t\sigma)$ is a semisimple group of type $\B_1\B_3\C_2$. Let $\mathcal{O}_0 = \{10100000\}$ then a simple system for $\Sigma_{t\sigma}$ is given by $\{\beta_{\mathcal{O}_4} \mid \beta_{\mathcal{O}_1},\beta_{\mathcal{O}_2},\beta_{\mathcal{O}_0} \mid \beta_{\mathcal{O}_5},\beta_{\mathcal{O}_6}\}$.

Pick a regular unipotent element $u \in C_{\bM}^{\circ}(t\sigma)$. One can check that $u$ is in the class $\D_5(a_1){+}\A_3$ of $\bM$ and the class $\E_8(a_7)$ of $\bG$. Moreover, let us fix a cocharacter $\lambda \in \mathcal{D}_u(C_{\bM}^{\circ}(t\sigma), C_{\bT_0}^{\circ}(t\sigma))$. Note that ${}^{t\sigma}s = {}^{\sigma}s = s^{-1}$ so the group $H = \langle s, t\sigma\rangle$ is a dihedral group of order $8$ with centre $Z(H) = \langle s^2\rangle$. Let $A = C_{\bL(\lambda)}(u)$ then clearly $H \leqslant A$ and $H \leqslant C_{A}(s^2)$. From the list of the possible centralisers in $A$ we see that $H = C_{A}(s^2)$ as $s^2 \neq 1$. The torus $\bS = \{\widecheck{\omega}_6(v)\widecheck{\omega}_7(-2v) \mid v \in \mathbb{Q}\} \leqslant \bT_0 \leqslant \bL(\lambda)$ is centralised by $s$ and $t\sigma$ and contains $s^2$. Hence, we have $s^2 \in \bS \leqslant C_{\bL(\lambda)}^{\circ}(C_{A}(s^2))$.

{\bfseries Case $C_{A}(a) \cong \mathrm{Dih}_{12}$}. Now consider the element $s = \widecheck{\omega}_7(\frac{1}{3})$ of order $3$. Let now $\alpha_0$ be the root $23465431$ then the set $\Pi = \{1,2,3,4,5,6 \mid 8,0\}$ is a basis of simple roots for the centraliser $\bM := C_{\bG}(s) = C_{\bG}^{\circ}(s)$, which is a semisimple group of type $\E_6\A_2$. Let $w = w_{\Delta}w_{\Pi}$ then ${}^w\Pi = \Pi$ and $\langle w\rangle$ acts on $\Pi$ with the orbits
\begin{equation*}
(\mathcal{O}_1,\mathcal{O}_2,\mathcal{O}_3,\mathcal{O}_4,\mathcal{O}_5) = (\{1,6\},\{3,5\},\{4\},\{2\},\{0,8\}).
\end{equation*}

The element $\sigma = \widecheck{\omega}_3(\frac{1}{2})\widecheck{\omega}_5(\frac{1}{2})n_w$ is a quasi-central involution of the coset $\bM\sigma$. Note that the orbit $\mathcal{O}_5$ is special and $C_{\sigma,\mathcal{O}_5} = 1$. The group $C_{\bM}^{\circ}(\sigma)$ is a semisimple group of type $\F_4\A_1$. The element $t = \widecheck{\omega}_2(\frac{1}{2})\widecheck{\omega}_7(\frac{1}{2})$ is an involution centralised by $\sigma$ so $t\sigma$ is also an involution. The group $C_{\bM}^{\circ}(t\sigma)$ is of type $\C_4\A_1$. Let $\mathcal{O}_0 = \{01110000,01011000\}$ then we have $\{\beta_{\mathcal{O}_0},\beta_{\mathcal{O}_1},\beta_{\mathcal{O}_2},\beta_{\mathcal{O}_3} \mid \beta_{\mathcal{O}_5}\}$ is a simple system for $\Sigma_{t\sigma}$.

Pick $u \in C_{\bM}^{\circ}(t\sigma)$ in the class $\C_4(a_1){+}\A_1$, i.e., $u$ is a product of a subregular element of $\C_4$ and a regular element of $\A_1$, and let $\lambda \in \mathcal{D}_u(C_{\bM}^{\circ}(t\sigma),C_{\bT_0}^{\circ}(t\sigma))$. Using the above one can check that $u$ is in the class $\E_6(a_3){+}\A_2$ of $\bM$ and the class $\E_8(a_7)$ of $\bG$. We have ${}^{\sigma}s = s^{-1}$ so $H = \langle t,s,\sigma\rangle \cong \mathrm{Dih}_{12}$ with $Z(H) = \langle t\rangle$. Let $A = C_{\bL(\lambda)}(u)$ then clearly $H \leqslant A$ and $H \leqslant C_{A}(t)$. From the list of possible centralisers in $A$ we must have $H = C_{A}(t)$ as $t \neq 1$. The torus $\bS = \{\widecheck{\omega}_2(v)\widecheck{\omega}_7(-v) \mid v \in \mathbb{Q}\} \leqslant \bT_0 \leqslant \bL(\lambda)$ is centralised by $\sigma$ and contains $t$. Hence, we have $t \in \bS \leqslant C_{\bL(\lambda)}^{\circ}(C_{A}(t))$.

The proof of \cref{thm:exceptional-grps} is now complete.

\section{Proof of \cref{thm:mainA,thm:mainB}}\label{subsec:main-classical}

\begin{assumption}
Let $\ell$ be a prime number different from $p = \Char(K)$. We fix an $\ell$-modular system $(\mathbb{K},\mathbb{O},\mathbb{F})$ which is big enough for all the finite groups encountered.
\end{assumption}

\begin{proof}[Proof of \cref{thm:mainB}]
Recall that $\bG$ is an adjoint simple group and $\mathscr{F} \in \Fam(W)^F$ is an $F$-stable family. It follows from \cref{lem:coveringforGad,prop:admiss-split-symp,prop:admissible-splitting-ortho,thm:exceptional-grps} that, with the assumptions listed in \cref{thm:mainB}, there exists a unipotent element $u \in \mathcal{O}_{\mathscr{F}}^F$ and an admissible covering $A \leqslant C_{\bG}(u)$ of the canonical quotient $\bar{A}_{\mathscr{F}}$ such that $|A|$ is divisible by bad primes only. Note that if $\bG$ is of type $\A$ then $\bar{A}_{\mathscr{F}}$ is trivial and we may take $A$ to be the trivial group.

In each case we see that either $A$ is abelian or $A \cong \bar{A}_{\mathscr{F}}$. If $\bG$ is of type $\A$ then for each unipotent character $\rho \in \Uch(\mathscr{F}^*)$ we set $K_{\rho} := \Gamma_u$, the GGGC, with $u \in \mathcal{O}_{\mathscr{F}}^F$. Now assume $\bG$ is not of type $\A$ then as $p$ is good for $\bG$ we have $p \neq 2$. By \cref{thm:kawanakaconj} we may choose, for each unipotent character $\rho \in \Uch(\mathscr{F}^*)$, a Kawanaka character $K_{\rho} := K_{[a,\psi]}$ such that $\rho$ is the projection of $K_{\rho}$ onto the subspace of $\Class(G)$ spanned by $\Uch(\mathscr{F}^*)$.

Now assume $\rho' \in \Uch(\mathscr{F}'^*)$ is another unipotent character so that $\langle K_{\rho},\rho'\rangle \neq 0$. By \cref{eq:GGGCs-sum-Kaw} there must exist a unipotent element $u \in \mathcal{O}_{\mathscr{F}}^F$ such that $\langle \Gamma_u,\rho'\rangle \neq 0$ because both $K_{\rho}$ and $\Gamma_u$ are characters. As $\mathcal{O}_{\rho'}^* = \mathcal{O}_{\mathscr{F}'}$ is the wave-front set of $\rho'$ it can be deduced from a (dual) result of Geck--Malle \cite[Prop.~4.3]{GeMa00} that $\mathcal{O}_{\mathscr{F}} \preceq \mathcal{O}_{\mathscr{F}'}$, see also \cite[Thm.~9.1]{AcAu07} and \cite[Prop.~15.2]{Tay16}. This completes the proof of \cref{thm:mainB} because $K_{\rho}$ contains only one character in $\Uch(\mathscr{F}^*)$. In addition, since $A$ is an $\ell'$-group, the character $K_\rho$ is the character of a projective $\mathbb{O}G$-module $P_\rho$.
\end{proof}

\begin{proof}[Proof of \cref{thm:mainA}]
We start by reducing to groups with trivial center. 
Let us consider a regular embedding $\iota : \bfG \hookrightarrow \widetilde{\bfG}$ and an adjoint quotient $\pi : \widetilde{\bfG} \twoheadrightarrow \bG_{\mathrm{ad}}$, both defined over $\mathbb{F}_q$. Write $\widetilde{\mathbf{Z}} := Z(\widetilde{\mathbf{G}}) = \mathrm{Ker}\, \pi$. By \cite[Prop.~13.20]{DiMibook} the induced maps on characters induce bijections
\begin{equation}\label{eq:bij-unichar}
 \iota^* : \Uch(\widetilde G) \simto \Uch(G) \quad and \quad \pi_* : \Uch(G_\mathrm{ad}) \simto \Uch(\widetilde G).
 \end{equation}
Here $\iota^*$ identifies with the restriction from $\widetilde G$ to $G$ and $\pi_*$ with the inflation from $G_\mathrm{ad}$ to $\widetilde G$ 
since $\widetilde{\mathbf{Z}}$ is connected. Note that as a consequence, $\widetilde G/G$ acts trivially on the set of unipotent characters of $G$. We claim that \cref{eq:bij-unichar} also holds for Brauer characters whenever $\ell$ is good and does not divide $q |Z(\bfG)_F|$.

We denote by $\mathrm{UBr}(G)$ the set of unipotent irreducible Brauer characters, by which we mean the Brauer characters of the simple $\mathbb{F}G$-modules lying in a unipotent block. By \cite[Thm.~B]{Ge93}, when $\ell$ is good for $G$, the restriction to $G$ of a Brauer character of $\widetilde G$ is multiplicity-free. If we assume in addition that $\ell$ does not divide $q|Z(\bfG)_F|$, then by \cite[Thm.~A]{Ge93} the unipotent characters of $G$ form a basic set for $\mathbb{Z} \mathrm{UBr}(G)$. Since $\widetilde G/G$ acts trivially on $\Uch(G)$ it must therefore act  trivially on $\mathrm{UBr}(G)$, which shows that unipotent Brauer characters of $\widetilde G$ remain irreducible under restriction to $G$. On the other hand, every unipotent character of $\widetilde G$ has $\widetilde Z$ in its kernel by \cref{eq:bij-unichar}. Again, since unipotent characters form a basic set, the same holds for unipotent Brauer characters, and we deduce that inflation induces a bijection $\pi_* : \mathrm{UBr}(G_\mathrm{ad}) \simto \mathrm{UBr}(\widetilde G)$.

For all the finite reductive groups considered, unipotent characters form a basic set for the sum of unipotent $\ell$-blocks so that we have a commutative diagram of isomorphisms
\begin{center}\begin{tikzcd}
\mathbb{Z} \Uch(G) \ar[d,"\mathrm{dec}_G"] 
&  \mathbb{Z} \Uch(\widetilde G) \ar[l,"\iota^*"']  \ar[r,"\pi_*"]\ar[d,"\mathrm{dec}_{\widetilde G}"]
&  \mathbb{Z} \Uch(G_\mathrm{ad}) \ar[d,"\mathrm{dec}_{G_{\mathrm{ad}}}"]\\
 \mathbb{Z}  \mathrm{UBr}(G) 
 &  \mathbb{Z}  \mathrm{UBr}(\widetilde G)  \ar[l,"\iota^*"']  \ar[r,"\pi_*"]
 &  \mathbb{Z}  \mathrm{UBr}(G_\mathrm{ad})
\end{tikzcd}
\end{center}
where the horizontal arrows preserve the bases, and the vertical maps are the $\ell$-decomposition maps. This shows that under the identification of unipotent characters of $G$ and $G_{\mathrm{ad}}$, their decomposition numbers coincide. Note that this identification preserves the unipotent support. Now, using standard reduction techniques, see for example \cite[\S8.8]{Lu84}, we can write $G_{\mathrm{ad}}$ as a product $\prod \bG_i^{F_i}$ where each $\bG_i$ is adjoint simple with Frobenius $F_i$. We conclude that it is enough to prove that \cref{thm:mainA} holds when $\bG$ is adjoint simple.

It follows from \cref{thm:mainB} that for every $F$-stable family $\mathscr{F} \in \Fam(W)^F$ and each unipotent character $\rho \in \Uch(\mathscr{F})$ there exists a projective $\mathbb{O}G$-module $P_{\rho}$ such that the following hold:
\begin{enumerate}[leftmargin=1.25cm,label=(\alph*)]
	\item $\langle [P_\rho], \rho'\rangle_G = \delta_{\rho,\rho'}$ for every $\rho' \in \Uch(\mathscr{F})$,
	\item if $\langle [P_\rho], \rho'\rangle_G \neq 0$ and $\rho' \in \Uch(\mathscr{F}')$ then $\mathcal{O}_{\mathscr{F}'} \preceq \mathcal{O}_{\mathscr{F}}$ (or equivalently for the dual families $\mathcal{O}_{\mathscr{F}^*} \preceq \mathcal{O}_{\mathscr{F}'^*}$).
\end{enumerate}

Consequently this means that for each unipotent character $\rho$ there is a unique indecomposable direct summand $Q_\rho$ of $P_\rho$ such that $\langle [Q_\rho], \rho\rangle \neq 0$. Since the unipotent characters form a basic set of characters for the unipotent blocks, we deduce that the modules $\{Q_\rho\}_{\rho \in \Uch(G)}$ give a complete set of representatives for the PIMs of the unipotent blocks and \cref{thm:mainA} follows.
\end{proof}

\textsc{O.~Brunat, Universit\'e de Paris, Sorbonne Universit\'e, CNRS, Institut de Math\'ematiques de Jussieu-Paris Rive Gauche, IMJ-PRG, F-75013, Paris, France.}\par\nopagebreak
\textit{E-mail address}: \texttt{olivier.brunat@imj-prg.fr}

\medskip

\textsc{O.~Dudas, Universit\'e de Paris, Sorbonne Universit\'e, CNRS, Institut de Math\'ematiques de Jussieu-Paris Rive Gauche, IMJ-PRG, F-75013, Paris, France.}\par\nopagebreak
\textit{E-mail address}: \texttt{olivier.dudas@imj-prg.fr}

\medskip

\textsc{J.~Taylor, University of Southern California, Department of Mathematics, 3620 South Vermont Avenue, Los Angeles, CA 90089, United States.}\par\nopagebreak
\textit{E-mail address}: \texttt{jayt@usc.edu}

\end{document}